\documentclass{amsart}
\usepackage[utf8]{inputenc}
\usepackage{enumerate}
\usepackage{xcolor}
\usepackage{imakeidx}
\usepackage{amsfonts}
\usepackage{amssymb}
\usepackage{amsmath}
\usepackage{amsthm}
\usepackage{tikz-cd}
\usepackage{float}
\usepackage{graphicx}
\usepackage{tikz}
\usepackage{thmtools}
\usepackage{thm-restate}
\usepackage{mathrsfs}
\usepackage{enumitem}
\usepackage{mathtools}
\usepackage{fancyhdr,textcase}
\usepackage{isomath}
\usepackage{MnSymbol}
\usepackage{yfonts}
\usepackage{wrapfig}
\usepackage{verbatim}
\usepackage[all]{xy}
  \SelectTips{cm}{10}     % Use the arrowheads that agree with inline arrows.
  \everyxy={<2.5em,0em>:} % Set the unit length in xy diagrams.
\newcommand{\cok}{\mathrm{cok}\,}
\usepackage{hyperref}
\usepackage[backend=biber, style=alphabetic, sorting=nyt]{biblatex}
\theoremstyle{definition}
\newtheorem{theorem}{Theorem}[subsection]
\newtheorem{definition}[theorem]{Definition}
\newtheorem{notation}[theorem]{Notation}
\newtheorem*{notation*}{Notation}

\newtheorem{example}[theorem]{Example}
\newtheorem{construction}[theorem]{Construction}
\newtheorem{remark}[theorem]{Remark}
\newtheorem{corollary}[theorem]{Corollary}
\newtheorem{proposition}[theorem]{Proposition}

\newtheorem{lemma}[theorem]{Lemma}
\newtheorem{question}[theorem]{Question}
\newtheorem{claim}[theorem]{Claim}
\newtheorem{warning}[theorem]{Warning}

\newtheorem*{question*}{Question I}
\newtheorem*{question'}{Question II}
\newtheorem*{question''}{Question III}

\newtheorem*{theorem*}{Theorem}
\newtheorem*{theorem_A}{Theorem A}
\newtheorem*{theorem_B}{Theorem B}
\newtheorem*{theorem_C}{Theorem C}
\newtheorem*{theorem_D}{Theorem D}
\newtheorem*{theorem_E}{Theorem E}

\newcommand{\ZZ}{{\mathbb Z}}
\newcommand{\NN}{{\mathbb N}}
\newcommand{\RR}{{\mathbb R}}
\newcommand{\QQ}{{\mathbb Q}}
\newcommand{\KK}{{\mathbb K}}
\newcommand{\CC}{{\mathbb C}}

\newcommand{\PP}{{\mathbb P}}
\newcommand{\Affine}{{\mathbb A}}

\newcommand{\cZ}{{\mathcal Z}}
\newcommand{\cY}{{\mathcal Y}}
\newcommand{\cX}{{\mathcal X}}
\newcommand{\cW}{{\mathcal W}}
\newcommand{\cV}{{\mathcal V}}
\newcommand{\cU}{{\mathcal U}}

\newcommand{\cO}{{\mathcal O}}

\newcommand{\cM}{{\mathcal M}}
\newcommand{\cL}{{\mathcal L}}
\newcommand{\cK}{{\mathcal K}}

\newcommand{\cH}{{\mathcal H}}

\newcommand{\cF}{{\mathcal F}}
\newcommand{\cE}{{\mathcal E}}
\newcommand{\cD}{{\mathcal D}}
\newcommand{\cC}{{\mathcal C}}
\newcommand{\cB}{{\mathcal B}}
\newcommand{\cA}{{\mathcal A}}
\newcommand{\cJ}{{\mathcal J}}

\DeclareMathOperator{\Spec}{Spec}

\DeclareMathOperator{\supp}{supp}
\DeclareMathOperator{\vol}{vol}

\DeclareMathOperator{\GVF}{GVF}

\DeclareMathOperator{\Val}{Val}

\DeclareMathOperator{\amp}{amp}
\DeclareMathOperator{\sdiv}{div} % Weil divisor of f
\DeclareMathOperator{\cdiv}{div} % Cartier divisor of f
\DeclareMathOperator{\Idiv}{IDiv}

\DeclareMathOperator{\Adiv}{ADiv}

\DeclareMathOperator{\Apic}{APic}
\DeclareMathOperator{\Rdiv}{RDiv}
\DeclareMathOperator{\Ldiv}{LDiv}

\DeclareMathOperator{\Div}{Div} % Cartier divisors
\DeclareMathOperator{\ord}{ord} % order of vanishing in a local ring or a rational function
\DeclareMathOperator{\height}{ht}
\DeclareMathOperator{\interior}{o}
\DeclareMathOperator{\Eff}{Eff}

\DeclareMathOperator{\Bigcone}{Big}

\DeclareMathOperator{\PSH}{PSH}
\DeclareMathOperator{\Z}{Z}
\DeclareMathOperator{\VF}{VF}
\DeclareMathOperator{\Ima}{Im}

\DeclareSymbolFont{yhlargesymbols}{OMX}{yhex}{m}{n} \DeclareMathAccent{\yhwidehat}{\mathord}{yhlargesymbols}{"62}

\newcommand{\smallk}{k}

\newcommand{\Field}{F}

\newcommand{\adiv}{\widehat{\cdiv}}
\newcommand{\avol}{\widehat{\vol}}
\newcommand{\adeg}{\widehat{\deg}}

\newcommand{\aH}{\widehat{H}}

\newcommand{\aN}{\widehat{N}}

\newcommand{\GVFpairing}{\beta}
\newcommand{\GVFk}{\ov{\smallk(t)}}
\newcommand{\GVFQ}{\ov{\QQ}}

\newcommand{\ov}{\overline}

\title[Existential closedness of $\GVFQ$ as a GVF via Arakelov geometry]{Existential closedness of $\GVFQ$ as a globally valued field via Arakelov geometry}
\author{Michał Szachniewicz}
\date{}
\address{Mathematical Institute, University of Oxford, Oxford, OX2 6GG, UK}
\email{michal.szachniewicz@maths.ox.ac.uk}

\begin{document}

\maketitle

\begin{abstract}
    We use the differentiability of the arithmetic volume function and an arithmetic Bertini type theorem to classify when one can find a closed point on the generic fiber of an arithmetic variety, whose heights with respect to some finite tuple of arithmetic $\RR$-divisors approximate a given tuple of real numbers. 
    
    We use this result to prove existential closedness of $\GVFQ$ as a globally valued field (abbreviated GVF). We introduce GVF functionals on the space of arithmetic $\RR$-divisors and interpret the essential infimum function as the infimum of values of normalised GVF functionals, at least when the generic part of the arithmetic $\RR$-divisor is big. We also give a new criterion on equality in one of the Zhang's inequalities.
\end{abstract}

\tableofcontents

\section{Introduction}

\subsection{Main results}

In this paper we study polynomial equations with height conditions. Recall that for a number $q \in \QQ^{\times}$ its height is $\height(q):=\log(\max(|a|, |b|))$, if $q=a/b$ for some coprime integers $a, b$. It measures the complexity of a number. This height extends to the absolute logarithmic Weil height on $\GVFQ$, and on tuples $\ov{a} \in \GVFQ^m$ we define it to be the maximum of heights of coordinates. Note that it is not invariant under (even linear) changes of coordinates. In general, if $X$ is a variety over $\GVFQ$ and $f, g:X \to \Affine_{\GVFQ}^m$ are morphisms, we get an induced height for $x \in X(\GVFQ)$ given by $\height(f(x)) - \height(g(x))$. If $h_1, \dots, h_n$ are functions on $X(\GVFQ)$ obtained in this way, one could ask the following question.
\begin{question*}~\label{question}
    Given real numbers $r_1, \dots, r_n$, does there exist $x \in X(\GVFQ)$ with $h_1(x)=r_1, \dots, h_n(x)=r_n$?
\end{question*}
This question is important when one wants to compare various heights, for example when on an abelian variety one looks at the naive heights and the N\'eron–Tate height, see for example inequalities in \cite[Part B]{Silverman_Hindry_diophantine}.

Without conditions on heights, the (weak) Nullstellensatz provides the answer. More precisely, $X$ has a $\GVFQ$-point if and only if $X$ has an $\Field$-point, for some field extension $\GVFQ \subset \Field$, which is equivalent to non-emptiness of $X$.

Of course, even for a non-empty $X$, the answer cannot be positive for any tuple of real numbers, essentially for two reasons. First, the height function is not onto $\RR$. Second, there may be some relations between heights $h_1, \dots, h_n$, for example $h_1+h_2=h_3$ or $h_1 \geq 0$. To fix the first problem we ask instead whether there is a point approximating values $r_1, \dots, r_n$ up to $\varepsilon$, for some $\varepsilon > 0$. To resolve the other problem, we need to make the definition of height robust enough, so that the necessary conditions on tuples $r_1, \dots, r_n$ are simple to state.

The Weil height machine (see \cite[Part B]{Silverman_Hindry_diophantine}) is an example of a robust formalism for heights on varieties over $\GVFQ$. For a projective variety $X$ over $\GVFQ$ with a line bundle $L$, the Weil height machine associates a function $h_L:X(\GVFQ) \to \RR$ which is well defined only up to a bounded term. If $L$ is very ample, then $h_L$ has a representative of the form described above, on the complement of the hypersurface given by the zeroes of a section of $L$. However, to make sense of our motivating question we need to have heights defined not up to bounded terms. Fortunately, one can equip $L$ with an extra datum which will give a representative of $h_L$. For simplicity assume that $X$ is over $\QQ$ and choose a projective integral scheme $\cX$ over $\ZZ$ with a line bundle $\cL$, such that $\cX_{\QQ} = X, \cL_{\QQ} = L$. Such a scheme is called an \textit{arithmetic variety}. Assume additionally that $X$ is smooth over $\QQ$ and that $h$ is a conjugation invariant hermitian metric on the complex analytification of $L$. A pair $\ov{\cL} = (\cL, h)$ is called a \textit{hermitian line bundle} on $\cX$ and it defines a height function $h_{\ov{\cL}}:X(\GVFQ) \to \RR$. If $x$ is a closed point in $X$, then $h_{\ov{\cL}}(x)$ is the normalised arithmetic degree of the line bundle $\ov{\cL}$ restricted to the curve defined by taking the closure of $\{x\}$ in $\cX$. For the details of this approach, see \cite{Chambert_Loir_survey}. We choose to work with slightly more general objects parametrising heights, namely \textit{arithmetic $\RR$-divisors of $C^0$-type} introduced by Moriwaki in \cite{Moriwaki_Zariski_arithmetic_decompositions}. These are pairs $\ov{\cD} = (\cD, g)$, where $\cD$ is an $\RR$-Cartier divisor on $\cX$ and $g$ is a $\cD$-Green function of $C^0$-type. A trivial example of such object is a \textit{principal} $\RR$-divisor of $C^0$-type, i.e., a real combination of $\adiv(f) = (\cdiv(f), -\log|f|^2)$ for $f$'s in the function field of $X$. For a closed point $x \in X$, one defines $h_{\ov{\cD}}(x)$ as the normalised (arithmetic) intersection of $\ov{\cD}$ and the closure of $\{x\}$ in $\cX$, see Subsection~\ref{subsection_heights_w_r_t_arithmetic_R_divisors}. Now we can ask a more refined version of the question above.
\begin{question'}
    Given real numbers $r_1, \dots, r_n$, $\varepsilon>0$, and arithmetic $\RR$-divisors of $C^0$-type $\ov{\cD}_1, \dots, \ov{\cD}_n$ on $\cX$, does there exist $x \in X(\GVFQ)$ with 
    \[ |h_{\ov{\cD}_i}(x)-r_i|<\varepsilon \textnormal{ for all } i=1, \dots, n ?\]
\end{question'}
The arithmetic $\RR$-divisors of $C^0$-type on $\cX$ form a vector space over $\RR$ denoted $\Adiv_{\RR}(\cX)$. The main theorem of this paper gives the following answer.
\begin{theorem_A}[Theorem~\ref{theorem_main_arithmetic_approximation_theorem}]~\label{theorem_A}
    Let $\ov{\cD}_0, \dots, \ov{\cD}_n$ be arithmetic $\RR$-divisors of $C^0$-type on $\cX$ with $\ov{\cD}_0 = (\cdiv(2),0)$. Fix $\varepsilon > 0$, a closed proper subscheme $Z \subset X=\cX \otimes \QQ$, and a linear functional
    \[ l:\Adiv_{\RR}(\cX) \to \RR, \]
    which is non-negative on effective arithmetic $\RR$-divisors of $C^0$-type, zero on principal arithmetic $\RR$-divisors of $C^0$-type, and sends $\ov{\cD}_0$ to $\log(2)$. Then there exists a closed point $x \in X \setminus Z$ such that for all $i=0, \dots, n$
    \[ |h_{\ov{\cD}_i}(x) - l(\ov{\cD}_i)| < \varepsilon. \]
\end{theorem_A}
The theorem still holds if $l$ is defined only on the real vector subspace of $\Adiv_{\RR}(\cX)$ generated by $\ov{\cD}_0, \dots, \ov{\cD}_n$, provided that one of $\ov{\cD}_1, \dots, \ov{\cD}_n$ is big in the sense of Definition~\ref{definition_properties_of_arithmetic_R_divisors_big_etc}, see Corollary~\ref{corollary_main_arithmetic_approximation_corollary}. Note that the conditions on values of heights $h_{\ov{\cD}_0}(x), \dots, h_{\ov{\cD}_n}(x)$ in Theorem~\ref{theorem_A}A are optimal, if we expect to find a solution $x$ general enough. Indeed, always $h_{\ov{\cD}_0}(x) = \log(2)$; if $x \not\in \supp(\cD_i)$ and $\ov{\cD}_i$ is effective, then $h_{\ov{\cD}_i}(x) \geq 0$ for all $i=1, \dots, n$; and for any principal $\ov{\cE}$ we have $h_{\ov{\cE}}(x) = 0$, by the product formula.

% Fix the hyperlinks here!
One can consider another way of making sense out of Question~\ref{question}I. Note that for a number $q \in \QQ^{\times}$, the height $\height(q)$ can be expressed in a way that does not involve a presentation as a quotient of two integers:
\[ \height(q) = - \min(-\log|q|, 0) + \sum_{p \textnormal{ prime}} -\min(v_p(q), 0) \log p, \]
where $v_p$ is the $p$-adic valuation on $\QQ$ with $v_p(p)=1$. More generally, let us consider the measure $\mu$ on the space of (non-Archimedean and Archimedean) valuations $\Val_K$ for a number field $K$, defined by the following formula
\[ \mu := \frac{1}{[K : \QQ]} \biggl( \sum_{p \in \Spec(\cO_K)} \delta_{\ord_p} \cdot \log \# \kappa(p) + \sum_{\sigma : K \to \CC} \delta_{-\log |\sigma(-)|} \biggr). \]
The first sum is taken only over closed points, and the second sum is taken over all embeddings of fields. By $\kappa(p)$ we mean the residue field of the point $p$ and $\delta$ denotes the Dirac delta. Then for $q \in K^{\times}$ we have
\[ \height(q) = \int_{\Val_K} -\min(v(q),0) d\mu(v) \]
and
\[ \int_{\Val_K} v(q) d\mu(v) = 0. \]
If $t(x_1, \dots, x_n)$ is an expression composed out of $n$ free variables, addition, $\QQ$-scalar multiplication and taking minima (i.e., it is a $\QQ$-tropical polynomial, see Definition~\ref{definition_GVF_Globally_valued_fields}), for a tuple $a_1, \dots, a_n \in (K^{\times})^n$ we can define the ``$t$-height" of this tuple by
\[ R_t(a_1, \dots, a_n) := \int_{\Val_K} t(v(a_1), \dots, v(a_n)) d\mu(v). \]
It does not depend on the number field $K$, so we get a function $R_t: (\GVFQ^{\times})^n \to \RR$. Note that (representatives of) Weil heights are included among these functions.

In general, a field $\Field$ equipped with a collection of functions $R_t:(\Field^{\times})^n \to \RR$ for each $\QQ$-tropical polynomial $t$ is called a \textit{globally valued field} (abbreviated GVF) if it satisfies the axioms from Definition~\ref{definition_GVF_Globally_valued_fields}. One denotes
\[ R_t(\ov{a}) =: \int t(v(\ov{a})) dv \]
and the product formula axiom asserts that
\[ \int v(a) dv = 0 \textnormal{ for all } a \in \Field^{\times}. \]
These unbounded continuous logic structures were introduced by Ben Yaacov and Hrushovski in \cite{GVF2}. The notation with integrals is natural in the sense that every GVF structure on $\Field$ comes from a unique (up to a renormalisation) measure $\mu$ on the space $\Omega_{\Field}$ of pre-valuations on $\Field$, such that for all $\QQ$-tropical polynomials $t$ and all tuples $\ov{a}$ in $\Field^{\times}$
\[ \int t(v(\ov{a})) dv = \int_{\Omega_{\Field}} t(v(\ov{a})) d\mu(v). \]
For the details on this construction see \cite[Section 2]{GVF3}.
% maybe write something about the measure theoretic approach

One can reformulate Question~\ref{question}I in the language of globally valued fields in the following way.
\begin{question''}
    Given real numbers $r_1, \dots, r_n$, $\varepsilon>0$, polynomials $f_1, \dots, f_s, g \in \GVFQ[x_1, \dots, x_m]$ and $\QQ$-tropical polynomials (in $m$ variables) $t_1, \dots, t_n$, does there exist $\ov{a} = (a_1, \dots, a_m) \in (\GVFQ^{\times})^m$ satisfying $f_1(\ov{a})=\dots=f_s(\ov{a})=0, g(\ov{a}) \neq 0$ with 
    \[ |R_{t_i}(\ov{a}) - r_i|< \varepsilon \textnormal{ for all } i=1, \dots, n ?\]
\end{question''}
The answer to this question cannot be positive for all tuples $r_1, \dots, r_n$ as for example it may happen that $R_{t_1+t_2}=R_{t_1} + R_{t_2}$, or $R_{t_3} \geq 0$. Moreover, equations $f_1(\ov{a})=0, \dots, f_s(\ov{a})=0$ can add some constrains on the heights. To see for which tuples $r_1, \dots, r_n$ existence of solutions is possible/expected one can use a definition common in model theory. Namely, call equations
\[ f_1(\ov{a})=\dots=f_s(\ov{a})=0, g(\ov{a}) \neq 0, R_{t_1}(\ov{a}) = r_1, \dots, R_{t_n}(\ov{a}) = r_n \]
\textit{consistent}, if there is a GVF extension $\GVFQ \subset \Field$ and a tuple $\ov{a}$ in $\Field$ satisfying them. We prove that consistent equations can be approximately solved in $\GVFQ$, which in the language of model theory means (essentially by definition) that $\GVFQ$ is \textit{existentially closed} as a GVF, see Definition~\ref{definition_existentially_closed_GVFs}. Note that this can bee seen as a (weak) Nullstellensatz with height conditions. 
\begin{theorem_B}[Theorem~\ref{theorem_QQ_bar_is_existentially_closed}]~\label{theorem_B}
    $\GVFQ$ is an existentially closed GVF.
\end{theorem_B}
This theorem is an arithmetic analogue of \cite[Theorem 2.1]{GVF2} which states that the field $\ov{k(t)}$ equipped with a natural GVF structure coming from the product formula on curves over the base field $k$ (which is an arbitrary field) is existentially closed. It is worth to mention that as an immediate corollary from Theorem~\ref{theorem_B}B, one gets a global Fekete-Szeg\H{o} type result, i.e., \cite[Theorem 3.11]{GVF2} for number fields.

Theorem~\ref{theorem_B}B served as a motivation for this project and it is a consequence of Theorem~\ref{theorem_A}A because of a translation theorem between the languages of globally valued fields and of functionals on arithmetic $\RR$-divisors. To express the translation, fix a finitely generated extension $\QQ \subset \Field$ and as in Definition~\ref{definition_F_models_and_global_divisors_on_F}, let $\Adiv_{\RR}(\Field)$ be the injective limit of $\Adiv_{\RR}(\cX)$ for all (normal, generically smooth) arithmetic varieties $\cX$ with function field isomorphic to $\Field$. We call $l:\Adiv_{\RR}(\Field) \to \RR$ a \textit{GVF functional}, if it satisfies the following conditions, already appearing in Theorem~\ref{theorem_A}A:
\begin{itemize}
    \item it sends principal arithmetic $\RR$-divisors of $C^0$-type to zero,
    \item it is non-negative on effective arithmetic $\RR$-divisors of $C^0$-type.
\end{itemize}
Additionally, we call such $l$ \textit{normalised}, if $l((\cdiv(2),0))= \log 2$. For the details see Definition~\ref{definition_GVF_functional}. In Definition~\ref{definition_lattice_structure_and_divisors_on_F_and_on_F_models} we introduce a lattice structure on the space of arithmetic $\QQ$-divisors of $C^0$-type denoted $\Adiv_{\QQ}(\Field) \subset \Adiv_{\RR}(\Field)$. In particular for a $\QQ$-tropical polynomial $t$ (of arity $m$) the expression $t(\ov{\cD}_1, \dots, \ov{\cD}_m)$ makes sense for $\ov{\cD}_1, \dots, \ov{\cD}_m \in \Adiv_{\QQ}(\Field)$. We prove the following.
\begin{theorem_C}[Theorem~\ref{theorem_GVF_functional_are_GVF_structures}]~\label{theorem_C}
    Let $\Field$ be a finitely generated extension of $\QQ$. There is a natural bijection 
    \[ \bigl\{ \textnormal{GVF functionals $\Adiv_{\RR}(\Field) \to \RR$}  \bigr\} \longleftrightarrow \bigl\{ \textnormal{GVF structures on $F$}  \bigr\} \]
    \[ l \longmapsto \bigl( R_t^l(\ov{a}):=l(\ov{\cD}_t(\ov{a})) \bigr)_{\textnormal{$\QQ$-tropical polynomials } t} \]
    where $\ov{\cD}_t(\ov{a})=t(\adiv(a_1), \dots, \adiv(a_n))$ if $\ov{a} = (a_1, \dots, a_n)$.
\end{theorem_C}
This theorem formally extends to arbitrary extensions of $\QQ$, see Remark~\ref{remark_GVF_functional_are_GVF_structures}. Over finite prime fields an analogous geometric description of globally valued fields follows from \cite[Section 11]{GVF2}.

At the end of the paper we present a new interpretation of the essential infimum function $\zeta$. If $\ov{\cD}$ is an arithmetic $\RR$-divisor of $C^0$-type on a normal, generically smooth arithmetic variety $\cX$, then the essential infimum of $\ov{\cD}$ is defined to be
\[ \zeta(\ov{\cD}) := \sup_{\cY \subset \cX} \inf_{x \in \cX(\ov{\QQ}) \setminus \cY} h_{\ov{\cD}}(x), \]
where the supremum is taken over all Zariski-closed proper subschemes $\cY \subset \cX$. However, since points approximate arbitrary GVF functionals by Theorem~\ref{theorem_A}A, the following holds.
\begin{theorem_D}~\cite[Theorem 4.6]{Arithmetic_Demailly_Qu_Yin}~\label{theorem_D}
    Let $\ov{\cD}$ be an arithmetic $\RR$-divisor of $C^0$-type on $\cX$. Then
    \[ \zeta(\ov{\cD}) \leq \inf \{ l(\ov{\cD}) : \textnormal{$l$ is a normalised GVF functional on } \Adiv_{\RR}(\Field) \}. \]
    Moreover, if $\cD_\QQ$ is big, this is an equality.
\end{theorem_D}
This result was proved by Qu and Yin in a different language, and it generalises \cite[Corollary 4.2]{F_Ballay_Succesive_minima}. The author would like to thank François Balla\"y for providing a reference. We decided to include the above theorem in this introduction to underline the new interpretation of the essential infimum function $\zeta$. To see our presentation of the proof of Theorem~\ref{theorem_D}D using Theorem~\ref{theorem_A}A, see Theorem~\ref{theorem_essential_infimum_as_infimum_of_GVF_functionals}. Note that Theorem~\ref{theorem_A}A is stronger, in the sense that it allows to consider several different arithmetic $\RR$-divisor of $C^0$-type at once.

Assume that $\cX$ is a normal, geometrically smooth arithmetic variety of dimension $d+1$ and let $\ov{\cD}$ be a pseudo-effective arithmetic $\RR$-divisor of $C^0$-type on $\cX$, with $\cD_{\QQ}$ big. The essential infimum function is related to other arithmetic invariants via Zhang's inequality which states (in the form from \cite[Theorem 7.2.(ii)]{F_Ballay_Succesive_minima}) that under the above assumptions
\[ \zeta(\ov{\cD}) \geq \frac{\avol(\ov{\cD})}{(d+1) \vol(\cD_{\QQ})} =: \varphi(\ov{\cD}). \]
We prove the following theorem which gives a characterisation of when the equality occurs in Zhang's inequality.
\begin{theorem_E}[Theorem~\ref{theorem_characterisation_of_Zhang_equality_divisors}]~\label{theorem_E}
    Let $\ov{\cD}$ be a big arithmetic $\RR$-divisor of $C^0$-type on $\cX$. Consider the statements:
    \begin{enumerate}[label=(\arabic*)]
        \item $\varphi(\ov{\cD}) = \zeta(\ov{\cD})$;
        \item $D_{\ov{\cD}} \varphi$ is a GVF functional;
        \item the infimum
        \[ \zeta(\ov{\cD}) = \inf \{ l(\ov{\cD}) : \textnormal{$l$ is a normalised GVF functional on } \Adiv_{\RR}(\Field) \} \]
        is achieved at a unique normalised GVF functional;
        \item $\zeta$ is Gateaux differentiable at $\ov{\cD}$ in every direction.
    \end{enumerate}
    Then $(1) \iff (2) \implies (3) \iff (4)$.
\end{theorem_E}
Standard methods show that Gateaux differentiability of $\zeta$ at $\ov{\cD}$ gives equidistribution, see Lemma~\ref{lemma_equdistributuin_with_zeta_differentiable} and \cite[Section 8]{Chambert_Loir_survey}. From the logic viewpoint, equidistribution corresponds to uniqueness of a GVF functional (at least its non-generic part), satisfying some equations, in the case of Theorem~\ref{theorem_E}E of the form $l(\ov{\cD}) = \zeta(\ov{\cD})$. It is an interesting question why the classical equdistribution theorems are formulated with respect to a single divisor only.

We choose to work with arithmetic $\RR$-divisors of $C^0$-type in this paper, however, one could work in a more general context of adelic $\RR$-divisors on varieties over number fields (see \cite{Moriwaki_Adelic_divisors_on_Arithmetic_Varieties}, \cite{F_Ballay_Succesive_minima}). We only work over $\QQ$ as this is enough for existential closedness of $\GVFQ$, essentially by the claim in the proof of Theorem~\ref{theorem_QQ_bar_is_existentially_closed}. Also, from the GVF point of view passing from arithmetic to adelic divisors does not make much difference, since GVF functionals on arithmetic $\RR$-divisors of $C^0$-type extend uniquely to GVF functionals on adelic $\RR$-divisors on the generic fiber of the corresponding arithmetic variety. This follows from \cite[Theorem 4.1.3]{Moriwaki_Adelic_divisors_on_Arithmetic_Varieties}.

\subsection{The technique used to prove Theorem~\ref{theorem_A}A}\label{subsection_the_technique_Theorem_A}

In this subsection we describe the proof technique used to prove Theorem~\ref{theorem_A}A. It is an arithmetic generalisation of the argument for existential closedness of $\GVFk$ from \cite{GVF2} for a fixed base field $\smallk$. Assume that $\ov{\smallk} = \smallk$. For a projective variety $X$ over $\smallk$, let $N^1(X)$ be the N\'eron–Severi group of $X$ tensored with $\RR$. The crucial geometric ingredient in \cite{GVF2} is the following theorem.
\begin{theorem*}~\cite[Theorem 10.9, Corollary 10.11]{GVF2}
    Let $X$ be a projective smooth variety with a dominant morphism to the projective line $\pi:X \to \PP^1$ over $\smallk$. Assume that $l:N^1(X) \to \RR$ is a linear functional having the following properties:
    \begin{itemize}
        \item $l$ is non-negative on classes of effective divisors;
        \item for $D = \pi^* p$, where $p \in \PP^1$ is closed, we have $l(D) = \deg_{\PP^1}(p) = 1$ (by algebraic closedness).
    \end{itemize}
    Fix $\varepsilon > 0$ and let $D_1, \dots, D_n$ be a basis of $N^1(X)$. Then there exists an irreducible curve $C \subset X$ such that for all $i=1, \dots, n$ we have
    \[ \Big| \frac{\deg_{C}(D_i)}{\deg(C/\PP^1)} - l(D_i) \Big| < \varepsilon. \]
\end{theorem*}
Note that in this case, if $\eta$ is the generic point of $\PP^1$, then $C$ is determined by its generic point $\eta' \in X_{\eta}$ and $\frac{\deg_{C}(D_i)}{\deg(C/\PP^1)}$ can be thought of as $h_{D_i}(\eta')$.

To outline the proof of this theorem from \cite{GVF2}, we recall the terminology concerning divisors and curves on a projective smooth variety $X$ over $\smallk$. For details, see \cite{Lazarsfeld_Positivity}. Assume that $X$ is of dimension $d+1$.
\begin{itemize}
    \item By $\Z_1(X), \Z_{d}(X) = \Z^1(X)$ we denote the cycles of dimension $1$ and $d$ respectively, in $X$. This means that these are free abelian groups generated by integral subvarieties of the corresponding dimension.
    \item There is a bilinear intersection pairing $\cdot:\Z_1(X) \times \Z^1(X) \to \ZZ$. It induces the numerical equivalences $\equiv$ on $\Z_1(X), \Z^1(X)$ where
    \[ D \equiv D' \iff (\textnormal{for all }C \in \Z_1(X))(C \cdot D = C \cdot D'), \]
    \[ C \equiv C' \iff (\textnormal{for all }D \in \Z^1(X))(C \cdot D = C' \cdot D). \]
    We use the following notation for the quotients tensored with reals: $N^1(X) := (\Z^1(X)/\equiv) \otimes \RR, N_1(X) := (\Z_1(X)/\equiv) \otimes \RR$. Note that $N_1(X), N^1(X)$ are finite dimensional, and the induced pairing $\cdot:N_1(X) \times N^1(X) \to \RR$ is perfect of signature $(1, -1, -1, \dots, -1)$, by the Hodge index theorem.
    \item Let $\Eff_1(X) \subset N_1(X), \Eff^1(X) \subset N^1(X)$ be the closed convex cones generated by effective curves and effective divisors in $N_1(X)$ and $N^1(X)$ respectively. Denote by $N_1^+(X), N_+^1(X)$ the dual cones. More precisely
    \[ N_1^+(X) = \{ C : (\textnormal{for all }D \in \Eff^1(X))(C \cdot D \geq 0) \}, \]
    \[ N_+^1(X) = \{ D : (\textnormal{for all }C \in \Eff_1(X))(C \cdot D \geq 0) \}. \]
    We call elements of $N_1^+(X)$ movable curves classes, and elements of $N_+^1(X)$ numerically effective (nef) divisors.
    \item Any $D \in \Z^1(X)$ is a Cartier divisor, so it induces a line bundle $\cO(D)$ on $X$. Its volume is the quantity
    \[ \vol(D) := \limsup_{m \to \infty} \frac{\dim_k H^0(X, \cO(mD))}{m^{d+1}/(d+1)!}. \]
    It extends to a continuous function $\vol:N^1(X) \to \RR$. The big cone is the set
     \[ \Bigcone(X) := \{ D : \vol(D) > 0 \}. \]
    It can be equivalently described as the (euclidean) interior of $\Eff^1(X)$. On the other hand, the interior of $N_+^1(X)$ is the convex cone generated by classes of ample divisors $D$, i.e., such that the linear system $|mD|$ induces an embedding into a projective space, for $m$ big enough.
    \item The volume function $\vol:N^1(X) \to \RR$ is differentiable on the big cone and the derivative at $D \in \Bigcone(X)$ is given by $d+1$ times the positive intersection product $\langle D^{d} \rangle:N^1(X) \to \RR$. The positive intersection product can be seen as an element $\langle D^{d} \rangle \in N_1^+(X)$. Moreover, $\langle D^{d} \rangle$ can be approximated by pushforwards of $d$'th self intersections of ample $D'$ on blowups $f:X' \to X$ with $D' \leq f^*D$. If $D$ itself is an ample Cartier divisor, then $\langle D^{d} \rangle$ can be given by the $d$'th self intersection of $D$. In this case one can use the Theorem of Bertini to get an irreducible curve. For the details of this point see \cite{Boucksom_DIFFERENTIABILITY_OF_VOLUMES}, \cite[Section 10]{GVF2}.
\end{itemize}

\begin{proof}[Sketch of the proof of the Theorem]
    We divide the argument from \cite{GVF2} into the following steps.
    \begin{enumerate}
        \item Replace $l$ by an approximation that is strictly positive on $\Eff^1(X) \setminus \{0\}$.
        \item Consider the function
        \[ \gamma := \frac{\vol^{1/(d+1)}}{l} : N^1(X) \setminus \{0\} \to \RR. \]
        By the first point it is continuous, and for $t>0$ we have $\gamma(D) = \gamma(tD)$. The restriction of $\gamma$ to the unit sphere of $N^1(X)$ determines the whole function, and by compactness of the sphere it achieves maximum. If at $B$ the maximum is achieved, then by calculating the derivative of (the logarithm of) $\gamma$ at $B$ we get
        \[ \frac{\langle B^{d} \rangle}{\vol(B)} = \frac{l}{l(B)}. \]
        \item Pick a blowup $f:X' \to X$ and an ample $B'$ on $X'$ with $B' \leq f^*B$ such that the pushforward of the $d$'th self intersection of $B'$ approximates $\langle B^{d} \rangle$ and $\vol(B')$ approximates $\vol(B)$. For simplicity assume that $X'=X, B'=B$. Let $C$ be a $d$'th self intersection of $B$, which is irreducible by the Bertini theorem.
        \item We get the equation
        \[ \frac{C}{\vol(B)} = \frac{l}{l(B)}. \]
        By applying this equation on $\pi^*p$, where $p$ is some closed point of $\PP^1$, we get
        \[ \frac{\deg(C/\PP^1)}{\vol(B)} = \frac{\deg_C(\pi^*p)}{\vol(B)} = \frac{l(\pi^*p)}{l(B)} = \frac{1}{l(B)}, \]
        where the first equality follows from the projection formula for $\pi$. Thus
        \[ \deg(C/\PP^1) = \frac{\vol(B)}{l(B)}. \]
        By putting everything together we get
        \[ \frac{C}{\deg(C/\PP^1)}=l, \]
        Note that in the first step we replaced $l$ by an approximation, so at the end we get the conclusion if the approximation was good enough.
    \end{enumerate}
\end{proof}

The proof of Theorem~\ref{theorem_A}A uses the same steps, but we need to work with arithmetic $\RR$-divisors of $C^0$-type, instead of Cartier divisors. Fortunately the necessary arithmetic variants of the geometric theorems used in the above sketch are available. However, in the arithmetic context a few additional steps are necessary. More precisely, in the arithmetic analogue of step (3) from the above sketch, a lower bound on the (geometric) volume of the generic part of the divisor is needed. This bound is obtained by using arithmetic Fujita approximation \cite{Chen2010_Fujita_approximation}, \cite{yuan_2009_Fujita_approximation} together with Zhang's inequality \cite{Zhang_thesis_inequality}, see Lemma~\ref{lemma_calculating_derivative_of_volume_uniformely_for_a_few_divisors} and Proposition~\ref{proposition_formula_for_derivative_of_avol_ample_in_blowup}.  The other crucial ingredients are differentiability of the arithmetic volume \cite{Chen_differentiability_of_arithmetic_volume} (see also \cite{Ikoma_Concavity_of_arithmetic_volume}) and arithmetic Bertini theorems \cite{Charles2021ArithmeticAA} and \cite{arithmetic_Bertini_Robert_Wilms}. Note that \cite{Charles2021ArithmeticAA} was motivated by a question of Hrushovski at the SAGA in Orsay, regarding arithmetic type Bertini theorems.

\subsection{Connections with other results}

Existential closedness of $\GVFQ$ and $\GVFk$ sets a goal of axiomatizing existencially closed models of the theory $\GVF$. Ben Yaacov has published research notes \cite{Ben_Yaacov_Vandermonde}, \cite{Ben_Yaacov_Estimates_on_volumes} related to this aim. The model companion is conjectured to exists, and it is possible that it is tame in some model theoretic sense. For example, by \cite[Theorem 14.2]{GVF3}, GVF structures are quantifier-free stable (i.e., quantifier-free formulas don't have the order property in the continuous logic sense).

On the other hand, a more geometrical approach to globally valued fields was developed by Chen and Moriwaki in the series of papers about \textit{adelic curves} \cite{Adelic_curves_1}, \cite{Adelic_curves_2}, \cite{Adelic_curves_3}, \cite{Adelic_curves_4}. Adelic curves are not exactly the same as globally valued fields, but one can translate from one world to another. The authors in \cite{Adelic_curves_2} introduced Arakelov intersection theory over an arbitrary countable adelic curve. Differentiability of the volume function in the context of adelic curves may lead to generalisations of the main result of this paper to a bigger class of GVF structures. The results in \cite{Wenbin_Luo_relative_Siu}, \cite{A_Sedillot_diff_of_relative_volume} are very promising in this direction.

\subsection{Acknowledges}

The author wishes to thank his DPhil supervisor Ehud Hrushovski for numerous discussions and many ideas. This project owes a huge intellectual debt to him. The author also wishes to thank Francesco Gallinaro and Robert Wilms for giving comments on an early draft of this paper, and to François Ballaÿ for providing a reference to \cite{Arithmetic_Demailly_Qu_Yin}. For helpful comments on the introduction, thanks go to Francesco Ballini, Andrés Ibáñez Núñez and Finn Wiersig.

\subsection{Outline}

The article is organized as follows. In Section~\ref{section_2} we introduce all the necessary background on arithmetic $\RR$-divisors of $C^0$-type. In Subsection~\ref{subsection_approx_global_functional} we give the proof of Theorem~\ref{theorem_A}A. Next in Section~\ref{section_3} we introduce lattice divisors and use them to prove Theorem~\ref{theorem_B}B using the translation from Theorem~\ref{theorem_C}C. In Subsection~\ref{subsection_ess_infimum_and_GVF_functionals} we prove Theorem~\ref{theorem_D}D and Theorem~\ref{theorem_E}E.

\section{Intersection theory of arithmetic divisors}\label{section_2}
\begin{notation*}
    We stick with the following conventions.
    \begin{enumerate}
        \item If $f$ is a function on a real vector space, then we write $D_x f(y)$ for the derivative of $f$ at $x$ in direction $y$, so that $D_x f$ is a linear map (if $f$ is Gateaux differentiable at $x$).
        \item If $A$ is an abelian group, and $\KK$ is one of $\ZZ, \QQ, \RR$ we write $A_{\KK}$ for the tensor product $A \otimes_{\ZZ} \KK$. If $A$ is written multiplicatively, then for $f_1, \dots, f_n \in A$ and $\alpha_1, \dots, \alpha_n \in \KK$ we write $f_1^{\alpha_1} \cdots f_n^{\alpha_n}$ for the product $f_1 \otimes \alpha_1 \cdot \ldots \cdot f_n \otimes \alpha_n$.
        \item For a field $\Field$ we call a function $v:\Field^{\times} \to \RR$ a \textit{non-Archimedean valuation} if it is a multiplicative homomorphism and satisfies the ultrametric inequality (with the convention $v(0) = \infty$).
        \item By an \textit{Archimedean valuation} on a field $\Field$ we mean a function $v:\Field^{\times} \to \RR$ of the form $v(-) = -\log|\sigma(-)|$, where $\sigma:\Field \to \CC$ is an embedding of fields. Note that conjugate embeddings give the same Archimedean valuation and every Archimedean valuation comes from a pair of conjugate embeddings (or from a single real embedding). 
        \item By a \textit{valuation} we mean either a non-Archimedean or an Archimedean valuation on a field and we denote by $\Val_{\Field}$ the space of valuations on a field $\Field$ (with the topology coming from the embedding into $\RR^{\Field \setminus \{0\}}$).
        \item If $\cX$ is a scheme over $\Spec(A)$ and we have a map $\Spec(B) \to \Spec(A)$ we write $\cX_B$ or $\cX \otimes B$ for the fiber product. By a proper subscheme of $\cX$ we mean a subscheme that is not equal to $\cX$.
        \item If $p \in \cX$ is a point in a scheme, we write $\kappa(p)$ for the residue field. The function field of $\cX$ is denoted by $\kappa(\cX)$. % If the scheme $\cX$ is over $\Spec(K)$ for a field $K$, we write $K(\cX)$ for $\kappa(\cX)$ as an $K$-algebra.
        \item For a scheme $\cX$ we denote by $\Z_d(\cX)$ the group of $d$-dimensional cycles on $\cX$, i.e., the free abelian group generated by dimension $d$ integral subschemes of $\cX$. Moreover, we denote by $\Div(\cX)$ the group of Cartier divisors on $\cX$ and by $\cO_{\cX}(\cD)$ the invertible subsheaf of the sheaf of rational functions on $\cX$ corresponding to $\cD$, as defined in \cite[Chapter II, Section 6]{Hartshorne_AG} or \cite[Definition 2.4.2]{Adelic_curves_1}. Note that then the rational function $1$ defines a canonical rational section of $\cO_{\cX}(\cD)$.
        \item If $\cL$ is a line bundle on a scheme $\cX$ and $s$ is a rational section of $\cL$, then as in \cite[Remark 2.4.4]{Adelic_curves_1}, by $\cdiv(s)$ we denote the Cartier divisor on $\cX$ associated to the invertible subsheaf of the sheaf $\cK$ of invertible rational functions on $\cX$, defined by the image of the embedding $\cL \to \cK$ sending $s$ to $1$. We use the same notation for the cycle associated with this Cartier divisor.
        \item An \textit{arithmetic variety} in this paper is an integral (irreducible and reduced) scheme, finite type and projective over $\Spec(\ZZ)$. It is generically smooth if the base change $\cX_\QQ$ is smooth over $\Spec(\QQ)$. We write $\deg(x) = [\kappa(x):\QQ]$ for a closed point of the generic fiber $x \in \cX_\QQ \subset \cX$.
        \item If $\cX$ is an arithmetic variety and $v$ is a non-Archimedean valuation on $\kappa(\cX)$, we write $\supp(v) \in \cX$ for the image of the maximal ideal of the value ring $\cO_v \subset \kappa(\cX)$ via the unique map $\phi$ making the following diagram commute:
        % https://q.uiver.app/?q=WzAsNCxbMCwwLCJcXFNwZWMoXFxrYXBwYSh4KSkiXSxbMCwxLCJcXFNwZWMoXFxjT192KSJdLFsxLDAsIlxcY1giXSxbMSwxLCJcXFNwZWMoXFxaWikiXSxbMiwzXSxbMCwxXSxbMCwyXSxbMSwzXSxbMSwyLCJcXHBoaSIsMSx7InN0eWxlIjp7ImJvZHkiOnsibmFtZSI6ImRhc2hlZCJ9fX1dXQ==
        \[\begin{tikzcd}
        	{\Spec(\kappa(\cX))} & \cX \\
        	{\Spec(\cO_v)} & {\Spec(\ZZ)}
        	\arrow[from=1-2, to=2-2]
        	\arrow[from=1-1, to=2-1]
        	\arrow[from=1-1, to=1-2]
        	\arrow[from=2-1, to=2-2]
        	\arrow["\phi"{description}, dashed, from=2-1, to=1-2].
        \end{tikzcd}\]
        Such map always exists by the valuative criterion of properness.
    \end{enumerate}
\end{notation*}

\subsection{Hermitian line bundles on arithmetic varieties}\label{section_hermitian_line_bundles_on_arithmetic_varieties}
For this subsection let $\cX$ be a generically smooth arithmetic variety. Let $d+1$ be the dimension of $\cX$ and let $X = \cX_\QQ$ be the generic fiber.

\begin{warning}
    In this subsection $\cX$ is \textbf{not} necessarily \textbf{normal}.
\end{warning}

\begin{remark}~\label{definition_C_type_functions_and_various_Cartier_divisors}
       We use the following conventions regarding complex analytifications of arithmetic varieties.
       \begin{enumerate}
            \item By $\cX(\CC)$ we mean the complex analytification of the base change $\cX_{\CC}=\cX \otimes \CC$. It is in a natural bijection with the set of maps $\Spec(\CC) \to \cX$. The complex conjugation induces a map $\Spec(\CC) \to \Spec(\CC)$ and by precomposition it induces \textit{complex conjugation} $F_{\infty}:\cX(\CC) \to \cX(\CC)$.
            \item In the context of the first point, if $x \in X$ is a closed point and $\sigma:\kappa(x) \to \CC$ is an embedding of fields, we denote by $x^{\sigma}$ the point of $\cX(\CC)$ corresponding to the composition
            \[ \Spec(\CC) \to \Spec(\kappa(x)) \to X \to \cX. \]
       \end{enumerate}
\end{remark}
\begin{definition}~\label{definition_hermitian_line_bundle}
    Let $M$ be a complex manifold with a (complex) line bundle $L$. A \textit{hermitian metric} $h$ on $L$ is a collection of hermitian metrics $h_x$ on fibers $L_x$ of $L$ for all $x \in M$. It is \textit{smooth} if for all open $U \subset M$ and all sections $s$ of $L$ over $U$, the function $|s|_h^2 (x) := h_x(s(x),s(x))$ is a smooth function $U \to \RR$. Since the data of a smooth hermitian metric is determined by functions $|\cdot|_h$ (and vice versa), we use the name ``smooth hermitian metric'' also for collections of norms of sections coming from some $h$ as described above. The \textit{Chern form} of $\ov{L} := (L,h)$ is a differential $(1,1)$-form $c_1(\ov{L})$ on $M$ such that if $s$ is an invertible (holomorphic) section of $L$ over $U \subset M$, then $c_1(\ov{L})|_U = \frac{i}{2 \pi} \partial \ov{\partial} \log |s|_h^2$.
\end{definition}
\begin{definition}
    Let $M = \cX(\CC)$ and let $\ov{L} = (L,h)$ be a line bundle on $M$ equipped with a hermitian metric $h$. We say that $h$ is of \textit{real type}, if for all $x \in M$ with $\ov{x} := F_{\infty}(x)$ and for all sections $s, s' \in L_{x}$ we have
    \[ h_x(s, s') = \ov{h_{\ov{x}} (F_{\infty}(s), F_{\infty}(s'))}, \]
    where the action on germs of sections is defined by precomposition with $F_{\infty}$.
\end{definition}
\begin{definition}
    A \textit{hermitian line bundle} $\ov{\cL}=(\cL, |\cdot|_{\ov{\cL}})$ on $\cX$ is a pair consisting of a line bundle $\cL$ on $\cX$ and a smooth hermitian metric of real type $|\cdot|_{\ov{\cL}}$ on the complex analytification of the line bundle $\cL_{\CC}$ on $\cX_\CC$. A section $s \in H^0(\cX, \cL)$ is \textit{(strictly) small}, if $|s|_{\ov{\cL}} \leq 1$ (resp. $|s|_{\ov{\cL}} < 1$) on $\cX(\CC)$. The (finite) set of small sections is denoted by $\aH^0(\cX, \ov{\cL})$. The Chern form of the analytification of $\cL_\CC$ on $\cX_\CC$ is denoted by $c_1(\ov{\cL})$. 
\end{definition}

\begin{theorem}~\cite[Theorem (2.5)]{Chambert_Loir_survey}~\label{theorem_intersection_product_of_hermitian_line_bundles}
    Let $\ov{\cL}_0, \dots, \ov{\cL}_d$ be hermitian line bundles on $\cX$. There exists a unique family of linear maps
    \[ \adeg(\ov{\cL}_0 \cdots \ov{\cL}_{n-1} | -):\Z_n(\cX) \to \RR, \]
    for $n=0, \dots, d+1$ (for $n=0$ we mean a map $\adeg(-):\Z_0(\cX) \to \RR$) satisfying the following properties:
    \begin{enumerate}
        \item  For every integer $n \in \{ 0, \dots, d\}$, every integral closed subscheme $\cZ$ of $\cX$ such that $\dim(\cZ) = n+1$, every integer $m \neq 0$ and every regular meromorphic (i.e., defined on a dense open subset of $\cZ$) section $s$ of $\cL_n^{\otimes m} |_{\cZ}$, one has
        \[ m \ \adeg(\ov{\cL}_0 \cdots \ov{\cL}_{n}|\cZ) \]
        \[ = \adeg(\ov{\cL}_0 \cdots \ov{\cL}_{n-1}|\cdiv(s))  + \int_{\cZ(\CC)} - \log \|s\| \wedge c_1(\ov{\cL}_0) \wedge \ldots \wedge c_1(\ov{\cL}_{n-1}). \]
        \item For a closed point $z$ of $\cX$ viewed as a $0$-dimensional cycle, one has
        \[ \adeg(z) = \log \#\kappa(z). \]
    \end{enumerate}
    Moreover, these maps are multilinear and symmetric in the hermitian line bundles $\ov{\cL}_0, \dots, \ov{\cL}_d$ and only depend on their isomorphism classes as hermitian line bundles.
\end{theorem}
\begin{notation}~\label{notation_arithmetic_degree_w_r_t_hermitian_line_bundle}
    If $n=d+1$ in the context of the above theorem, we skip $\cX$ in the notation and write
    \[ \adeg(\ov{\cL}_0 \cdots \ov{\cL}_{d}) \]
    or just $\ov{\cL}_0 \cdot \ldots \cdot \ov{\cL}_{d}$.
\end{notation}
\begin{remark}
    For a morphism of generically smooth arithmetic varieties $f:\cY \to \cX$, if $\ov{\cL} = (\cL, h)$ is a hermitian line bundle on $\cX$, then one can pullback it to $\cY$ to get a hermitian line bundle $f^* \ov{\cL} := (f^* \cL, f^* h)$ on $\cY$. 
\end{remark}
\begin{remark}
    Let $\ov{\cL}_0, \dots, \ov{\cL}_{l-1}$ be hermitian line bundles on $\cX$. If $i:\cZ \hookrightarrow \cX$ is an $l$-dimensional irreducible reduced subscheme which is generically smooth, then the following intersection numbers are the same
    \[ \adeg(\ov{\cL}_0 \cdots \ov{\cL}_{l-1}|\cZ) = \adeg(i^*\ov{\cL}_0 \cdots i^*\ov{\cL}_{l-1}). \]
    Here on the right hand side we treat $\cZ$ as a generically smooth arithmetic variety, and $i^*\ov{\cL}_0, \dots, i^*\ov{\cL}_{l-1}$ are hermitian line bundles on $\cZ$. In this case we also write $\ov{\cL}_0|_{\cZ} \cdot \ldots \cdot \ov{\cL}_{l-1}|_{\cZ}$ for this intersection number.
\end{remark}
\begin{theorem}~\cite[Proposition (2.9)]{Chambert_Loir_survey}~\label{theorem_arithmetic_projection_formula}
    Let $f: \cX' \to \cX$ be a generically finite morphism of arithmetic varieties, let $\cZ$ be an integral closed subscheme of $\cX'$ and let $l = \dim(\cZ)$. Assume that $\ov{\cL}_0, \dots, \ov{\cL}_{l-1}$ are hermitian line bundles on $\cX$.
    \begin{enumerate}
        \item If $\dim(f(\cZ)) < l$, then $\adeg(f^* \ov{\cL}_0 \cdots f^* \ov{\cL}_{l-1}|\cZ) = 0$.
        \item Otherwise, $\dim(f(\cZ)) = l$ and
        \[ \adeg(f^* \ov{\cL}_0 \cdots f^* \ov{\cL}_{l-1}|\cZ) = \adeg(\ov{\cL}_0 \cdots \ov{\cL}_{l-1}|f_*(\cZ)). \]
    \end{enumerate}
    where $f_*(\cZ) = [\kappa(\cZ) : \kappa(f(\cZ))]f(\cZ)$ is an $l$-cycle on $\cX$.
\end{theorem}

\begin{definition}
    A hermitian line bundle $\ov{\cL}$ on $\cX$ is $\textit{ample}$ if the following three conditions are satisfied:
    \begin{itemize}
        \item $\cL$ is relatively ample with respect to $\cX \to \Spec(\ZZ)$,
        \item $c_1(\ov{\cL})$ is positive,
        \item $H^0(\cX, n \cL)$ is generated (as a $\ZZ$-module) by strictly small sections of $n\ov{\cL}$, for $n$ big enough.
    \end{itemize}
\end{definition}

\begin{definition}~\label{definition_irreducibility_defined_by_Wilms}~\cite[Definition 1.1]{arithmetic_Bertini_Robert_Wilms}
    Let $\ov{\cM}, \ov{\cL}$ be ample hermitian line bundles on $\cX$. Denote by $\|\cdot\|$ the hermitian norm of $\ov{\cL}$. A section $s \in \aH^0(\cX, \ov{\cL})$ is \textit{$(\varepsilon, \ov{\cM})$-irreducible} if the following conditions are satisfied.
    \begin{enumerate}
        \item The $d$-cycle $\cZ := \cdiv(s)$ is irreducible (hence an arithmetic variety).
        \item The following inequality holds
        \[ |\adeg(\ov{\cM}^d \cdot \ov{\cL}) - \adeg(\ov{\cM}^d|\cZ)| < \varepsilon. \]
        %\item The Chern form $c_1(\ov{\cL})$ is semi-positive.
    \end{enumerate}
    By the inductive definition of the arithmetic intersection number (see Theorem~\ref{theorem_intersection_product_of_hermitian_line_bundles}), we can present the second condition in the following equivalent forms:
    \[ \int_{\cX(\CC)} -\log\|s\|^2 \wedge c_1(\ov{\cM})^d < \varepsilon \cdot \adeg(\ov{\cM}^d|\cdiv(s)) \]
    or
    \[ \int_{\cX(\CC)} -\log\|s\|^2 \wedge c_1(\ov{\cM})^d < \frac{\varepsilon}{1+\varepsilon} \cdot \adeg(\ov{\cM}^d \cdot \ov{\cL}). \]
    % This is \cite[Proposition 6.3.(a)]{arithmetic_Bertini_Robert_Wilms}.
\end{definition}

\begin{remark}
    The definition from \cite[Definition 1.1]{arithmetic_Bertini_Robert_Wilms} is more general and has an additional assumption that the current $\frac{i}{2 \pi} \partial \ov{\partial} (-\log \|s\|^2) + \delta_{\cZ(\CC)}$ is represented by a semi-positive form. However, in our case, by the Poincar\'e–Lelong formula we get
    \[ \frac{i}{2 \pi} \partial \ov{\partial} (-\log \|s\|^2) + \delta_{\cZ(\CC)} = [c_1(\ov{\cL})] \]
    as currents. Since $\ov{\cL}$ is an ample hermitian line bundle, the form $c_1(\ov{\cL})$ is positive, so the additional assumption is automatically satisfied in this case.
\end{remark}

\begin{theorem}~\label{theorem_arithmetic_Bertini_by_Wilms}~\cite[Remark 6.7.(ii)]{arithmetic_Bertini_Robert_Wilms}
    Assume that $\dim(\cX) \geq 2$. Let $\ov{\cL}, \ov{\cM}$ be ample hermitian line bundles on $\cX$. Fix $\varepsilon > 0$. Then
    \[ \lim_{n \to \infty} \frac{ \# \{ s \in \aH^0(\cX, \ov{\cL}^{\otimes n}) : s \textnormal{ is } (\varepsilon, \ov{\cM}) \textnormal{-irreducible}, \cdiv(s)_{\QQ} \textnormal{ is smooth} \} }{ \# \aH^0(\cX, \ov{\cL}^{\otimes n}) } = 1. \]
\end{theorem}

\begin{theorem}~\label{theorem_Charles_Bertini_to_omit_subvariety}~\cite[Theorem 2.21]{Charles2021ArithmeticAA}.
    Let $\ov{\cL}$ be an ample hermitian line bundle on $\cX$ and $\cY \subset \cX$ be a closed proper subscheme. Then
    \[ \lim_{n \to \infty} \frac{ \# \{ s \in \aH^0(\cX, \ov{\cL}^{\otimes n}) : \cdiv(s) \textnormal{ is not contained in } \cY \} }{ \# \aH^0(\cX, \ov{\cL}^{\otimes n}) } = 1. \]
\end{theorem}
%\begin{proof}
%    This follows from \cite[Theorem 2.22]{Charles2021ArithmeticAA}.
%\end{proof}

\subsection{Arithmetic $\RR$-divisors}
\begin{warning}
    For the rest of this section (unless stated otherwise) we impose \textbf{stronger} assumptions on $\cX$. Namely here $\cX$ is a generically smooth,  \textbf{normal} arithmetic variety of dimension $d+1$.
\end{warning}
\begin{remark}~\label{remark_subsheaves_definition}
    We recall here definitions of objects related to $\cX$ that we need.
    \begin{enumerate}
        \item Let $\cJ$ be a subsheaf of the sheaf of continuous real valued functions on $\cX(\CC)$ (with respect to the euclidean topology). A \textit{$\cJ$- function} (or a $\cJ$-type function) on an euclidean open conjugation invariant subset $U \subset \cX(\CC)$ is a continuous function $U \to \RR$ lying in $\cJ(U)$ that is invariant under complex conjugation.
                \item We use three subsheafs $\cJ$ of the sheaf of continuous (real valued) functions on $\cX(\CC)$, in the above context:
                \begin{itemize}
                    \item $\cJ = C^0$ - the full sheaf of continuous functions,
                    \item $\cJ = C^{\infty}$ - the subsheaf of smooth functions,
                    \item $\cJ = C^0 \cap \PSH$ - the subsheaf of plurisubharmonic continuous functions.
                \end{itemize}
                For the definition of a plurisubharmonic function, see \cite[Section 1.4]{Moriwaki_Adelic_divisors_on_Arithmetic_Varieties}.
                %\item A \textit{$C^0$-type [resp. $C^{\infty}$-type] function} on an open conjugation invariant subset $U \subset \cX(\CC)$ is a continuous [resp. smooth] function $U \to \RR$ that is invariant under complex conjugation.
                %\item If $f$ is a smooth real valued function on an euclidean open subset $U \subset \cX(\CC)$, then by $dd^c f$ we denote 
                \item Let $\KK$ be one of $\ZZ, \QQ, \RR$. A \textit{$\KK$-Cartier divisor} on $\cX$ is an element of the group $\Div(\cX)_{\KK}$. A \textit{$\KK$-rational function} is an element of $\kappa(\cX)_{\KK}^\times$. Let $\cD = \sum_i \alpha_i \cD_i \in \Div(\cX)_{\KK}$ for $\alpha_i \in \KK$ and Cartier divisors $\cD_i$. If $f_i$ is a local equation for $\cD_i$ for all $i$, then we call $\prod_i f_i^{\alpha_i} \in \kappa(\cX)_{\KK}^{\times}$ a local equation for $\cD$. On the other hand, if $f \in \kappa(\cX)_{\KK}^{\times}$, then we write $(f)$ or $\cdiv(f)$ for the corresponding $\KK$-Cartier divisor. For $\cD \in \Div(\cX)_{\KK}$ we write $\cD_{\QQ} \in \Div(\cX_{\QQ})_{\KK}$ for its restriction to the generic fiber. 
                \item For an $\RR$-Cartier divisor $\cD$ on $\cX$, its \textit{support} is the set
                \[ \supp(\cD) := \{ p \in \cX : f \not\in (\cO_{\cX,p}^{\times})_{\RR} \textnormal{ for a local equation $f$ of $\cD$ at $p$} \}. \]
                If $\cD = \sum_i \alpha_i \cD_i$ is a decomposition as in the previous point (for $\KK=\RR$) with $\alpha_i$'s linearly independent over $\QQ$, then $\supp(\cD) = \bigcup_i \supp(\cD_i)$ where the latter $\supp$ is the support of a Cartier divisor. For details see \cite[Section 1.2]{Moriwaki_Adelic_divisors_on_Arithmetic_Varieties}.
                \item A \textit{$\ZZ$-Zariski open} set in $\cX(\CC)$ is a set of the form $\cV(\CC) \subset \cX(\CC)$ for some Zariski open $\cV \subset \cX$ (i.e., $\cV$ open in the topology of the scheme $\cX$). If $p \in \cX(\CC)$, then a \textit{$\ZZ$-Zariski neighbourhood} is an open neighbourhood of $p$ being a $\ZZ$-Zariski open set.
    \end{enumerate}
\end{remark}

\begin{definition}~\label{definition_Green_function}
    Let $\cD$ be an $\RR$-Cartier divisor on $\cX$. A \textit{$\cD$-Green function on $\cX$ of $\cJ$-type} (if $\cJ$ is $C^0$, we skip it in the notation) is a function $g:\cU(\CC) \to \RR$ for some Zariski open $\cU \subset \cX$ such that if $\cD = \sum a_i \cD_i$ for some real $a_i$'s and Cartier divisors $\cD_i$'s, then for every $p \in \cX(\CC)$ there is a $\ZZ$-Zariski neighbourhood $p \in \cV(\CC) \subset \cX(\CC)$ with local equations $f_i$'s for $\cD_i$'s, such that the function
    \[ g(x) + \sum_i a_i \log |f_i(x)|^2 \]
     restricted to the intersection of $\cU(\CC)$ and $\cV(\CC)$, extends to a $\cJ$-function on $\cV(\CC)$. We identify arithmetic $\RR$-divisors of $C^0$-type $(\cD, g) = (\cD, g')$ if $g:\cU(\CC) \to \RR, g':\cU'(\CC) \to \RR$ coincide on the intersection $(\cU \cap \cU')(\CC)$.
\end{definition}

\begin{remark}~\label{remark_open_set_of_definition_of_g_contains_support_of_divisors}
    The identification in the above definition makes sense, because for a pair $(\cD, g)$ as above, an extension $g':\cU'(\CC) \to \RR$ for $\cU \subset \cU'$ is unique if it exists. This follows from the fact that complex points of a Zariski open subset of $\cX$ are dense in the euclidean topology on $\cX(\CC)$. Thus, for every arithmetic $\RR$-divisor of $C^0$-type, there exists the biggest Zariski open $\cU \subset \cX$ such that $g$ can be defined on $\cU(\CC)$ (just take the sum of all possible extensions and glue by uniqueness) and this open set is $\cX \setminus \supp(\cD)$.
\end{remark}

\begin{definition}~\label{definition_arithmetic_R_divisors}
    An \textit{arithmetic $\RR$-divisor of $C^0$-type [resp. $\cJ$-type] on $\cX$} is a pair $\ov{\cD} = (\cD,g)$, where $\cD$ is an $\RR$-Cartier divisor on $\cX$ and $g$ is a $\cD$-Green function on $\cX$ [of $\cJ$-type]. An arithmetic $\RR$-divisor of $C^{0}$-type is called \textit{principal}, if it can be written as an $\RR$-combination of arithmetic $\RR$-divisors (of $C^{\infty}$-type) $\adiv(f) := (\cdiv(f), - \log|f|^2)$ for some rational functions $f$ on $\cX$. Also, $(\cD,g)$ is an \textit{arithmetic $\ZZ$-divisor [resp. $\QQ$-divisor] of $C^0$-type (or $\cJ$-type)} if it is an arithmetic $\RR$-divisor of the corresponding type, with $\cD$ being a Cartier [resp. $\QQ$-Cartier] divisor on $\cX$.
\end{definition}

\begin{remark}~\label{remark_pullbacks_of_arithmetic_R_divisors_of_C_0_type}
    Let $\phi:\cY \to \cX$ be a map of generically smooth, normal arithmetic varieties and let $\ov{\cD}=(\cD, g)$ be an arithmetic $\RR$-divisor of $C^0$-type on $\cX$. If the pullback $\phi^* \cD$ exists and $g$ can be defined on an open subset $\cU$ of $\cX$ whose preimage is non-empty in $\cY$, one can equip $\phi^* \cD$ with a Green function being the pullback of $g$. In particular, if $\cY \to \cX$ is a birational map, one can pullback arithmetic $\RR$-divisors of $C^0$-type.
\end{remark}

\begin{notation}
    Arithmetic $\RR$-divisors of $C^0$-type form an abelian group (with component-wise addition) and this group is denoted by $\Adiv_{\RR}(\cX)$. The subgroup of principal arithmetic $\RR$-divisors of $C^0$-type is denoted by $\Rdiv_{\RR}(\cX)$. We use the symbol $\equiv$ to indicate equality up to principal arithmetic $\RR$-divisors of $C^0$-type. The quotient $\Adiv_{\RR}(\cX)/\Rdiv_{\RR}(\cX)$ is denoted by $\aN^1(\cX)$. Note that this notation is not standard, but we use it so that the proof of Theorem~\ref{theorem_main_arithmetic_approximation_theorem} looks similar to the geometric setting outlined in Subsection~\ref{subsection_the_technique_Theorem_A}.
\end{notation}

\subsection{Positivity of $\RR$-Cartier divisors}

\begin{definition}~\label{definition_GVF_pairing}
    Let $\cD$ be an $\RR$-Cartier divisor on $\cX$ and let $v$ be a non-Archimedean valuation on $\kappa(\cX)$. Let $p = \supp(v) \in \cX$ and let $f \in \kappa(\cX)^{\times}_{\RR}$ be a local equation for $\cD$ at $p$. We then write $\GVFpairing(v, \cD)$ for the value $v(f)$ which is the $\RR$-linear (multiplicative) extension of the valuation $v:\kappa(\cX)^{\times} \to \RR$. Moreover, if $\ov{\cD}$ is an $\RR$-divisor of $C^0$-type on $\cX$ we define $\GVFpairing(v, \ov{\cD}) := \GVFpairing(v, \cD)$. On the other hand, if $v$ is an Archimedean valuation on $\kappa(\cX)$ induced by $\sigma:\kappa(\cX) \to \CC$, then we define $\GVFpairing(v, \ov{\cD}) := \frac{1}{2} g(p)$, where $p$ is the point of $\cX(\CC)$ coming from the composition
    % https://q.uiver.app/?q=WzAsMyxbMCwwLCJcXFNwZWMoXFxDQykiXSxbMSwwLCJcXFNwZWMoXFxrYXBwYShcXGNYKSkiXSxbMiwwLCJcXGNYIl0sWzAsMSwiXFxTcGVjKFxcc2lnbWEpIl0sWzEsMl1d
    \[\begin{tikzcd}
    	{\Spec(\CC)} & {\Spec(\kappa(\cX))} & \cX
    	\arrow["{\Spec(\sigma)}", from=1-1, to=1-2]
    	\arrow[from=1-2, to=1-3].
    \end{tikzcd}\]
    %\[ \Spec(\CC) \longrightarrow^{\Spec(\sigma)} \Spec(\kappa(\cX)) \to \cX. \]
    Note $g$ is defined on $p$ since $\Spec(\kappa(\cX))$ is the generic point of $\cX$. Furthermore, since $g$ is conjugation-invariant, the value $g(p)$ does not depend on the choice of $\sigma$ inducing $v$.
\end{definition}

\begin{definition}~\label{definition_Weil_decomposition_of_Cartier_R_divisor}
    Let $\cD$ be an $\RR$-Cartier divisor on $\cX$. The \textit{Weil decomposition of $\cD$} is the codimension one cycle on $\cX$ of the form
    \[ \sum_{\Gamma} \ord_{\Gamma}(\cD) \Gamma, \]
    where the sum is taken over all codimension one integral subschemes of $\cX$ and $\ord_{\Gamma}(\cD) = \GVFpairing(\ord_\Gamma, \cD)$. This makes sense, as we assume that $\cX$ is normal, which implies that the order of vanishing at an integral subscheme $\Gamma$ of codimension one is a valuation on $\kappa(\cX)$.
\end{definition}

\begin{lemma}~\label{lemma_notions_of_effective_for_R_divisors}
    Let $\cD$ be an $\RR$-Cartier divisor on $\cX$. The following are equivalent.
    \begin{enumerate}[label=(\arabic*)]
        \item $\cD$ is a positive $\RR$-combination of effective Cartier divisors.
        \item For all $p \in \cX$ if $f \in \kappa(\cX)_\RR^{\times}$ is a local equation for $\cD$ at $p$, then $f \in (\cO_{\cX, p} \setminus \{0\})_{\RR}$.
        \item For all non-Archimedean valuations $v$ on $\kappa(\cX)$ we have $\GVFpairing(v, \cD) \geq 0$.
        \item If $\cD = \sum_{\Gamma} a_{\Gamma} \Gamma$ is the Weil decomposition of $\cD$, then all $a_{\Gamma}$ are non-negative.
    \end{enumerate}
    Moreover, the same holds (with $\QQ$-tensors instead of $\RR$-tensors) for a $\QQ$-Cartier divisor $\cD$.
\end{lemma}
In the above by $(\cO_{\cX, p} \setminus \{0\})_{\RR}$ we mean the subsemigroup of $\kappa(\cX)_\RR^{\times}$ generated by elements of the form $g^r$ with $g \in \cO_{\cX, p} \setminus \{0\}$ and $r>0$.
\begin{proof}
    Implications $(1) \implies (2) \implies (3) \implies (4)$ are skipped. The implication $(4) \implies (1)$ is \cite[Proposition 2.4.16]{Adelic_curves_1}. Implications in the $\QQ$ case are skipped.
\end{proof}

\begin{definition}
    Let $\cD$ be an $\RR$-Cartier divisor on $\cX$. We say that $\cD$ is \textit{effective}, if it satisfies the equivalent conditions from Lemma~\ref{lemma_notions_of_effective_for_R_divisors}. In this case we write $\cD \geq 0$. Similarly, if $\cE$ is another $\RR$-Cartier divisor on $\cX$, we write $\cD \geq \cE$, if $\cD - \cE \geq 0$. We also define the space rational functions having poles at most given by $\cD$ as
    \[ H^0(\cX, \cD) := \{ f \in \kappa(\cX)^{\times} : \cD + \cdiv(f) \geq 0 \} \cup \{0\}. \]
    A similar definition applies to $\RR$-Cartier divisors on $X=\cX_\QQ$.
\end{definition}

\begin{lemma}~\label{lemma_every_R_divisor_comb_of_effective_and_a_principal}~\cite[Lemma 5.2.3]{Moriwaki_Zariski_arithmetic_decompositions}
    The following hold.
    \begin{enumerate}
        \item Let $\cZ$ be a Weil divisor on $\cX$. Then there is an effective Cartier divisor $\cA$ on $\cX$ such that $\cZ \leq \cA$.
        \item Let $\cD$ be a Cartier divisor on $\cX$. Then there are effective Cartier divisors $\cA$ and $\cB$ on $\cX$ such that $\cD = \cA - \cB$.
        \item Let $x_1, \dots, x_l$ be points of $\cX$ and let $\cD$ be a Cartier divisor on $\cX$. Then there are effective Cartier divisors $\cA$ and $\cB$, and a non-zero rational function $\phi$ on $\cX$ such that $\cD + (\phi) = \cA - \cB$ and $x_1, \dots , x_l \not\in \supp(\cA) \cup \supp(\cB)$.
    \end{enumerate}
\end{lemma}
%\begin{proof}
%    See \cite[Lemma 5.2.3]{Moriwaki_Zariski_arithmetic_decompositions}.
%\end{proof}

\begin{lemma}~\label{lemma_divisors_here_are_the_same_is_in_Ikoma_up_to_a_principal}
    %Let $\ov{\cD}=(\cD,g)$ be an arithmetic $\RR$-divisor of $C^0$-type. It has the following properties:
    %\begin{enumerate}
        %\item If $\cD = \cA - \cB$ and $\cA, \cB$ are effective $\RR$-Cartier divisors, then there exists $\cA$-Green function $a$, and a $\cB$-Green function $b$ (both of $C^0$-type) such that $\ov{\cA} = (\cA,a), \ov{\cB} = (\cB,b)$ are effective and $\ov{\cD} = \ov{\cA} - \ov{\cB}$.
        %\item Let $x_1, \dots, x_l$ be points of $\cX$ and assume that for effective $\RR$-Cartier divisors $\cA$ and $\cB$, and a non-zero $\RR$-rational function $\phi$ on $\cX$, the identity $\cD + (\phi) = \cA - \cB$ holds, and $x_1, \dots , x_l \not\in \supp(\cA) \cup \supp(\cB)$. Then there exists $\cA$-Green function $a$, and a $\cB$-Green function $b$ (both of $C^0$-type) such that $\ov{\cA} = (\cA,a), \ov{\cB} = (\cB,b)$ are effective and $\ov{\cD} + \adiv(\phi) = \ov{\cA} - \ov{\cB}$.
    %\end{enumerate}
    %Note that by Lemma~\ref{lemma_every_R_divisor_comb_of_effective_and_a_principal} decompositions from both points exist.
    Let $\ov{\cD}=(\cD,g)$ be an arithmetic $\RR$-divisor of $C^0$-type on $\cX$. Fix a point $x \in \cX$. Then there exists an arithmetic $\RR$-divisor of $C^0$-type $\ov{\cD}'=(\cD',g')$ such that $\cD' = \sum_{i=1}^m \alpha_i' \cD'_i$ for some real $\alpha_i'$'s and effective Cartier $\cD'_i$'s, and such that $\ov{\cD} - \ov{\cD}'$ is principal. Moreover, $g'$ can be defined on $(\cX \setminus \bigcup_{i=1}^m \supp(\cD_i'))(\CC)$ and one can assume that $x \in \cX \setminus \bigcup_{i=1}^m \supp(\cD_i')$.
\end{lemma}
\begin{proof}
    %$\cV = \bigcup_{i=1}^n (\supp(A_i) \cup \supp(B_i))$
    Let $\cD = \sum_{i=1}^n \alpha_i \cD_i$ for some real $\alpha_i$'s and Cartier (but not necessarily effective) divisors $\cD_i$'s. By Lemma~\ref{lemma_every_R_divisor_comb_of_effective_and_a_principal} there are effective Cartier divisors $\cA_i, \cB_i$'s and rational functions $f_i$'s on $\cX$, such that
    \[ \cD_i = \cA_i - \cB_i + \cdiv(f_i) \textnormal{ for all } i=1, \dots, n \]
    and with $x \not\in \supp(\cA_i) \cup \supp(\cB_i)$. Let $\cD' = \sum_{i=1}^n \alpha_i (\cA_i - \cB_i)$. Let $\cU = \cX \setminus \supp(\cD)$ and let $\cV = \cX \setminus \bigcup_{i=1}^n \supp(\cdiv(f_i))$. Let $g' = g + \sum_{i=1}^n \alpha_i \log|f_i|^2$ on $(\cU \cap \cV)(\CC)$. Pick $P \in (\cU \cap \cV)(\CC)$ and let $\cW(\CC) \subset (\cU \cap \cV)(\CC)$ be a $\ZZ$-Zariski neighbourhood of $P$ with local equations $a_i, b_i$ for $\cA_i, \cB_i$ respectively (for all $i$) such that on $\cW(\CC)$ a continuous extension of
    \[ g(x) + \sum_{i=1}^n \alpha_i( \log|a_i|^2(x) - \log|b_i|^2(x) + \log|f_i|^2(x) ) \]
    can be defined. Then the function
    \[ g'(x) + \sum_{i=1}^n \alpha_i( \log|a_i|^2(x) - \log|b_i|^2(x)) \]
    can also be defined on $\cW(\CC)$, which proves that $(\cD',g')$ is an arithmetic $\RR$-divisor of $C^0$-type. By Remark~\ref{remark_open_set_of_definition_of_g_contains_support_of_divisors}  $g'$ can be defined on $\bigl(\cX \setminus \bigcup_{i=1}^n (\supp(\cA_i) \cup \supp(\cB_i))\bigr)(\CC)$. Moreover, by construction $x \not\in \bigcup_{i=1}^n (\supp(\cA_i) \cup \supp(\cB_i))$, which finishes the proof.
\end{proof}

\subsection{Heights with respect to arithmetic $\RR$-divisors}\label{subsection_heights_w_r_t_arithmetic_R_divisors}

\begin{definition}
    Let $x \in X = \cX_\QQ$ be a closed point and let $v$ be a valuation on $\kappa(x)$. Let $\ov{\cD} = (\cD, g)$ be an arithmetic $\RR$-divisor of $C^0$-type on $\cX$ with $\cD$ effective and $x \not\in \supp(\cD)$. We define the \textit{local degree} of $\ov{\cD}$ with respect to $v$ $\adeg_v(\ov{\cD}|\ov{\{x\}})$ in the following way:
    \begin{itemize}
        \item If $v$ is Archimedean we put $\adeg_v(\ov{\cD}|\ov{\{x\}}) := \frac{1}{2} g(x^{\sigma})$ where $\sigma:\kappa(x) \to \CC$ is any embedding which induces $v$ on $\kappa(x)$.
        \item If $v$ is non-Archimedean we put $\adeg_v(\ov{\cD}|\ov{\{x\}}) = v(d|_x)$ where $d$ is a local equation for $\cD$ at the point $q \in \cX$ defined as the image of the maximal ideal of the value ring $\cO_v \subset \kappa(x)$ via the map $\phi$ in the following diagram
        % https://q.uiver.app/?q=WzAsNCxbMCwwLCJcXFNwZWMoXFxrYXBwYSh4KSkiXSxbMCwxLCJcXFNwZWMoXFxjT192KSJdLFsxLDAsIlxcY1giXSxbMSwxLCJcXFNwZWMoXFxaWikiXSxbMiwzXSxbMCwxXSxbMCwyXSxbMSwzXSxbMSwyLCJcXHBoaSIsMSx7InN0eWxlIjp7ImJvZHkiOnsibmFtZSI6ImRhc2hlZCJ9fX1dXQ==
        \[\begin{tikzcd}
        	{\Spec(\kappa(x))} & \cX \\
        	{\Spec(\cO_v)} & {\Spec(\ZZ)}
        	\arrow[from=1-2, to=2-2]
        	\arrow[from=1-1, to=2-1]
        	\arrow[from=1-1, to=1-2]
        	\arrow[from=2-1, to=2-2]
        	\arrow["\phi"{description}, dashed, from=2-1, to=1-2].
        \end{tikzcd}\]
        Existence and uniqueness of $\phi$ follows from the valuative criterion of properness. If $d$ is an $\RR$-rational function, we use the unique (multiplicative) $\RR$-linear extension of $v$ to calculate $v(d|_x)$. Note that $q$ generalises to $x$, so $d$ is a regular function on the neighbourhood of $x$ and $d|_x \in \kappa(x)^{\times}_{\RR}$.
        % If $v$ is non-Archimedean we put $\adeg_v(\ov{\cD}|\ov{\{x\}}) = v(d|_x)$ where $d$ is a local equation for $\cD$ at the point $x$ and $d|_x$ is the corresponding element of $\kappa(x)^{\times}_{\RR}$. If $d$ is an $\RR$-rational function, we use the unique (multiplicative) $\RR$-linear extension of $v$ to calculate the result.
    \end{itemize}
\end{definition}

\begin{definition}~\label{definition_height_of_a_point_w_r_t_arithmetic_R_divisor}
    Let $\ov{\cD} = (\cD, g)$ be an arithmetic $\ZZ$-divisor of $C^0$-type on $\cX$ with $\cD$ effective and let $x \in X = \cX_\QQ$ be a closed point of the generic fiber of $\cX$. If $x \not\in \supp(\cD)$ and $\cD$ is a Cartier divisor on $\cX$ the \textit{height} of $x$ with respect to $\ov{\cD}$ is defined to be
    \[ h_{\ov{\cD}}(x) := \frac{1}{[\kappa(x):\QQ]} \Bigl(\log \# \bigl(\cO_{\cC}(\cD)/\cO_{\cC}\bigl) + \frac{1}{2} \sum_{\sigma:\kappa(x) \to \CC} g(x^{\sigma})\Bigl), \]
    where $\cC = \ov{ \{x\} } \subset \cX$ is the $1$-cycle corresponding to $x$ and we treat $\cO_{\cC}(\cD)/\cO_{\cC}$ as a module over the ring of functions of $\cC \cap \cD$. For a general $\ov{\cD}$, by Lemma~\ref{lemma_divisors_here_are_the_same_is_in_Ikoma_up_to_a_principal} there is a decomposition $\ov{\cD} = \sum_{i=1}^m \alpha_i \ov{\cD}_i + \adiv(f)$ with $\alpha_i$ real, $\cD_i$ effective Cartier, such that $x \not\in \supp(\cD_i)$ for all $i=1, \dots, m$ and where $f$ is an $\RR$-rational function on $\cX$. Then we put
    \[ h_{\ov{\cD}}(x) := \sum_{i=1}^m \alpha_i h_{\ov{\cD}_i}(x). \]
    Note that this is well defined by the product formula on number fields. For details, see \cite[Section 4.2]{Moriwaki_Adelic_divisors_on_Arithmetic_Varieties}. Also, we define the \textit{arithmetic degree} of $\ov{\cD}$ with respect to $x$ as the quantity $\adeg(\ov{\cD}|\ov{\{x\}})$ satisfying
    \[ h_{\ov{\cD}}(x) = \frac{\adeg(\ov{\cD}|\ov{\{x\}})}{\deg(x)}. \]
\end{definition}

\begin{lemma}~\label{lemma_local_and_global_arithmetic_degree}
    Let $\ov{\cD}$ be an arithmetic $\RR$-divisor of $C^0$-type on $\cX$ and let $x \in X = \cX_\QQ$ be a closed point of the generic fiber of $\cX$ with $x \not\in \supp(\cD)$. Let $K = \kappa(x)$ and denote by $\cO_K$ its integer ring. Then
    \[ h_{\ov{\cD}}(x) = \int_{\Val_{K}} \adeg_v(\ov{\cD}|\ov{\{x\}}) d\mu(v), \]
    where $\mu$ is the measure
    \[ \mu = \frac{1}{[K:\QQ]} \biggl( \sum_{p \in \Spec(\cO_K)} \delta_{\ord_p} \cdot \log \# \kappa(p) + \sum_{\sigma : K \to \CC} \delta_{-\log |\sigma(-)|} \biggr) \]
    and $\Val_K$ is the space of valuations on $K$ (see Example~\ref{example_number_field_GVF}).
\end{lemma}
\begin{proof}
    By Lemma~\ref{lemma_divisors_here_are_the_same_is_in_Ikoma_up_to_a_principal} we can write $\ov{\cD} = \sum_{i=1}^m \alpha_i \ov{\cD}_i + \adiv(f)$ with $\alpha_i$ real, $\cD_i$ effective Cartier, such that $x \not\in \supp(\cD_i)$ for all $i=1, \dots, m$ and where $f$ is an $\RR$-rational function on $\cX$. Note that
    \[ \supp(\cdiv(f)) \subset \supp(\cD) \cup \bigcup_{i=1}^m \supp(\cD_i) \]
    and $x$ is not in the right hand side, so $x \not\in \supp(\adiv(f))$. Since both the height and the integral are linear in $\ov{\cD}$ and give zero on $\adiv(f)$ (see the product formula in Example~\ref{example_number_field_GVF}), it is enough to check that the lemma holds for an arithmetic $\ZZ$-divisor of $C^0$-type $\ov{\cD}$ with $\cD$ effective. This is standard, but for convenience we provide a proof here.

    By unraveling the definitions, we need to prove that
    \[ \log \# \bigl(\cO_{\cC}(\cD)/\cO_{\cC}\bigl) = \sum_{p \in \Spec(\cO_K)} \adeg_{\ord_p}(\ov{\cD}|\ov{\{x\}}) \cdot \log \# \kappa(p). \]
    Let $\cC$ be the closure of $x$ in $\cX$ (with the reduced scheme structure). By valuative criterion of properness (not directly, one has to use it for primes of $\cO_K$), we get a diagonal map $\phi$ on the commutative diagram
    % https://q.uiver.app/?q=WzAsNCxbMCwwLCJcXFNwZWMoXFxrYXBwYSh4KSkiXSxbMCwxLCJcXFNwZWMoXFxjT192KSJdLFsxLDAsIlxcY1giXSxbMSwxLCJcXFNwZWMoXFxaWikiXSxbMiwzXSxbMCwxXSxbMCwyXSxbMSwzXSxbMSwyLCJcXHBoaSIsMSx7InN0eWxlIjp7ImJvZHkiOnsibmFtZSI6ImRhc2hlZCJ9fX1dXQ==
    \[\begin{tikzcd}
        {\Spec(K)} & \cX \\
        {\Spec(\cO_K)} & {\Spec(\ZZ)}
        \arrow[from=1-2, to=2-2]
        \arrow[from=1-1, to=2-1]
        \arrow["x", from=1-1, to=1-2]
        \arrow[from=2-1, to=2-2]
        \arrow["\phi"{description}, dashed, from=2-1, to=1-2].
    \end{tikzcd}\]
    Since $x$ is the generic point of $\cC$, $\phi$ factors through $\cC$. As $\cC \to \Spec(\ZZ)$ is projective and has finite fibers (as the composition with $\Spec(\cO_K) \to \cC$ has finite fibers), by \cite[Theorem 19.1.7]{Ravi_Vakil_FOAG} we get that $\cC$ is affine (and finite) over $\Spec(\ZZ)$. Let $\cC = \Spec (A)$. Note that then $\ZZ \subset A \subset \cO_K \subset K$ and $\cO_K$ is the integral closure of $A$. Let $I, J$ be the ideals of $\cD$ in $A, \cO_K$ respectively (more precisely $J$ is the ideal of the pullback of $\cD$ to $\Spec(\cO_K)$). Consider the diagram coming from the snake lemma
    \[\xymatrix{
      & 0 \ar[r]\ar[d] & 0 \ar[r]\ar[d] & \ker h \ar[d]\\
      0 \ar[r] & I \ar[r]\ar[d]_(.6)f|(.3)\hole
          & A \ar[r]\ar[d]_(.6)g|(.3)\hole
          & A/I \ar[r]\ar[d]_(.6)h|(.3)\hole="hole" & 0 \\
      0\ar[r] & J \ar[r]\ar[d] & \cO_K \ar[r]\ar[d] & \cO_K/J \ar[d]\ar[r] & 0 \\
      & \cok f \ar[r]
          \ar@{<-} `l [lu]-(.5,0) `"hole" `[rrruu]+(.5,0) `[rruuu] [rruuu]
          & \cok g \ar[r] & \cok h \ar[r] & 0. \\}\]
    Note that $\log\# \bigl(\cO_{\cC}(\cD)/\cO_{\cC}\bigl) = \log\# (A/I)$. We now show that $\# (A/I) = \# (\cO_K/J)$. If $q \in \Spec(A)$, then we can localise everything at $q$ and as $\cD$ is an effective Cartier divisor, it means that in $A_q$ the ideal $I_q$ is generated by a single element $a \in A_q$. Also $J_q$ is generated by $a$ in $(\cO_K)_q$. Thus $(\cok f)_q \simeq (\cok g)_q$. We calculate
    \[ \log \#(\cok f) = \sum_{q \in \Spec(A)} \log \#(\cok f)_q = \sum_{q \in \Spec(A)} \log \#(\cok g)_q = \log \#(\cok g). \]
    Using two exact sequences
    \[\xymatrix{ 0 \ar[r] & \ker h \ar[r] & \cok f \ar[r] & \cok g \ar[r] & \cok h \ar[r] & 0 }\]
    \[\xymatrix{ 0 \ar[r] & \ker h \ar[r] & A/I \ar[r] & \cO_K/J \ar[r] & \cok h \ar[r] & 0 }\]
    we get that 
    \[ \log \#(\ker h) - \log \#(\cok f) + \log \#(\cok g) - \log \#(\cok h) = 0, \]
    \[ \log \#(\ker h) - \log \#(A/I) + \log \#(\cO_K/J) - \log \#(\cok h) = 0. \]
    Putting it all together we get that indeed $\log\# (A/I) = \log\# (\cO_K/J)$. 
    We calculate
    \[ \log\# (\cO_K/J) = \sum_{p \in \Spec(\cO_K)} \log\# (\cO_K/J)_p = \sum_{p \in \Spec(\cO_K)} \ord_p(J) \log\# \kappa(p), \]
    where by $\ord_p(J)$ we mean the value of the valuation $\ord_p$ at a local equation for $J$ at $p$. By the definition $\adeg_{\ord_p}(\ov{\cD}|\ov{\{x\}}) = \ord_p(J)$, which finishes the proof.
\end{proof}

\begin{lemma}~\label{lemma_height_with_respect_to_divisor_max_div_two_and_zero}
    Let $x \in X = \cX \otimes \QQ$ be a closed point and let $\ov{\cD}=(\cdiv(2),0)$. Then
    \[ h_{\ov{\cD}}(x) = \log(2). \]
\end{lemma}
\begin{proof}
    Note that $\ov{\cD} = \adiv(2) + (0, 2 \log(2))$. Hence we get
    \[ h_{\ov{\cD}}(x) = h_{(0, 2 \log(2))}(x) = \frac{1}{2 [\kappa(x):\QQ]} \sum_{\sigma:\kappa(x) \to \CC} 2 \log(2) = \log(2). \]
\end{proof}

\begin{lemma}~\label{lemma_behaviour_of_height_for_birational_maps_and_pullbacks_proj_formula}
    Assume that $\cX'$ is a generically smooth, normal arithmetic variety. Let $f:\cX' \to \cX$ be a generically finite map and let $x' \in X' = \cX' \otimes \QQ$ be a closed point. Then for $x = f(x') \in X = \cX \otimes \QQ$ and for any arithmetic $\RR$-divisor of $C^0$-type $\ov{\cD}$ on $\cX$ we have
    \[ h_{f^* \ov{\cD}}(x') = h_{\ov{\cD}}(x). \]
\end{lemma}
\begin{proof}
    As in Definition~\ref{definition_height_of_a_point_w_r_t_arithmetic_R_divisor}, we reduce to the case that $\ov{\cD}$ is an arithmetic $\ZZ$-divisor of $C^0$-type with $\cD$ effective and $x \not\in \supp(\cD)$. Let $\cC = \ov{\{x\}}$ and $\cC' = \ov{\{x'\}}$. Then $\cC' \to \cC$ is a finite morphism of degree $\deg(x'/x) := \deg(x')/\deg(x)$. Both sides of the equality depend only on the restriction of $\ov{\cD}$ to $\cC$, so we can without loss of generality assume that $\cX = \cC, \cX' = \cC'$. In this case $\ov{\cD}$ is a $\ZZ$-divisor of $C^{\infty}$-type. Let $\ov{\cL} = \cO(\ov{\cD})$ be the corresponding metrized line bundle. By Theorem~\ref{theorem_arithmetic_projection_formula} we get
    \[ \adeg(f^*\ov{\cL}|\cC') = \adeg(\ov{\cL}|f_* \cC') = \frac{\deg(x')}{\deg(x)} \cdot \adeg(\ov{\cL}|\cC). \]
    Thus
    \[ h_{f^* \ov{\cD}}(x') = \frac{\adeg(f^*\ov{\cL}|\cC')}{\deg(x')} = \frac{\adeg(\ov{\cL}|\cC)}{\deg(x)} = h_{\ov{\cD}}(x). \]
\end{proof}

\subsection{Positivity of arithmetic $\RR$-divisors}

\begin{definition}~\label{definition_properties_of_arithmetic_R_divisors}
    We use the following notions regarding an arithmetic $\RR$-divisor of $C^0$-type $\ov{\cD}=(\cD,g)$ on $\cX$.
    \begin{enumerate}
        \item We write $\ov{\cD} \geq 0$ if $\cD \geq 0$ and $g \geq 0$. Note that this does not depend on the choice of $g$ since if $g'$ extends $g$ and $g'(x) < 0$ for some $x \in \cX(\CC)$, then the set $\{y \in \cX(\CC) : g'(y) <0\}$ is open, so it has to intersect the domain of $g$ which is a dense open subset of $\cX(\CC)$. We then call $\ov{\cD}$ \textit{effective}. If $\ov{\cE}$ is another arithmetic $\RR$-divisor of $C^0$-type on $\cX$, then $\ov{\cD} \geq \ov{\cE}$ means that $\ov{\cD} - \ov{\cE} \geq 0$.
        \item There is a natural norm $|\cdot|^{\ov{\cD}}$ associating to each rational function $f \in  H^0(\cX, \cD)$ a function
        \[ |f|^{\ov{\cD}} : \cX(\CC) \to \RR, \]
        which is the unique continuous extension of $|f| \cdot \exp(-g/2)$ on $\cX(\CC)$. For well definiteness, see \cite[Proposition 1.4.2]{Moriwaki_Adelic_divisors_on_Arithmetic_Varieties}. It induces a supremum norm
        \[ \|\cdot\|_{\sup}^{\ov{\cD}} : H^0(\cX, \cD) \to \RR. \]
        A rational function $f \in H^0(\cX, \cD)$ is called \textit{small}, if $\|f\|_{\sup}^{\ov{\cD}} \leq 1$. It is called \textit{strictly small} if $\|f\|_{\sup}^{\ov{\cD}}<1$. The set of small sections is denoted by $\aH^0(\cX, \ov{\cD})$.
        \item Assume that $\ov{\cD}$ is an arithmetic $\ZZ$-divisor of $C^{\infty}$-type. Let $L$ be the (complex analytic) line bundle over $\cX(\CC)$ induced by the base change of $\cO_{X}(\cD_{\QQ})$ to $\CC$ (recall that $X = \cX_\QQ$). Then for a Zariski open $\cU \subset \cX$ where $d \in \kappa(\cX)^{\times}$ is a local equation for $\cD$ and a section $s$ of $L$ over $\cU(\CC)$ one can define a norm of $s$ by presenting $s = f \cdot d^{-1}$ for some holomorphic function $f:\cU(\CC) \to \CC$ and by putting
        \[ |s|_{\ov{L}} := |f| \cdot \exp(-(g+\log|d|^2)/2), \]
        where by $g+\log|d|^2$ we mean the unique continuous (even smooth) extension of this function to $\cU(\CC)$. In fact, this metric comes from a unique hermitian product on $L$ and the Chern form of the pair $(L, |\cdot|_{\ov{L}})$ is denoted by $c_1(\ov{\cD})$ (see Definition~\ref{definition_hermitian_line_bundle}). Note that $c_1(\ov{\cD})|_{\cU(\CC)} = \frac{i}{2\pi}\partial\ov{\partial} (g+\log|d|^2)$.
        \item The \textit{arithmetic volume} of $\ov{\cD}$ is defined to be the following number
        \[ \avol(\ov{\cD}) := \limsup_{m \to \infty} \frac{\log \# \aH^0(\cX, m \ov{\cD})}{m^{d+1}/(d+1)!}. \]
        The arithmetic volume only depends on the class of $\ov{\cD}$ in $\aN^1(\cX)$. Also, recall that \textit{volume} of an $\RR$-divisor $\cD_{\QQ}$ on the generic fiber $X = \cX_\QQ$ is
        \[ \vol(\cD_\QQ) := \limsup_{m \to \infty} \frac{\dim_{\QQ} H^0(X, m \cD_\QQ)}{m^d/d!}. \]
        %\item If $x \in X = \cX_\QQ$ is a closed point in the generic fiber of $\cX$, one can define a \textit{height} of $x$ with respect to $\ov{\cD}$. We denote it by $h_{\ov{\cD}}(x)$.
    \end{enumerate}
\end{definition}

\begin{remark}~\label{remark_alternative_definition_of_effectivity_of_arithmetic_R_divisors}
    Note that for an arithmetic $\RR$-divisor of $C^0$-type $\ov{\cD}$ on $\cX$, it is effective if and only if for every valuation $v$ on $\kappa(\cX)$ (non-Archimedean or Archimedean) we have $\GVFpairing(v, \ov{\cD}) \geq 0$.
\end{remark}

\begin{notation}
    If $\ov{\cD}$ is an arithmetic $\ZZ$-divisor of $C^{\infty}$-type on $\cX$, we denote by $\cO(\ov{\cD})$ the pair $(\cO_\cX(\cD), |\cdot|)$, where $|\cdot|$ is the smooth hermitian metric of real type described in Definition~\ref{definition_properties_of_arithmetic_R_divisors}. Note that in this case $\aH^0(\cX, \ov{\cD}) = \aH^0(\cX, \cO(\ov{\cD}))$ (both definitions of small sections agree).
\end{notation}

\begin{remark}~\label{remark_passing_from_hermitian_line_bundle_to_an_arithmeic_divisor}
    If $\ov{\cL}$ is a hermitian line bundle on $\cX$, then a rational section $s$ of $\cL$ defines an arithmetic $\ZZ$-divisor of $C^{\infty}$-type on $\cX$. Namely, take the pair $\adiv(s) = (\cdiv(s), -\log|s|_{\ov{\cL}}^2)$. Then $\cO(\adiv(s)) \simeq \ov{\cL}$ as hermitian line bundles. Moreover, if $\ov{\cL} = \cO(\ov{\cD})$ and $s$ is the rational section corresponding to the rational function $1$, then $\adiv(s) = \ov{\cD} = (\cD, g)$. To see that $-\log|s|_{\ov{\cL}}^2 = g$, take an open subset $\cU$ where $d$ is a local equation for $\cD$ and where $d \in \cO_\cX(\cU)^{\times}$. Then $s = d \cdot d^{-1}$, so by the definition we have
    \[ |s|_{\ov{\cL}} = |d| \cdot \exp(-(g+\log|d|^2)/2) = \exp(-g/2). \]
\end{remark}

\begin{remark}
    Let $\ov{\cL} = \cO(\ov{\cD})$ for an arithmetic $\ZZ$-divisor of $C^{\infty}$-type $\ov{\cD}$ on $\cX$. Then for a closed point in the generic fiber $x \in \cX_{\QQ} \subset \cX$, Definition~\ref{definition_height_of_a_point_w_r_t_arithmetic_R_divisor} and the arithmetic degree map from Theorem~\ref{theorem_intersection_product_of_hermitian_line_bundles} agree, i.e.,
    \[ \adeg(\ov{\cD}|\ov{\{x\}}) = \adeg(\ov{\cL}|\ov{\{x\}}). \]
    %This follows straightforward from the inductive definition in Theorem~\ref{theorem_intersection_product_of_hermitian_line_bundles}.
\end{remark}

\begin{definition}~\label{definition_properties_of_arithmetic_R_divisors_big_etc}
    We gather here properties of arithmetic $\RR$-divisors of $C^0$-type on $\cX$ crucial for the proof of Theorem~\ref{theorem_main_arithmetic_approximation_theorem}.
    \begin{enumerate}
        \item An arithmetic $\RR$-divisors of $C^0$-type $\ov{\cD}$ on $\cX$ is \textit{big}, if $\avol(\ov{\cD}) > 0$.
        \item An arithmetic $\RR$-divisors of $C^0$-type $\ov{\cE}$ on $\cX$ is \textit{pseudo-effective}, if for all big arithmetic $\RR$-divisors of $C^0$-type $\ov{\cB}$ on $\cX$, the sum $\ov{\cB} + \ov{\cE}$ is big.
        \item We say that $\ov{\cD}=(\cD,g) \in \Adiv_{\RR}(\cX)$ is \textit{nef}, if the following conditions are satisfied:
        \begin{itemize}
            \item $\cD$ is relatively nef,
            \item $\ov{\cD}$ is an arithmetic $\RR$-divisor of $(C^0 \cap \PSH)$-type,
            \item $h_{\ov{\cD}}(x) \geq 0$ for all closed points $x \in X = \cX_\QQ$.
        \end{itemize}
        \item $\ov{\cD} \in \Adiv_{\RR}(\cX)$ is called \textit{integrable}, if it is in a real vector space generated by nef arithmetic $\RR$-divisors of $C^0$-type. The space of integrable $\RR$-divisors of $C^0$-type on $\cX$ is denoted by $\Idiv_{\RR}(\cX)$.
        \item $\ov{\cD} \in \Adiv_{\RR}(\cX)$ is called \textit{ample}, if there exists a decomposition $\ov{\cD} = \sum_{i=1}^m \alpha_i \ov{\cA}_i$ with $\alpha_i > 0$ for $i=1, \dots, m$ and where $\ov{\cA}_i$ are arithmetic $\ZZ$-divisors of $C^{\infty}$-type such that each $\cO(\ov{\cA}_i)$ is an ample hermitian line bundle.
        %\begin{itemize}
            %\item $\cA$ (more precisely $\cO_{\cX}(\cA)$) is relatively ample with respect to $\cX \to \Spec(\ZZ)$;
            %\item the curvature form $c_1(\ov{\cA})$ is positive pointwise on $\cX(\CC)$;
            %\item $H^0(\cX, n\cA)$ is generated (as a $\ZZ$-module) by strictly small sections of $n\ov{\cA}$, for $n$ big enough.
        %\end{itemize}
    \end{enumerate}
\end{definition}

\begin{lemma}~\label{lemma_effective_are_pseudo_effective}
    Let $\ov{\cD}$ be an arithmetic $\RR$-divisor of $C^0$-type on $\cX$. If $\ov{\cD}$ is big, then there is a natural number $n \neq 0$ and  a rational function $f \in \kappa(\cX)^{\times}$ such that $\adiv(f) + n \ov{\cD}$ is effective on $\cX$. Moreover, $\cD_{\QQ}$ is big. Also, if $\ov{\cD}$ is effective on $\cX$, then it is pseudo-effective on $\cX$.
\end{lemma}
\begin{proof}
    We skip the proof of the first part. For the second part let $\ov{\cA}$ be an ample $\RR$-divisor of $C^{\infty}$-type on $\cX$ and using the continuity of arithmetic volume (see Theorem~\ref{theorem_differentiability_of_arithmetic_volume}) pick a natural number $n$ with $\avol(\ov{\cD} - \frac{1}{n} \ov{\cA}) > 0$. By the first part of the lemma, there is a $\QQ$-rational function $f$ on $\cX$ with $\adiv(f) + \ov{\cD} - \frac{1}{n} \ov{\cA} \geq 0$. Thus $\cdiv(f) + \cD_{\QQ} - \frac{1}{n} \cA_{\QQ} \geq 0$, so $\vol(\cD_{\QQ}) \geq \vol(\frac{1}{n} \cA_{\QQ}) > 0$. For the last part, let $\ov{\cB}$ be a big arithmetic $\RR$-divisor of $C^0$-type on $\cX$. Since $\ov{\cD}$ is effective, the rational function $1$ is small as an element of  $H^0(\cX, \cD)$. Thus any small $s \in \aH^0(\cX, m\ov{\cB})$ is also a small element of $\aH^0(\cX, m(\ov{\cB}+\ov{\cD}))$. Hence $\avol(\ov{\cB}+\ov{\cD}) \geq \avol(\ov{\cB}) > 0$.
\end{proof}

\begin{proposition}~\label{proposition_Moriwaki_Proposition_2_4_2}~\cite[Proposition 2.4.2]{Moriwaki_Zariski_arithmetic_decompositions}
    Let $g$ be a $\cD$-Green function of $\cJ$-type (for $\cJ$ equal $C^0, \PSH \cap C^0$ or $C^{\infty}$) on $\cX$ and let
    \[ \cD = b_1 \cE_1 + \dots + b_r \cE_r \]
    be a decomposition such that $\cE_1, \dots, \cE_r \in \Div(\cX)$ and $b_1, \dots , b_r \in \RR$. Note that
    $\cE_i$ is not necessarily an irreducible Weil divisor. Then we have the following:
    \begin{enumerate}
        \item There are $g_1, \dots , g_r$ such that $g_i$ is an $\cE_i$-Green function of $\cJ$-type for each $i$ and $g = b_1 g_1 + \dots + b_r g_r$.
        \item If $\cE_1, \dots , \cE_r$ are effective, $b_1, \dots , b_r \geq 0, g \geq 0$, then there are $g_1, \dots , g_r$ such that $g_i$ is a non-negative $\cE_i$-Green function of $\cJ$-type for each $i$ and $g = b_1 g_1 + \dots + b_r g_r$.
    \end{enumerate}
\end{proposition}
%\begin{proof}
%    See \cite[Proposition 2.4.2]{Moriwaki_Zariski_arithmetic_decompositions}.
%\end{proof}

\begin{lemma}~\label{lemma_every_Cartier_divisor_can_be_equipped_with_a_Green_function}
    Let $\cD$ be an $\RR$-Cartier divisor on $\cX$. Then there exists a $\cD$-Green function $g$ on $\cX$. %Moreover, if $\cD$ is effective, one can find a blowup $f:\cX' \to \cX$ and a $f^*\cD$-Green function $g'$ on $\cX'$ so that $(f^*\cD,g')$ is effective.
\end{lemma}
\begin{proof}
    % write later why l \cH + \cD is ample for l >>0!
    %We start with a proof of the first part. 
    It is enough to prove the lemma assuming that $\cD$ is a Cartier divisor. In this case we can use an ample Cartier divisor $\cH$ on $\cX$ and take a natural number $l$ such that $l \cH + \cD$ is ample. Since $\cD = (l \cH + \cD) - l \cH$, we can without loss of generality assume that $\cD$ is ample. Then the corresponding line bundle $\cO(\cD)$ has a hermitian metric, for example by \cite[Corollary 2.11]{Charles2021ArithmeticAA}. By taking the rational section of this bundle corresponding to the rational function $1$, we are done (see Remark~\ref{remark_passing_from_hermitian_line_bundle_to_an_arithmeic_divisor}).
\end{proof}

\begin{lemma}~\label{lemma_decomposition_of_an_arithmetic_divisor_into_a_difference_of_amples_plus_a_function}
    Let $\ov{\cD}$ be an arithmetic $\RR$-divisor of $C^0$-type on $\cX$. Fix $\varepsilon > 0$. Then there exist:
    \begin{itemize}
        \item ample arithmetic $\RR$-divisors of $C^{\infty}$-type $\ov{\cA}, \ov{\cA}'$;
        \item a $C^0$-function $\phi$ on $\cX(\CC)$ with $\| \phi \|_{\sup} < \varepsilon$;
    \end{itemize}
    such that $\ov{\cD} \equiv \ov{\cA} - \ov{\cA}' - (0, \phi)$ (equality modulo $\RR$-rational equivalence).
\end{lemma}
\begin{proof}
    % This is possible because every divisor has a metric
    By Proposition~\ref{proposition_Moriwaki_Proposition_2_4_2} we can write $\ov{\cD} = \sum_i a_i \ov{\cD}_i$ for some arithmetic $\ZZ$-divisors of $C^0$-type $\ov{\cD}_i$ and $a_i \in \RR$. Let $\phi_i$ be a $C^0$-function on $\cX(\CC)$ such that $\ov{\cD}_i + (0, \phi_i)$ are of $C^{\infty}$-type and with $\| \phi_i \|_{\sup} < \delta$. By \cite[Corollary 2.6]{Charles2021ArithmeticAA} we can present $\cO(\ov{\cD}_i + (0, \phi_i))$ as a difference of two ample hermitian line bundles. By taking rational sections of those hermitian line bundles, we can write $\ov{\cD}_i + (0, \phi_i)$ as a difference of ample arithmetic $\ZZ$-divisors of $C^{\infty}$-type, at least up to a principal divisor. Let
    \[ \ov{\cD}_i + (0, \phi_i) \equiv \ov{\cA}_i - \ov{\cA}'_i \]
    be such a decomposition. Then we get
    \[ \ov{\cD} \equiv \sum_i a_i(\ov{\cA}_i - \ov{\cA}'_i) - (0, \sum_i a_i \phi_i). \]
    Note that $\| \sum_i a_i \phi_i \|_{\sup} \leq \sum_i |a_i| \| \phi_i \|_{\sup} < (\sum_i |a_i|) \delta$. By taking $\delta$ small enough and rearranging the terms we get the result (note that a sum of ample is ample).
\end{proof}

\subsection{Intersection product and arithmetic volume}

We use the intersection product of arithmetic divisors defined in \cite{Moriwaki_Zariski_arithmetic_decompositions} and extended by \cite{Ikoma_Concavity_of_arithmetic_volume}. In our notation an arbitrary arithmetic $\RR$-divisor of $C^0$-type appears on the right in contrary to \cite{Ikoma_Concavity_of_arithmetic_volume}. More precisely, we have the following.
\begin{theorem}~\label{theorem_Ikoma_extension_intersection_product_arithmetic_R_divisors}
    There exists a symmetric multi-linear map
    \[ \adeg : \Idiv_{\RR}(\cX)^{\times (d+1)} \to \RR \]
    extending the intersection product of arithmetic $\ZZ$-divisors of $C^{\infty}$-type from Theorem~\ref{theorem_intersection_product_of_hermitian_line_bundles}. Moreover, one can extend it even further, to a multi-linear function
    \[ \adeg : \Idiv_{\RR}(\cX)^{\times d} \times \Adiv_{\RR}(\cX) \to \RR. \]
    This extension is the unique one having the following property: If $\ov{\cD}_0, \dots, \ov{\cD}_{d}$ are nef and $\ov{\cD}_{d+1}$ is pseudo-effective, then $\adeg(\ov{\cD}_0 \cdots \ov{\cD}_{d} \cdot \ov{\cD}_{d+1}) \geq 0$. Also, it only depends on the classes of the divisors in $\aN^1(\cX)$.
\end{theorem}
\begin{proof}
    This is \cite[Lemma 2.5]{Ikoma_Concavity_of_arithmetic_volume}. The fact that the product only depends on the classes of divisors in $\aN^1(\cX)$ follows for example from \cite[Theorem 5.2.2 and Claim 6.4.2.2.(c)]{Moriwaki_Zariski_arithmetic_decompositions}.
\end{proof}

\begin{notation}
    We usually omit $\adeg$ in the notation, and just write $\ov{\cD}_0 \cdot \ldots \cdot \ov{\cD}_{d+1}$ for $\adeg(\ov{\cD}_0 \cdots \ov{\cD}_{d} \cdot \ov{\cD}_{d+1})$ where $\ov{\cD}_0, \dots, \ov{\cD}_{d}$ are integrable arithmetic $\RR$-divisors of $C^0$-type and $\ov{\cD}_{d+1}$ is an arithmetic $\RR$-divisor of $C^0$-type.
\end{notation}

\begin{lemma}~\label{lemma_intersection_doesnt_depend_on_the_model}~\cite[Lemma 2.3 and Lemma 2.5]{Ikoma_Concavity_of_arithmetic_volume}
    Let $\pi : \cX' \to \cX$ be a birational morphism from a generically smooth
    normal arithmetic variety. Then
    \[ \ov{\cD}_0 \cdot \ldots \cdot \ov{\cD}_{d+1} = \pi^* \ov{\cD}_0 \cdot \ldots \cdot \pi^* \ov{\cD}_{d+1}, \]
    for all arithmetic $\RR$-divisors of $C^0$-type $\ov{\cD}_0, \dots, \ov{\cD}_{d+1}$ on $\cX$ with $\ov{\cD}_0, \dots, \ov{\cD}_{d}$ integrable.
\end{lemma}
%\begin{proof}
%    This follows from \cite[Lemma 2.3]{Ikoma_Concavity_of_arithmetic_volume}.
%\end{proof}

\begin{theorem}~\label{theorem_Ikoma_characterisation_of_psef_arithmetic_R_divisors}~\cite[Theorem 6.4]{Ikoma_Concavity_of_arithmetic_volume}
    Let $\ov{\cD}$ be an arithmetic $\RR$-divisor of $C^0$-type on $\cX$.
    \begin{enumerate}
        \item The following are equivalent.
        \begin{enumerate}
            \item $\ov{\cD}$ is pseudo-effective.
            \item For any normalized blowup $\phi:\cX' \to \cX$ and for any nef arithmetic $\RR$-divisor of $C^0$-type $\ov{\cH}$ on $\cX'$, we have
            \[ \adeg(\ov{\cH}^d \cdot \phi^* \ov{\cD}) \geq 0. \]
            \item For any blowup $\phi:\cX' \to \cX$ such that $\cX'$ is normal, generically smooth, and for any ample arithmetic $\QQ$-divisor of $C^{\infty}$-type $\ov{\cH}$ on $\cX'$, we have
            \[ \adeg(\ov{\cH}^d \cdot \phi^* \ov{\cD}) \geq 0. \]
        \end{enumerate}
        \item Suppose that $\ov{\cD}$ is pseudo-effective. The following are equivalent.
        \begin{enumerate}
            \item $\ov{\cD}$ is principal.
            \item There exist a blowup $\phi:\cX' \to \cX$ with normal, generically smooth $\cX'$, and an ample arithmetic $\RR$-divisor of $C^{\infty}$-type $\ov{\cH}$ on $\cX'$, such that
            \[ \adeg(\ov{\cH}^d \cdot \phi^* \ov{\cD}) = 0. \]
        \end{enumerate}
    \end{enumerate}
\end{theorem}

\begin{lemma}~\label{lemma_pseudo_effective_cone_is_invariant_under_birational_pullbacks}
    Let $f:\cX' \to \cX$ be a blowup and $\ov{\cD}$ be an arithmetic $\RR$-divisor of $C^0$-type on $\cX$. Then $\avol(\ov{\cD}) = \avol(f^* \ov{\cD})$.
    Moreover, $\ov{\cD}$ is pseudo-effective [effective] on $\cX$ if and only if $f^* \ov{\cD}$ is pseudo-effective [effective] on $\cX'$. 
\end{lemma}
\begin{proof}
    The first part is standard, see for example \cite[Section 2]{Ikoma_Concavity_of_arithmetic_volume} or \cite[Definition 4.3.2]{Moriwaki_Adelic_divisors_on_Arithmetic_Varieties}.
    Now we prove the second part. Assume that $f^* \ov{\cD}$ is pseudo-effective on $\cX'$. Take a big arithmetic $\RR$-divisor of $C^0$-type $\ov{\cB}$ on $\cX$. Then, by the invariance of the arithmetic volume under birational maps, $f^* \ov{\cB}$ is big, so we get 
    \[ \avol(\ov{\cB} + \ov{\cD}) = \avol(f^* \ov{\cB} + f^* \ov{\cD}) > 0. \] Thus $\ov{\cD}$ is pseudo-effective on $\cX$.
    The other implication follows from Theorem~\ref{theorem_Ikoma_characterisation_of_psef_arithmetic_R_divisors}. To see that effectivity of $\ov{\cD}, f^* \ov{\cD}$ is equivalent, use Remark~\ref{remark_alternative_definition_of_effectivity_of_arithmetic_R_divisors} and the definition of the pullback of an arithmetic $\RR$-divisor of $C^0$-type.
\end{proof}

\begin{lemma}~\label{lemma_nef_cone_is_invariant_under_birational_pullbacks}   
    Let $f:\cX' \to \cX$ be a blowup and $\ov{\cD}$ be an arithmetic $\RR$-divisor of $C^0$-type on $\cX$. Then $\ov{\cD}$ is nef on $\cX$ if and only if $f^* \ov{\cD}$ is nef on $\cX'$.
\end{lemma}
\begin{proof}
    We check that the three conditions defining nefness are equivalent for $\ov{\cD}$ and $f^* \ov{\cD}$.
    \begin{itemize}
        \item Relative nefness: here equivalence follows from the projection formula and the fact that if $\cC$ is a curve in $\cX$ mapped to a closed point in $\Spec(\ZZ)$, then there is a curve $\cC'$ in $\cX'$ whose push-forward gives a positive multiple of $\cC$.
        \item $(C^0 \cap \PSH)$-type: here equivalence follows from \cite[Proposition 1.4.1]{Moriwaki_Adelic_divisors_on_Arithmetic_Varieties}.
        \item Non-negative height on closed points: here equivalence follows from Lemma~\ref{lemma_behaviour_of_height_for_birational_maps_and_pullbacks_proj_formula} and from the fact that $f$ is surjective on closed points.
    \end{itemize}
\end{proof}

\begin{remark}
    Because of Lemma~\ref{lemma_pseudo_effective_cone_is_invariant_under_birational_pullbacks} and Lemma~\ref{lemma_nef_cone_is_invariant_under_birational_pullbacks} we use the notions of bigness, pseudo-effectivity, effectivity and nefness without referring to an arithmetic variety where the divisor is defined.
\end{remark}

\subsection{Differentiability of the arithmetic volume}

Results in this section follow easily from \cite{Ikoma_Concavity_of_arithmetic_volume}. We provide some proofs for the reader's convenience.

\begin{definition}\label{definition_continuity_of_functions_from_adelic_divisors_to_R}
    We call a function $\alpha: \Adiv_{\RR}(\cX) \to \RR$ \textit{continuous}, if
    \[ \alpha(\ov{\cD}) = \lim_{\varepsilon_1, \dots, \varepsilon_r, \| f \|_{\sup} \to 0} \alpha(\ov{\cD} + \sum_{i=1}^r \varepsilon_i \ov{\cD}_i + (0, f)), \]
    for any arithmetic $\RR$-divisors of $C^0$-type $\ov{\cD}, \ov{\cD}_1, \dots, \ov{\cD}_r$ on $\cX$ and any $f$ of $C^0$-type on $\cX$.
\end{definition}

\begin{definition}~\label{definition_positive_intersection_power}
    Let $\ov{\cD}$ be a big arithmetic $\RR$-divisor of $C^0$-type on $\cX$. 
    Let $\hat{\Theta}(\ov{\cD})$ be the set of pairs $(f:\cX' \to \cX, \ov{\cA})$ consisting of a blowup $f:\cX' \to \cX$ together with a nef arithmetic $\RR$-divisor of $C^0$-type on $\cX'$ such that $f^*\ov{\cD} - \ov{\cA}$ is pseudo-effective. Let $\hat{\Theta}_{\amp}(\ov{\cD})$ be the set of pairs $(f:\cX' \to \cX, \ov{\cA})$ in $\hat{\Theta}(\ov{\cD})$ where $\ov{\cA}$ is an ample arithmetic $\RR$-divisor of $C^{\infty}$-type on $\cX'$ and $f^*\ov{\cD} - \ov{\cA}$ is an effective arithmetic $\QQ$-divisor of $C^0$-type on $\cX'$. The \textit{positive intersection $d$'th power of $\ov{\cD}$} is the function
    \[ \langle \ov{\cD}^d \rangle :\Adiv_{\RR}(\cX) \to \RR \]
    defined in the following way:
    \[ \langle \ov{\cD}^d \rangle \ov{\cE} = \sup_{(f, \ov{\cA}) \in \hat{\Theta}_{\amp}(\ov{\cD})} \ov{\cA}^d \cdot f^* \ov{\cE}, \textnormal{ for nef and big } \ov{\cE} \in \Adiv_{\RR}(\cX); \]
    and for other $\ov{\cE} \in \Adiv_{\RR}(\cX)$ we extend by linearity and continuity in the sense of Definition~\ref{definition_continuity_of_functions_from_adelic_divisors_to_R}.
\end{definition}

\begin{remark}~\label{remark_equiv_of_def_of_pos_intersection_in_Ikoma_and_pos_int_for_nef}
    The Definition~\ref{definition_positive_intersection_power} is equivalent to the one in \cite{Ikoma_Concavity_of_arithmetic_volume}. This equivalence follows from \cite[Proposition 3.9, Proposition 3.10.(3), proof of Theorem 5.3]{Ikoma_Concavity_of_arithmetic_volume}. Moreover, if $\ov{\cD}$ is nef in the above definition, then $\langle \ov{\cD}^d \rangle \ov{\cE} = \ov{\cD}^d \cdot \ov{\cE}$ by \cite[Remark 3.8.(3)]{Ikoma_Concavity_of_arithmetic_volume}.
\end{remark}

\begin{theorem}~\label{theorem_differentiability_of_arithmetic_volume}~\cite[Theorem 5.3]{Ikoma_Concavity_of_arithmetic_volume}
    The function $\avol:\Adiv_{\RR}(\cX) \to \RR$ has the following properties:
    \begin{enumerate}
        \item It is continuous in the sense of Definition~\ref{definition_continuity_of_functions_from_adelic_divisors_to_R}.
        \item It is positively $(d+1)$-homogeneous, i.e.,
        \[ \avol(t \ov{\cD}) = t^{d+1} \avol(\ov{\cD}), \]
        for any $t>0$ and arithmetic $\RR$-divisor of $C^0$-type $\ov{\cD}$ on $\cX$.
        \item It is differentiable on the big cone and at a big $\ov{\cD} \in \Adiv_{\RR}(\cX)$ the derivative is given by
        \[ D_{\ov{\cD}} \avol(\ov{\cE}) = (d+1) \langle \ov{\cD}^d \rangle \ov{\cE}. \]
    \end{enumerate}
\end{theorem}

\begin{lemma}~\label{lemma_psef_cone_increases_arithmetic_volume}
    Let $\ov{\cA}, \ov{\cB}$ be arithmetic $\RR$-divisors of $C^0$-type on $\cX$ with $\ov{\cB} - \ov{\cA}$ pseudo-effective. Then $\avol(\ov{\cA}) \leq \avol(\ov{\cB})$. 
\end{lemma}
\begin{proof}
    Without loss of generality $\avol(\ov{\cA}) > 0$. Fix $\varepsilon > 0$. We have
    \[ \avol(\ov{\cB}) = \avol((1-\varepsilon) \ov{\cA} + (\varepsilon \ov{\cA} + \ov{\cB} - \ov{\cA})) \geq \avol((1-\varepsilon) \ov{\cA}) = (1-\varepsilon)^{d+1} \avol(\ov{\cA}), \]
    where the inequality follows as in Lemma~\ref{lemma_effective_are_pseudo_effective}, as $\varepsilon \ov{\cA} + \ov{\cB} - \ov{\cA}$ is big. By taking $\varepsilon \to 0$, we get the result.
\end{proof}

\begin{definition}~\label{definition_partial_order_on_positive_approximations_of_a_divisor}
    Let $\ov{\cD}$ be a big arithmetic $\RR$-divisor of $C^0$-type on $\cX$. Following \cite{Ikoma_Concavity_of_arithmetic_volume}, we equip the set $\hat{\Theta}(\ov{\cD})$ with a partial order defined in the following way:
    \[ (f_1:\cX_1 \to \cX, \ov{\cA}_1) \leq (f_2:\cX_2 \to \cX, \ov{\cA}_2) \]
    if and only if there exists a blowup $f:\cX' \to \cX$ which factors as
    % https://q.uiver.app/?q=WzAsMyxbMCwwLCJcXGNYJyJdLFsxLDAsIlxcY1hfaSJdLFsyLDAsIlxcY1giXSxbMCwxLCJnX2kiXSxbMSwyLCJmX2kiXV0=
    \[\begin{tikzcd}
    	{\cX'} & {\cX_i} & \cX
    	\arrow["{g_i}", from=1-1, to=1-2]
    	\arrow["{f_i}", from=1-2, to=1-3]
    \end{tikzcd}\]
    so that on $\cX'$ the arithmetic $\RR$-divisor of $C^0$-type $g_2^* \ov{\cA}_2 - g_1^* \ov{\cA}_1$ is pseudo-effective.
\end{definition}

\begin{lemma}~\label{lemma_filtering_result_from_Ikoma}~\cite[Proposition 3.2]{Ikoma_Concavity_of_arithmetic_volume}
    Let $\ov{\cD}$ be a big arithmetic $\RR$-divisor of $C^0$-type on $\cX$. Then $\hat{\Theta}(\ov{\cD})$ is filtered with respect to the order defined in Definition~\ref{definition_partial_order_on_positive_approximations_of_a_divisor}. Moreover, for any $(f_1:\cX_1 \to \cX, \ov{\cA}_1), (f_2:\cX_2 \to \cX, \ov{\cA}_2) \in  \hat{\Theta}(\ov{\cD})$ there is a blowup $f:\cX' \to \cX$ together with a nef arithmetic $\RR$-divisor of $C^0$-type $\ov{\cA}$ such that $f$ factors as in Definition~\ref{definition_partial_order_on_positive_approximations_of_a_divisor} and
    \[ g_1^* \ov{\cA}_2, g_2^* \ov{\cA}_2 \leq \ov{\cA} \leq f^* \ov{\cD}, \]
    where the inequalities are written with respect to the pseudo-effective cone.
\end{lemma}
%\begin{proof}
%    This is \cite[Proposition 3.2]{Ikoma_Concavity_of_arithmetic_volume}.
%\end{proof}

\begin{lemma}~\label{lemma_every_nef_approximation_is_almost_dominated_by_an_ample_one}~\cite[Proposition 3.1]{Ikoma_Concavity_of_arithmetic_volume}
    Let $\ov{\cD}$ be a big arithmetic $\RR$-divisor of $C^0$-type on $\cX$. Let $(f:\cX' \to \cX, \ov{\cA}) \in \hat{\Theta}(\ov{\cD})$ and fix a number $0 < \gamma < 1$. Then there is $\ov{\cB} \in \Adiv_{\RR}(\cX')$ with $(f, \ov{\cB}) \in \hat{\Theta}_{\amp}(\ov{\cD})$ such that $f^* \ov{\cD} - \ov{\cB}$ is an effective arithmetic $\QQ$-divisor of $C^0$-type and $\ov{\cB} - \gamma \ov{\cA}$ is pseudo-effective.
\end{lemma}
%\begin{proof}
%    This is \cite[Proposition 3.1]{Ikoma_Concavity_of_arithmetic_volume}.
%\end{proof}

\begin{lemma}~\label{lemma_calculating_derivative_of_volume_uniformely_for_a_few_divisors}
    Let $\ov{\cD}_1, \dots, \ov{\cD}_n$ be big and nef arithmetic $\RR$-divisors of $C^0$-type on $\cX$. Let $\ov{\cB}$ be a big arithmetic $\RR$-divisor of $C^0$-type on $\cX$. Fix $\varepsilon > 0$. Then there is a pair $(f, \ov{\cA}) \in \hat{\Theta}_{\amp}(\ov{\cB})$ such that for all $i=1, \dots, n$
    \[ |\langle \ov{\cB}^d \rangle \ov{\cD}_i - \ov{\cA}^d \cdot f^* \ov{\cD}_i| < \varepsilon \]
    and satisfying
    \[ \avol(\ov{\cB}) \geq \avol(\ov{\cA}) > \frac{1}{2} \avol(\ov{\cB}). \]
\end{lemma}
\begin{proof}
    First of all, we take $(f_0, \ov{\cA}_0) \in \hat{\Theta}(\ov{\cB})$ satisfying
    \[ \avol(\ov{\cB}) \geq \avol(\ov{\cA}_0) > \frac{3}{4} \avol(\ov{\cB}). \]
    Such exists by the arithmetic Fujita approximation theorem \cite[Theorem 4.3]{Chen2010_Fujita_approximation}, \cite[Theorem C]{yuan_2009_Fujita_approximation}; see also \cite[Proposition 3.11.(1)]{Ikoma_Concavity_of_arithmetic_volume}. Next, take pairs $(f_i, \ov{\cA}_i) \in \hat{\Theta}_{\amp}(\ov{\cB})$ such that
    \[ \langle \ov{\cB}^d \rangle \ov{\cD}_i - \varepsilon \leq \ov{\cA}_i^d \cdot f_i^* \ov{\cD}_i \leq \langle \ov{\cB}^d \rangle \ov{\cD}_i \]
    for all $i=1, \dots, n$. By Lemma~\ref{lemma_filtering_result_from_Ikoma} and induction, there is a pair $(f:\cX' \to \cX, \ov{\cA}') \in \hat{\Theta}(\ov{\cD})$ such that for all $i=0, \dots, n$ there exists a factorisation
    \[\begin{tikzcd}
    	{\cX'} & {\cX_i} & \cX
    	\arrow["{g_i}", from=1-1, to=1-2]
    	\arrow["{f_i}", from=1-2, to=1-3]
    \end{tikzcd}\]
    and
    \[ g_i^* \ov{\cA}_i \leq \ov{\cA}' \leq f^* \ov{\cD} \]
    with respect to the pseudo-effective cone. We can proceed by induction because of Lemma~\ref{lemma_pseudo_effective_cone_is_invariant_under_birational_pullbacks}. Choose a real number $0 < \gamma < 1$ to be specified later. By Lemma~\ref{lemma_every_nef_approximation_is_almost_dominated_by_an_ample_one} there is an ample $\ov{\cA} \in \Adiv_{\RR}(\cX')$ such that $(f, \ov{\cA}) \in \hat{\Theta}_{\amp}(\ov{\cB})$ and $\ov{\cA} - \gamma \ov{\cA}'$ is pseudo-effective. Using Lemma~\ref{lemma_psef_cone_increases_arithmetic_volume} we get:
    % Why if A>=B (pseudo-effective cone) then 
    \[ \avol(\ov{\cA}) \geq \avol(\gamma \ov{\cA}') = \gamma^d \avol(\ov{\cA}') \geq \gamma^d \avol(\ov{\cA}_0) > \gamma^d \cdot \frac{3}{4} \avol(\ov{\cB}) > \frac{1}{2} \avol(\ov{\cB}), \]
    for $\gamma$ such that $\gamma^d \cdot \frac{3}{4} > \frac{1}{2}$. The inequality $\avol(\ov{\cA}) \leq \avol(\ov{\cB})$ also follows from Lemma~\ref{lemma_psef_cone_increases_arithmetic_volume}. Moreover for $i=1, \dots, n$ we have
    \[ \langle \ov{\cB}^d \rangle \ov{\cD}_i \geq \ov{\cA}^d \cdot f^* \ov{\cD}_i \geq (\gamma \ov{\cA}')^d \cdot f^* \ov{\cD}_i \geq \gamma^{d} \cdot \ov{\cA}_i^d \cdot f_i^* \ov{\cD}_i \]
    \[ \geq \gamma^{d} (\langle \ov{\cB}^d \rangle \ov{\cD}_i - \varepsilon) = \gamma^{d} \langle \ov{\cB}^d \rangle \ov{\cD}_i - 
\gamma^{d} \varepsilon. \]
    Thus
    \[ |\langle \ov{\cB}^d \rangle \ov{\cD}_i - \ov{\cA}^d \cdot f^* \ov{\cD}_i| \leq (1-\gamma^{d}) \langle \ov{\cB}^d \rangle \ov{\cD}_i + \gamma^{d} \varepsilon. \]
    Let $M := \max(\langle \ov{\cB}^d \rangle \ov{\cD}_i :i=1, \dots, n)$. Then for $\gamma$ such that
    \[ (1-\gamma^{d})M + \gamma^{d} \varepsilon < 2 \varepsilon \]
    and using the fact that $\varepsilon$ was arbitrary, we are done.
\end{proof}

\begin{lemma}~\label{lemma_continuity_of_intersection_height_and_derivative_of_volume_in_the_euclidean_topology}
    The following functions are continuous in the sense of Definition~\ref{definition_continuity_of_functions_from_adelic_divisors_to_R}:
    \begin{itemize}
        \item $h_{(-)}(x)$ for a closed point $x \in \cX \otimes \QQ$.
        \item $\ov{\cD}^d \cdot (-)$ for $\ov{\cD}$ being an ample arithmetic $\RR$-divisor of $C^{\infty}$-type on $\cX$.
        \item $\langle \ov{\cB}^d \rangle (-)$ for $\ov{\cB}$ being a big arithmetic $\RR$-divisor of $C^0$-type on $\cX$.
    \end{itemize}
\end{lemma}
\begin{proof}
    All of the functions in question are linear, so it's enough to check that for $\alpha$ being one of those, we have
    \[ \lim_{\| f \|_{\sup} \to 0} \alpha((0,f)) = 0. \]
    For the height function it follows directly from the definition. For the second function this is \cite[Lemma 2.5]{Ikoma_Concavity_of_arithmetic_volume}. For the last function this is by Definition~\ref{definition_positive_intersection_power}.
\end{proof}

\subsection{Essential infimum and Zhang's inequality}

\begin{definition}~\label{definition_essential_infimum_of_an_arithmetic_divisor}
    Let $\ov{\cD}$ be an arithmetic $\RR$-divisor of $C^0$-type on $\cX$. Use the notation $X = \cX \otimes \QQ$. The \textit{essential infimum} of $\ov{\cD}$ is the following number
    \[ \zeta(\ov{\cD}) := \sup_{\cY \subset \cX} \inf_{x \in \cX \setminus \cY(\ov{\QQ})} h_{\ov{\cD}}(x), \]
    where the supremum is taken taken over all closed proper subschemes $\cY \subset \cX$. Note that $\zeta$ depends only on the rational class of an arithmetic $\RR$-divisor of $C^0$-type and a priori has values in $\RR \cup \{ \pm \infty \}$.
\end{definition}

\begin{lemma}~\label{lemma_properties_of_zeta_function}
    The essential infimum function $\zeta$ has the following properties:
    \begin{enumerate}
        \item If $\ov{\cD}$ is an arithmetic $\RR$-divisor of $C^0$-type on $\cX$ and $f:\cX' \to \cX$ is a blowup, then $\zeta(\ov{\cD}) = \zeta(f^* \ov{\cD})$.
        \item For any arithmetic $\RR$-divisors of $C^0$-type $\ov{\cA}, \ov{\cB}$ on $\cX$
        \[ \zeta(\ov{\cA} + \ov{\cB}) \geq \zeta(\ov{\cA}) + \zeta(\ov{\cB}). \]
        \item It is positively $1$-homogeneous.
        \item It is non-negative on effective arithmetic $\RR$-divisors of $C^0$-type.
        \item For any arithmetic $\RR$-divisors of $C^0$-type $\ov{\cA}, \ov{\cB}$ on $\cX$ such that $\ov{\cB}$ is big and $\ov{\cA} \leq \ov{\cB}$ with respect to the pseudo-effective cone, we have $\zeta(\ov{\cA}) \leq \zeta(\ov{\cB})$.
        \item It has values in $\RR \cup \{ - \infty \}$.
    \end{enumerate}
\end{lemma}
\begin{proof}
    The first property follows from the fact that blowups are isomorphisms on an open set and from Lemma~\ref{lemma_behaviour_of_height_for_birational_maps_and_pullbacks_proj_formula}. For the proof of properties two, three and four, see \cite[Lemma 3.15]{F_Ballay_Succesive_minima}. For the fifth one fix arithmetic $\RR$-divisors of $C^0$-type $\ov{\cA}, \ov{\cB}$ on $\cX$ such that $\ov{\cB}$ is big and $\ov{\cA} \leq \ov{\cB}$. Then we have
    \[ \zeta(\ov{\cB}) = \lim_{\varepsilon \to 0^+} \zeta((1+\varepsilon)\ov{\cB}) = \lim_{\varepsilon \to 0^+} \zeta( \varepsilon \ov{\cB} + (\ov{\cB} - \ov{\cA}) + \ov{\cA} ) \]
    \[ \geq \liminf_{\varepsilon \to 0^+} \zeta(\varepsilon \ov{\cB} + (\ov{\cB} - \ov{\cA})) + \zeta(\ov{\cA}). \]
    By pseudo-effectivity of $\ov{\cB} - \ov{\cA}$ and bigness of $\ov{\cB}$ the divisor $\varepsilon \ov{\cB} + (\ov{\cB} - \ov{\cA})$ is big for $\varepsilon>0$, hence the value of $\zeta$ on it is non-negative. Thus we get $\zeta(\ov{\cB}) \geq \zeta(\ov{\cA})$. 

    By Lemma~\ref{lemma_decomposition_of_an_arithmetic_divisor_into_a_difference_of_amples_plus_a_function} and the fact that values of $\zeta$ only depend on the rational equivalence classes of arithmetic $\RR$-divisors of $C^0$-type, to get the last property it is enough to show that for an ample $\ZZ$-divisor of $C^{\infty}$-type $\ov{\cA}$ one has $\zeta(\ov{\cA}) < \infty$. This follows from Zhang's inequalities, but for an elementary proof using Weil's height machine, see \cite[Proposition 2.6]{Boucksom_Chen_Okounkov_bodies_of_filtered_linear_series}.
\end{proof}

\begin{lemma}~\label{lemma_F_Ballay_lemma_3_16}~\cite[Lemma 3.16]{F_Ballay_Succesive_minima}
    Let $\ov{\cD}$ be an arithmetic $\RR$-divisor of $C^0$-type on $\cX$ and let $\ov{\cD}_0 = (\cdiv(2),0)$. Then 
    \begin{enumerate}
        \item $\zeta(\ov{\cD} - \frac{r}{\log 2} \cdot \ov{\cD}_0) = \zeta(\ov{\cD}) - r$ for all real $r$;
        \item if $\ov{\cD}$ is pseudo-effective and $\cD_{\QQ}$ is big, then $\zeta(\ov{\cD}) \geq 0$;
        \item if $\cD_{\QQ}$ is big, the following inequality holds
        \[ \zeta(\ov{\cD}) \geq \sup \{ r : \ov{\cD} - \frac{r}{\log 2} \cdot \ov{\cD}_0 \textnormal{ is pseudo-effective} \}. \]
    \end{enumerate}
\end{lemma}
\begin{proof}
    The first point follows from additivity of height and Lemma~\ref{lemma_height_with_respect_to_divisor_max_div_two_and_zero}. The second point is \cite[Lemma 3.16]{F_Ballay_Succesive_minima}. The third point follows from the first two.
\end{proof}

\begin{theorem}~\label{theorem_Zhang_inequality}
    Let $\ov{\cD}$ be a pseudo-effective arithmetic $\RR$-divisor of $C^0$-type on $\cX$ with $\cD_{\QQ}$ big. Then the following inequality holds
    \[ \zeta(\ov{\cD}) \geq \frac{\avol(\ov{\cD})}{(d+1) \vol(\cD_{\QQ})}. \]
\end{theorem}
\begin{proof}
    In this form the theorem follows from \cite[Theorem 7.2.(ii)]{F_Ballay_Succesive_minima}.
\end{proof}

\subsection{Convex geometry}

\begin{lemma}~\label{lemma_convex_geometry_interior_of_closure_is_the_starting_open_cone}
    An open convex set in $\RR^n$ is the interior of its closure.
\end{lemma}
\begin{proof}
    This is classical, for a proof see for example \cite[Lemma 5.6]{GVF2}.
\end{proof}

\begin{lemma}~\label{lemma_convex_geometry_separation_by_a_linear_functional}
    Let $V$ be a finite dimensional real vector space and let $V^+$ be a closed convex cone in $V$ not containing any line. Let $l:V \to \RR$ be a linear functional non-negative on $V^+$. Then for any vectors $v_0, \dots, v_n \in V$ and for any $\varepsilon>0$ there is a linear functional $l': V \to \RR$ strictly positive on $V^+ \setminus \{0\}$ such that for all $i=0, \dots, n$
    \[ |l(v_i) - l'(v_i)| \leq \varepsilon. \]
    Moreover, if $v_0 \in V^+$, then one can additionally assume that $l'(v_0) \geq l(v_0)$.
\end{lemma}
\begin{proof}
    Fix $v_0, \dots, v_n \in V$ and $\varepsilon>0$. It is enough to find a linear functional $m:V \to \RR$ such that $V^+ \setminus \{0\} \subset \{v \in V: m(v)>0\}$. Indeed, in this case we can just put $l' := l+\delta m$ for small enough positive $\delta$. To define $m$ one can equip $V$ with a (non-degenerate, positively defined) scalar product and take $m$ to be the product with a vector from the interior of the dual cone of $V^+$. For details, see for example \cite[Lemma 5.7]{GVF2}.
\end{proof}

% This lemma is based on Theorem 5.2, Corollary 5.3 in GVF!
\begin{lemma}~\label{lemma_convex_geometry_volume_divided_by_linear}
    Let $V$ be a finite dimensional real vector space and let $V^+$ be a closed convex cone in $V$ generating the whole space. Assume that $l:V \to \RR$ is a linear functional strictly positive on $V^+ \setminus \{0\}$ and that $\alpha:V \to \RR$ is a function which is:
    \begin{enumerate}
        \item nonzero,
        \item continuous,
        \item positively $1$-homogeneous,
        \item differentiable on $(V^+)^{\interior}$,
        \item vanishing outside $(V^+)^{\interior}$.
    \end{enumerate}
    Then there is a unique continuous continuation of the function $\frac{\alpha}{l}$ to $V \setminus \{0\}$. Moreover, there exists a point $p \in (V^+)^{\interior}$ such that
    \[ \frac{D_p \alpha}{\alpha(p)} = \frac{l}{l(p)}. \]
\end{lemma}
\begin{proof}
    Pick an euclidean norm on $V$ and let $S$ be the unit sphere of it. Since both $l$ and $\alpha$ are positively $1$-homogeneous, for the first part it is enough to show that $\frac{\alpha}{l}$ has a unique continuous extension to $S$. This follows from the fact that if $l(q)=0$ for $q \in S$, then $q \in S \cap V \setminus V^+$ and $S \cap V \setminus V^+$ is a relatively open set in $S$, where $\alpha$ is zero. To prove the second part, let $p \in S$ be a point where the maximum of $\frac{\alpha}{l}$ is achieved (it exists by compactness of $S$). Then $\alpha(p), l(p) > 0$, because $\alpha$ is nonzero. Point $p$ is a local maximum of $\frac{\alpha}{l}$ and so also of $\log(\frac{\alpha}{l}) = \log(\alpha) - \log(l)$ (note that both functions are defined in a neighbourhood of $p$). Hence the derivative of $\log(\frac{\alpha}{l})$ at $p$ vanishes and we get
    \[ \frac{D_p \alpha}{\alpha(p)} = D_p(\log(\alpha)) = D_p(\log(l)) = \frac{l}{l(p)}. \]
\end{proof}

\subsection{Approximating a global functional}\label{subsection_approx_global_functional}

For this subsection we fix a normal, generically smooth arithmetic variety $\cX$ over $\Spec(\ZZ)$. Let $d+1 \geq 2$ be the dimension of $\cX$.
% and let $V^{\geq 0}$ be the closure (in the euclidean topology) of the intersection of $V$ with the cone of classes of effective divisors
% Let $V = \Span_{\RR}(\ov{\cM}_0, \dots, \ov{\cM}_n)$ inside $\Apic_{\RR}(\cX)$. 
% \cX has dimension d+1!!!
\begin{proposition}~\label{proposition_formula_for_derivative_of_avol_ample_in_blowup}
    Let $\ov{\cB}, \ov{\cD}_0, \dots, \ov{\cD}_n$ be arithmetic $\RR$-divisors of $C^0$-type on $\cX$ with $\ov{\cB}$ big. Fix $\varepsilon>0$. Then there exists a blowup $f:\cX' \to \cX$ and an ample arithmetic $\QQ$-divisor of $C^{\infty}$-type $\ov{\cB}'$ on $\cX'$ such that:
    \begin{itemize}
        \item $\vol(\cB'_{\QQ}) > \frac{\avol(\ov{\cB})}{2(d+1) \zeta(\ov{\cB})} =: A$;
        \item $|\langle \ov{\cB}^d \rangle \ov{\cD}_i - \ov{\cB}'^d \cdot f^* \ov{\cD}_i| < A \varepsilon$ for all $i = 0, \dots, n$.
    \end{itemize}
\end{proposition}
\begin{proof}
    % Step 1 - reduction to nice divisors!
    First, note that the expression
    \[ \langle \ov{\cB}^d \rangle \ov{\cD} - \ov{\cB}'^d \cdot f^* \ov{\cD} \]
    is linear and continuous in $\ov{\cD}$, see Lemma~\ref{lemma_continuity_of_intersection_height_and_derivative_of_volume_in_the_euclidean_topology}. Moreover, it only depends on the rational equivalence class of $\ov{\cD}$. Hence, by Lemma~\ref{lemma_decomposition_of_an_arithmetic_divisor_into_a_difference_of_amples_plus_a_function} we can without loss of generality assume that $\ov{\cD}_0, \dots, \ov{\cD}_n$ are ample arithmetic $\RR$-divisors. In particular they are big and nef. Thus, by Lemma~\ref{lemma_calculating_derivative_of_volume_uniformely_for_a_few_divisors} for $A \varepsilon$ instead of $\varepsilon$, there is a pair $(f:\cX' \to \cX, \ov{\cA}) \in \hat{\Theta}_{\amp}(\ov{\cB})$ such that for all $i=0, \dots, n$
    \[ |\langle \ov{\cB}^d \rangle \ov{\cD}_i - \ov{\cA}^d \cdot f^* \ov{\cD}_i| < A \varepsilon \]
    and
    \[ \avol(\ov{\cB}) \geq \avol(\ov{\cA}) > \frac{1}{2} \avol(\ov{\cB}). \]
    We use Theorem~\ref{theorem_Zhang_inequality} to get the middle inequality of
    \[ \frac{\avol(\ov{\cB})}{2(d+1) \vol(\cA_{\QQ})} < \frac{\avol(\ov{\cA})}{(d+1) \vol(\cA_{\QQ})} \leq \zeta(\ov{\cA}) \leq \zeta(f^* \ov{\cB}) = \zeta(\ov{\cB}). \]
    The last inequality is by Lemma~\ref{lemma_properties_of_zeta_function}. We get
    \[ \vol(\cA_{\QQ}) > \frac{\avol(\ov{\cB})}{2(d+1) \zeta(\ov{\cB})} = A. \]
    Note that $\ov{\cA}$ is an ample arithmetic $\RR$-divisor of $C^{\infty}$-type on $\cX'$, so it is of the form
    \[ \ov{\cA} = \sum_j a_j \ov{\cA}_j, \]
    for some some $a_j \in \RR$, and $\ov{\cA}_j$ being ample arithmetic $\ZZ$-divisors of $C^{\infty}$-type. Replacing $a_j$'s by sufficiently good rational approximations, we can find an ample arithmetic $\QQ$-divisor of $C^{\infty}$-type $\ov{\cB}'$ on $\cX'$ satisfying the desired inequalities. This is because of multilinearity of the arithmetic intersection product and the continuity of the volume $\vol$.
\end{proof}

The below proposition is proven in a way similar to \cite[Theorem 4.1]{F_Ballay_Succesive_minima} or \cite[Theorem 4.6]{Arithmetic_Demailly_Qu_Yin}.

\begin{proposition}~\label{proposition_arithmetic_Bertini_closed_point_with_good_heights}
    Let $\ov{\cB}, \ov{\cD}_0, \dots, \ov{\cD}_n$ be arithmetic $\RR$-divisors of $C^0$-type on $\cX$ with $\ov{\cB}$ being an ample arithmetic $\QQ$-divisor of $C^{\infty}$-type. Let $Z \subset X = \cX \otimes \QQ$ be a closed proper subscheme. Fix $\varepsilon>0$. Then there exists a closed point $x \in X \setminus Z$ such that for all $i=0, \dots, n$
    \[ \Big| h_{\ov{\cD}_i}(x) - \frac{\ov{\cB}^d \cdot \ov{\cD}_i}{\vol(\cB_{\QQ})} \Big| < \varepsilon \textnormal{ for all } i=0, \dots, n. \]
\end{proposition}
\begin{proof}
    % Step 1 - reduction to the case where all $\ov{\cD}_i$ are ample $\ZZ$-divisors.
    First, note that $\frac{\ov{\cB}^d \cdot \ov{\cD}_i}{\vol(\cB_{\QQ})}$ does not change when we replace $\ov{\cB}$ with its multiples, so we can without loss of generality assume that $\ov{\cB}$ is an ample $\ZZ$-arithmetic divisor of $C^{\infty}$-type. Also, the expression
    \[ h_{\ov{\cD}}(x) - \frac{\ov{\cB}^d \cdot \ov{\cD}}{\vol(\cB_{\QQ})} \]
    is linear in $\ov{\cD}$ and depends only on its rational equivalence class. Thus, by Lemma~\ref{lemma_continuity_of_intersection_height_and_derivative_of_volume_in_the_euclidean_topology} and Lemma~\ref{lemma_decomposition_of_an_arithmetic_divisor_into_a_difference_of_amples_plus_a_function}, we can without loss of generality assume that all $\ov{\cD}_i$'s are ample arithmetic $\ZZ$-divisors of $C^{\infty}$-type.
    % Step 2 - Use Wilms's result
    
    The above reduction allows us to work with ample hermitian line bundles 
    \[ \cO(\ov{\cB}), \cO(\ov{\cD}_0), \dots, \cO(\ov{\cD}_n) \textnormal{ on } \cX. \]
    To ease the notation we don't write $\cO(-)$ for the corresponding hermitian line bundles. Let $C$ be a positive number, such that for all $i=0, \dots, n$ we have $C \cdot c_1(\ov{\cB}) \geq c_1(\ov{\cD}_i)$. Now we use Theorem~\ref{theorem_arithmetic_Bertini_by_Wilms} together with Theorem~\ref{theorem_Charles_Bertini_to_omit_subvariety} for $\ov{\cL} = \ov{\cB}$ and $\cY$ being the closure of $Z$ in $\cX$, to find a small section $s_1 \in \aH^0(\cX, \ov{\cL}^{\otimes n_1})$ which is $(\delta, \ov{\cL})$-irreducible (we specify $\delta$ later) and with $\cdiv(s_1)$ generically smooth and not contained in $\cY$. Then for all $i=0, \dots, n$ we have
    \[ \ov{\cD}_i \cdot \ov{\cB}^d = \frac{1}{n_1} \ov{\cD}_i \cdot \ov{\cB}^{d-1} \cdot \adiv(s_1) \]
    \[ = \frac{1}{n_1} \Bigl( \ov{\cD}_i|_{\cdiv(s_1)} \cdot (\ov{\cB}|_{\cdiv(s_1)})^{d-1} + \int_{\cX(\CC)} - \log \|s\| 
 \wedge c_1(\ov{\cD}_i) \wedge c_1(\ov{\cB})^{d-1} \Bigr). \]
    Note that in the above formula we treat $\cdiv(s_1)$ as a cycle on $\cX$ as in Theorem~\ref{theorem_intersection_product_of_hermitian_line_bundles}. By the choice of $s_1$ (and using one of the equivalent conditions from Definition~\ref{definition_irreducibility_defined_by_Wilms}) we get
    \[ 0 \leq \int_{\cX(\CC)} - \log \|s\| 
 \wedge c_1(\ov{\cD}_i) \wedge c_1(\ov{\cB})^{d-1} \]
    \[ \leq C \cdot \int_{\cX(\CC)} - \log \|s\| \wedge c_1(\ov{\cB})^d \leq \frac{\delta}{1+\delta} \cdot n_1 C \cdot \ov{\cB}^{d+1} < n_1 \varepsilon, \]
    for $\delta$ small enough for the last inequality. Together we obtain
    \[ \big| \ov{\cD}_i \cdot \ov{\cB}^d - \frac{1}{n_1} \ov{\cD}_i|_{\cdiv(s_1)} \cdot (\ov{\cB}|_{\cdiv(s_1)})^{d-1} \big| < \varepsilon. \]
    Now we use exactly the same strategy, to find a small section $s_2 \in \aH^0(\cdiv(s_1), \ov{\cL}|_{\cdiv{(s_1)}}^{\otimes n_2})$ which satisfies
    \[ \big| \ov{\cD}_i|_{\cdiv(s_1)} \cdot (\ov{\cB}|_{\cdiv(s_1)})^{d-1} - \frac{1}{n_2} \ov{\cD}_i|_{\cdiv(s_1) \cap \cdiv(s_2)} \cdot (\ov{\cB}|_{\cdiv(s_1) \cap \cdiv(s_2)})^{d-2} \big| < \varepsilon. \]
    We can do this since the cycle $\cdiv(s_1)$ is irreducible, hence it is an (generically smooth, but not necessarily normal) arithmetic variety. Moreover, restrictions of ample hermitian line bundles are ample, for example by \cite[Corollary 2.8]{Charles2021ArithmeticAA}. In the above inequality $\cdiv(s_1) \cap \cdiv(s_2)$ is the cycle associated to zeroes of $s_2$ on the arithmetic variety $\cdiv(s_1)$.
    % Step 3 - putting everything together

    Repeating the above procedure $d$ times, we find small sections $s_1, \dots, s_d$ of $\ov{\cL}^{\otimes n_1}, (\ov{\cL}|_{\cdiv(s_1)})^{\otimes n_2}, \dots, (\ov{\cL}|_{\cdiv(s_1) \cap \ldots \cap \cdiv(s_{d-1})})^{\otimes n_d}$ respectively, such that the following hold:
    \begin{enumerate}
        \item $\cdiv(s_1) \cap \dots \cap \cdiv(s_d)$ is an irreducible, generically smooth arithmetic variety of dimension $1$, hence it is of the form $\ov{\{x\}}$ for some closed point $x \in X$. Moreover, we can assume that $x \not\in Z$.
        \item For all $i=0, \dots, n$ we have \begin{equation}~\label{equation_proof_of_proposition_arithmetic_Bertini_closed_point_with_good_heights}
            \big| \ov{\cD}_i \cdot \ov{\cB}^d - \frac{1}{n_1 \cdots n_d} \adeg(\ov{\cD}_i | \ov{\{x\}}) \big| < d \varepsilon.
        \end{equation}
    \end{enumerate}
    By the inductive definition of the ordinary intersection product of line bundles on a variety $X$ (see for example \cite[Remark (2.6)]{Chambert_Loir_survey}) and by ampleness of $\cB_{\QQ}$ we get
    \[ \vol(\cB_{\QQ}) = \cB_{\QQ}^d = \frac{1}{n_1} \deg(\cB_{\QQ}^{d-1}|\cdiv(s_1)_{\QQ}) \]
    \[ = \frac{1}{n_1 \cdot n_2} \deg(\cB_{\QQ}^{d-2}|\cdiv(s_1)_{\QQ} \cap \cdiv(s_2)_{\QQ}) \]
    \[ = \dots = \frac{1}{n_1 \cdots n_d} \deg(\cdiv(s_1)_{\QQ} \cap \dots \cap \cdiv(s_d)_{\QQ}) \]
    \[ = \frac{1}{n_1 \cdots n_d} \deg(x). \]
    Hence, using Equation~(\ref{equation_proof_of_proposition_arithmetic_Bertini_closed_point_with_good_heights}) we end up with
    \[ \bigg| \frac{\ov{\cD}_i \cdot \ov{\cB}^d}{\vol(\cB_{\QQ})} - \frac{\adeg(\ov{\cD}_i | \ov{\{x\}})|}{\deg(x)} \bigg| < \frac{d}{\vol(\cB_{\QQ})} \varepsilon \textnormal{ for all } i=0, \dots, n. \]
    By the definition of height we have
    \[ h_{\ov{\cD}_i}(x) = \frac{\adeg(\ov{\cD}_i | \ov{\{x\}})}{\deg(x)}, \]
    so
    \[ \Big| h_{\ov{\cD}_i}(x) - \frac{\ov{\cD}_i \cdot \ov{\cB}^d}{\vol(\cB_{\QQ})} \Big| < \frac{d}{\vol(\cB_{\QQ})} \varepsilon \textnormal{ for all } i=0, \dots, n. \]
    This finishes the proof, as $\varepsilon$ was arbitrary.
\end{proof}

\begin{theorem}~\label{theorem_main_arithmetic_approximation_theorem}
    Let $\ov{\cD}_0, \dots, \ov{\cD}_n$ be arithmetic $\RR$-divisors of $C^0$-type on $\cX$ with $\ov{\cD}_0 = (\cdiv(2),0)$. Fix $\varepsilon > 0$, a closed proper subscheme $Z \subset X=\cX \otimes \QQ$, and a linear functional
    \[ l:\Adiv_{\RR}(\cX) \to \RR, \]
    which is non-negative on effective arithmetic $\RR$-divisors of $C^0$-type, zero on principal arithmetic $\RR$-divisors of $C^0$-type, and sends $\ov{\cD}_0$ to $\log(2)$. Then there exists a closed point $x \in X \setminus Z$ such that for all $i=0, \dots, n$
    \[ |h_{\ov{\cD}_i}(x) - l(\ov{\cD}_i)| < \varepsilon. \]
\end{theorem}
\begin{proof}
    % Step 1
    Put $M := \max(|l(\ov{\cD}_0)|, \dots, |l(\ov{\cD_n})|)$. Let $V$ be the real span of classes of $\ov{\cD}_0, \dots, \ov{\cD}_n, \ov{\cM}$ in $\aN^1(\cX)$, where $\ov{\cM}$ is some big arithmetic $\RR$-divisor of $C^0$-type in $\Adiv_{\RR}(\cX)$. Let $V^+$ be the (euclidean) closure of the subset of classes of big arithmetic $\RR$-divisors of $C^0$-type from $V$. It is a convex closed cone that does not contain any line in $V$ by Theorem~\ref{theorem_Ikoma_characterisation_of_psef_arithmetic_R_divisors}. Moreover, $\ov{\cD}_0$ is effective, hence by Lemma~\ref{lemma_effective_are_pseudo_effective} it lies in $V^+$. Thus, by Lemma~\ref{lemma_convex_geometry_separation_by_a_linear_functional} we can approximate $l$ (by the assumptions it is well defined on $V$) arbitrarily close by linear functionals strictly positive on $V^+ \setminus \{0\}$ and non-decreasing the value on $\ov{\cD}_0$. Since $\varepsilon$ is arbitrary in the statement of the theorem, we can without loss of generality replace $l$ with sufficiently good approximation, so we assume that $l$ is strictly positive on $V^+ \setminus \{0\}$, $l(\ov{\cD}_0) = \log(2) + \varepsilon'$ for some $0 \leq \varepsilon'<\varepsilon$, and that $\max(|l(\ov{\cD}_0)|, \dots, |l(\ov{\cD_n})|) < 2M$.
    
    % Step 2
    We use Lemma~\ref{lemma_convex_geometry_volume_divided_by_linear} with $\alpha = \avol^{1/(d+1)}$ and $V^+, V, l$ as above. Note that the assumptions of the lemma are satisfied by Theorem~\ref{theorem_differentiability_of_arithmetic_volume} and by Lemma~\ref{lemma_convex_geometry_interior_of_closure_is_the_starting_open_cone} which implies that $(V^+)^{\interior}$ is the set of classes of big arithmetic $\RR$-divisors of $C^0$-type in $V$. Thus there exists $p \in (V^+)^{\interior}$ such that
    \[ \frac{D_p \alpha}{\alpha(p)} = \frac{l}{l(p)}. \]
    If $\ov{\cB}$ represents the class of $p$, then by Theorem~\ref{theorem_differentiability_of_arithmetic_volume} we get
    \[ \frac{\langle \ov{\cB}^d \rangle}{\avol(\ov{\cB})} = \frac{l}{l(\ov{\cB})}. \]
    
    % Step 3
    By Proposition~\ref{proposition_formula_for_derivative_of_avol_ample_in_blowup} there is a blowup $f:\cX' \to \cX$ and an ample arithmetic $\QQ$-divisor of $C^\infty$-type $\ov{\cB}'$ on $\cX'$ such that the following hold:
    \begin{itemize}
        \item $\vol(\cB'_{\QQ}) > \frac{\avol(\ov{\cB})}{2(d+1) \zeta(\ov{\cB})} =: A$;
        \item $|\langle \ov{\cB}^d \rangle \ov{\cD}_i - \ov{\cB}'^d \cdot f^* \ov{\cD}_i| < A \varepsilon$ for all $i = 1, \dots, n$.
    \end{itemize}

    % Step 4
    Now, we can use Proposition~\ref{proposition_arithmetic_Bertini_closed_point_with_good_heights} for the ample arithmetic $\QQ$-divisor of $C^\infty$-type $\ov{\cB}'$ and arithmetic $\RR$-divisors of $C^0$-type $f^* \ov{\cD}_0, \dots, f^* \ov{\cD}_n$ to find a closed point $x' \in X' \setminus f^{-1}(Z)$ (where $X' := \cX' \otimes \QQ$) which satisfies
    \[ \Big| h_{f^* \ov{\cD}_i}(x') - \frac{\ov{\cB}'^d \cdot f^* \ov{\cD}_i}{\vol(\cB'_{\QQ})} \Big| < \varepsilon \textnormal{ for all } i=0, \dots, n. \]
    %Moreover, by enlarging $Z$, we can without loss of generality assume that $x'$ is in an open subset of $\cX'$ on which $f$ is an isomorphism. 
    By putting the above equations and inequalities together we get that for any natural $i$ smaller or equal $n$ we have
    \begin{equation}~\label{equation_height_is_C_times_l}
        \Big| h_{f^* \ov{\cD}_i}(x') - \frac{\avol(\ov{\cB})}{\vol(\cB'_{\QQ}) l(\ov{\cB})} \cdot l(\ov{\cD}_i) \Big| < \Bigl(1+\frac{A}{\vol(\cB'_{\QQ})}\Bigl) \varepsilon < 2 \varepsilon.
    \end{equation}
    Let $C := \frac{\avol(\ov{\cB})}{\vol(\cB'_{\QQ})l(\ov{\cB})}$. Applying the last inequality for $i=0$ we get
    %\[ |\log(2) - C \log(2)| = |h_{f^* \ov{\cD}_0}(x') - C \cdot l(\ov{\cD}_0)| < 2 \varepsilon, \]
    %hence $|C-1| < 2 \varepsilon / \log(2)$.
    \[ |\log(2) - C (\log(2)+\varepsilon')| = |h_{f^* \ov{\cD}_0}(x') - C \cdot l(\ov{\cD}_0)| < 2 \varepsilon, \]
    because of Lemma~\ref{lemma_height_with_respect_to_divisor_max_div_two_and_zero}. Thus by Lemma~\ref{lemma_basic_inequalities} we get that $|C-1| < 3 \varepsilon / \log(2)$. Recall that $\max(|l(\ov{\cD}_0)|, \dots, |l(\ov{\cD_n})|) < 2M$. By using the bound on $|C-1|$ in Equation~(\ref{equation_height_is_C_times_l}) we get
    \[ |h_{f^* \ov{\cD}_i}(x') - l(\ov{\cD}_i)| < \Bigl( 2 + \frac{6M}{\log(2)} \Bigl) \varepsilon. \]
    % This point coincides with the pushforward of $x'$ by $f$ treated as a cycle, because $f$ is an isomorphism in the neighbourhood of $x'$. 
    Let $x = f(x') \in X$. By Lemma~\ref{lemma_behaviour_of_height_for_birational_maps_and_pullbacks_proj_formula} we have $h_{f^* \ov{\cD}_i}(x') = h_{\ov{\cD}_i}(x)$. Since $M$ is independent of $\varepsilon$ (it only depends on the initial data), this finishes the proof of the theorem.
\end{proof}

% maybe if I choose \varepsilon' multiplicatively (l(\ov{\cD}_0)=(1+\varepsilon') \log(2) - then I don't need this lemma, but just common sense.
\begin{lemma}~\label{lemma_basic_inequalities}
    Let $C, \varepsilon>0$, $\varepsilon > \varepsilon'\geq 0$ and assume that $|1-C(1+\varepsilon')|<2\varepsilon$. Then $|C-1| < 3\varepsilon$.
\end{lemma}
\begin{proof}
    Skipped.
\end{proof}

Note that from the proof of Theorem~\ref{theorem_main_arithmetic_approximation_theorem} it follows that $l$ does not have to be defined on the whole space $\Adiv_{\RR}(\cX)$. However, by the M.Riesz extension theorem it does not give a more general statement, see Subsection~\ref{subsection_ess_infimum_and_GVF_functionals} and Corollary~\ref{corollary_main_arithmetic_approximation_corollary}.

\section{Applications in globally valued fields}~\label{section_3}

\subsection{Lattice divisors and globally valued field functionals}
Fix a finitely generated field extension $\QQ \subset \Field$ for this subsection.
%Let $\QQ \subset \Field$ be a finitely generated field extension for the rest of this subsection.

\begin{definition}~\label{definition_F_models_and_global_divisors_on_F}
    We call $\cX$ an \textit{$\Field$-model}, if it is a generically smooth, normal arithmetic variety together with an isomorphism $\kappa(\cX) \simeq \Field$. The system of $\Field$-models together with morphisms respecting isomorphisms of function fields with $\Field$ is filtered. We use the following notations for direct limits of various groups of divisors with respect to this filtering
    \[ \Adiv_{\KK}(\Field) := \varinjlim \Adiv_{\KK}(\cX) \textnormal{ for } \KK = \ZZ, \QQ \textnormal{ or } \RR; \]
    \[ \Rdiv_{\KK}(\Field) := \varinjlim \Rdiv_{\KK}(\cX) \textnormal{ for } \KK = \ZZ, \QQ \textnormal{ or } \RR; \]
    \[ \aN^1(\Field) := \varinjlim \aN^1(\cX). \]
    We write $\Adiv_{\KK}^0(\cX)$ for the group of arithmetic $\KK$-divisors of $C^0$-type such that the underlying divisor is $0$. This space can be naturally identified with the set of $C^0$-type functions on $\cX(\CC)$. We define
    \[ \Adiv_{\KK}^0(\Field) := \varinjlim \Adiv_{\KK}^0(\cX) \textnormal{ for } \KK = \ZZ, \QQ \textnormal{ or } \RR. \]
    Note that $\Adiv_{\KK}^0(\Field)$ has a norm induced by the supremum norms on $C^0$-functions on (analytifications of) $\Field$-models. We use notation $\|\cdot\|_{\sup}$ for this norm.
\end{definition}
% Give a better statement about divisors before!
\begin{comment}
\begin{lemma}~\label{lemma_every_Cartier_divisor_has_a_metric}
    Let $\cD$ be a Cartier divisor on an $\Field$-model $\cX$. There is a $\cD$-Green function $g$ on $\cX$. Moreover, one can take $g$ so that $\ov{\cD} = (\cD, g)$ is of $C^{\infty}$-type.
\end{lemma}
\begin{proof}
    
\end{proof}    
\end{comment}

\begin{remark}
    Note that there is a natural exact sequence
    \[ 0 \longrightarrow \Rdiv_{\RR}(\Field) \longrightarrow \Adiv_{\RR}(\Field) \longrightarrow \aN^1(\Field) \longrightarrow 0. \]
\end{remark}

\begin{definition}~\label{definition_lattice_structure_and_divisors_on_F_and_on_F_models}
    Let $\ov{\cD}, \ov{\cE} \in \Adiv_{\QQ}(\Field)$ and assume that $\cX$ is an $\Field$-model such that $\ov{\cD}, \ov{\cE} \in \Adiv_{\QQ}(\cX)$. We define the \textit{lattice infimum $\ov{\cD} \wedge \ov{\cE}$} of $\ov{\cD} = (\cD, g)$ and $\ov{\cE} = (\cE, h)$ in the following way.
    \begin{enumerate}
        \item Assume that $\cD, \cE$ are effective Cartier divisors on $\cX$ and that $\cD \cap \cE$ is a Cartier divisor. Then we put $\ov{\cD} \wedge \ov{\cE} := (\cD \cap \cE, \min(g, h))$ as an element of $\Adiv_{\QQ}(\Field)$.
        \item Assume that $\cD, \cE$ are effective Cartier divisors on $\cX$. Let $\cX'$ be an $\Field$-model over $\cX$ such that the intersection of the pullbacks of $\cD, \cE$ to $\cX'$ is a Cartier divisor. Then use the above definition on $\cX'$.
        \item Assume that $\cD, \cE$ are Cartier divisors on $\cX$. By Lemma~\ref{lemma_every_R_divisor_comb_of_effective_and_a_principal} there exists an effective Cartier divisor $\cA$ on $\cX$ such that $\cD+\cA, \cE+\cA$ are effective. Equipp $\cA$ with an $\cA$-Green function $a$ on $\cX$. We can do this by Lemma~\ref{lemma_every_Cartier_divisor_can_be_equipped_with_a_Green_function}. Define $\ov{\cD} \wedge \ov{\cE} := ((\ov{\cD} + \ov{\cA}) \wedge (\ov{\cE} + \ov{\cA})) - \ov{\cA}$ for $\ov{\cA} = (\cA, a)$.
        \item Assume that $\cD, \cE$ are $\QQ$-Cartier divisors on $\cX$. Let $n$ be a natural number such that $n \cD, n \cE$ are Cartier divisors on $\cX$. Put $\ov{\cD} \wedge \ov{\cE} := \frac{1}{n} (n \ov{\cD} \wedge n \ov{\cE})$.
    \end{enumerate}
    %We define the \textit{lattice supremum} of $\ov{\cD}, \ov{\cE} \in \Adiv_{\QQ}(\Field)$ by $\ov{\cD} \vee \ov{\cE} := - ( (-\ov{\cD}) \wedge (-\ov{\cE}) )$. 
    We say that $\ov{\cD} \wedge \ov{\cE}$ \textit{can be calculated on $\cX$}, if we don't have to pass to a model $\cX'$ above $\cX$ to define it. More precisely, this means that there is natural number $n$ and an effective Cartier divisor $\cA$ on $\cX$ such that $n \ov{\cD}, n \ov{\cE}$ are arithmetic $\ZZ$-divisors of $C^0$-type and the intersection $(n \cD+\cA) \cap (n \cE+\cA)$ is a Cartier divisor on $\cX$.
    
    Moreover, if $\cD, \cE$ are $\QQ$-Cartier divisors on $\cX$, we define $\cD \wedge \cE$ using the same procedure. 
\end{definition}

\begin{remark}~\label{remark_if_wedge_can_be_calculated_on_X_it_can_be_calculated_over_it}
    Let $\cX, \cX'$ be $\Field$-models with a morphism $f:\cX' \to \cX$, and let $\ov{\cD}, \ov{\cE} \in \Adiv_{\QQ}(\cX)$. If $\ov{\cD} \wedge \ov{\cE}$ can be calculated on $\cX$, then $f^* \ov{\cD} \wedge f^* \ov{\cE}$ can be calculated on $\cX'$. Indeed, it's enough to take a natural number $n$ and a Cartier divisor $\cA$ on $\cX$ as in Definition~\ref{definition_lattice_structure_and_divisors_on_F_and_on_F_models}, and consider the pullback $f^* \cA$ on $\cX'$.
\end{remark}

\begin{lemma}~\label{lemma_GVF_lattice_structure_well_defined}
    The lattice structure on $\Adiv_{\QQ}(\Field)$ is well defined and does not depend on the choices in the definition.
\end{lemma}
\begin{proof}
    Fix the notation from Definition~\ref{definition_lattice_structure_and_divisors_on_F_and_on_F_models}. We check point by point, that the procedure does not depend on the choices made.
    \begin{enumerate}
        \item Assume that $\cD, \cE$ are effective Cartier divisors on $\cX$ such that $\cD \cap \cE$ is an effective Cartier divisor. We need to show that $\min(g,h)$ is a $(\cD \cap \cE)$-Green function on $\cX$. Let $\Spec A \subset \cX$ be an open subset where $\cD, \cE, \cD \cap \cE$ are given respectively by $d, e, f$ and assume that $x, y, d', e' \in A$ are such that $d=fd', e=fe', f=xd+ye$. Note that $f \cdot (1-(xd'+ye'))=0$, so $f=0$ or $xd'+ye'=1$, as $A$ is an integral domain. If $f=0$, then $d=e=0$ and this cannot happen. We must show that the function
        \[ \min(g(p),h(p)) + \log |f|^2(p) \]
        extends to a $C^0$-type function on $(\Spec A)(\CC)$. By the assumption we know that functions
        \[ g(p) + \log |d|^2(p) \textnormal{ and } h(p) + \log |e|^2(p) \]
        extend to a $C^0$-type functions on $(\Spec A)(\CC)$. Note that
        \[ g(p) + \log |d|^2(p) = g(p) + \log |f|^2(p) + \log |d'|^2(p), \]
        \[ h(p) + \log |e|^2(p) = h(p) + \log |f|^2(p) + \log |e'|^2(p). \]
        Thus it is enough to show that for every $p \in (\Spec A)(\CC)$ we have $|d'|(p) \neq 0$ or $|e'|(p) \neq 0$. However, $xd'+ye'=1$, so if $d'(p) = e'(p) =0$, then $0=x(p) d'(p) + y(p) e'(p)=1$ which gives a contradiction.
        \item Assume that $\cD, \cE$ are effective Cartier divisors on $\cX$ and let $\cX', \cX''$ be $\Field$-models over $\cX$ where intersections of pullbacks of $\cD, \cE$ are Cartier divisors. Without loss of generality $\cX''$ is over $\cX'$. Consider open sets $\Spec C \to \Spec B \to \Spec A$ inside $\cX'' \to \cX' \to \cX$ respectively, and assume that $\cD, \cE$ have local equations $d, e \in A$ on $\Spec(A)$. Assume also, that on $\Spec B$ we have $(d,e)=(f)$ for some $f \in B$. Let $x, y, d', e' \in B$ be elements such that $d=fd', e=fe', f=xd+ye$. The same equations hold in $C$ (for pullbacks), so $(d,e)$ calculated in $\Spec B$ and pulled back to $\Spec C$ is the same as $(d,e)$ calculated in $\Spec C$. Note that to define the Green function for the intersection of pullbacks of $\cD, \cE$ on $\cX''$ one can take the minimum of pullbacks of $g, h$ to $\cX''$, or the pullback of the minimum of pullbacks of $g, h$ on $\cX'$. 
        \item Assume that $\cD, \cE$ are Cartier divisors on $\cX$ and let $\cA, \cB$ be effective Cartier divisors on $\cX$ with $\cD+\cA, \cE+\cA$ effective. Replace $\cX$ by an $\Field$-model $\cX'$ over $\cX$ where $(\cD+\cA) \cap (\cE+\cA)$ is an effective Cartier divisor. To prove independence of $\ov{\cA}$ in the definition of the lattice infimum (point 3.) it's enough to show that for any effective Cartier divisor $\cB$ on $\cX'$ we have
        \[((\cD+\cA) \cap (\cE+\cA)) - \cA = ((\cD+\cA+\cB) \cap (\cE+\cA+\cB)) - (\cA+\cB).\]
        To ease the notation, replace $\cD+\cA, \cE+\cA$ by $\cD, \cE$ respectively. We want to prove that $\cD \cap \cE =((\cD+\cB) \cap (\cE+\cB)) - \cB$ assuming that $\cD, \cE, \cD \cap \cE, \cB$ are effective Cartier divisors on $\cX$. Let $\Spec A \subset \cX$ be an open subset where the effective Cartier divisors from the previous sentence are given respectively by $d, e, f, b \in A$. Also, assume that $\Spec A$ is chosen so that there are $x, y, d', e' \in A$ such that $d=fd', e=fe', f=xd+ye$. Note that then $bd=bfd', be=bfe', bf=xbd+ybe$, hence $(bf) = (be,bd)$, which gives the desired result. The fact that the $(\cD \cap \cE)$-Green function on $\cX$ does not depend on $\ov{\cA}$ follows from the fact that for real valued functions $g, h, j$ we have $\min(g+j, h+j) = \min(g,h)+j$.
        \item At last, assume that $\cD, \cE$ are $\QQ$-Cartier and let $n \neq 0$ be a natural number such that $n \cD, n \cE$ are Cartier divisors on $\cX$. Fix a natural number $m \neq 0$. We need to show that using the point 3. of Definition~\ref{definition_lattice_structure_and_divisors_on_F_and_on_F_models} we get $(n \cD) \wedge (n \cE) = \frac{1}{m} ((mn \cD)\wedge(mn \cE))$. Replace $\cD, \cE$ by $n \cD, n \cE$ respectively, to assume that $n=1$. By the previous points in this lemma, it is enough to prove that if $\cD, \cE$ are effective Cartier divisors on $\cX$ with $\cD \cap \cE$ effective Cartier, then $(m \cD) \cap (m \cE) = m (\cD \cap \cE)$. Let $\Spec A \subset \cX$ be an open subset where $\cD, \cE, \cD \cap \cE$ are given respectively by $d, e, f$ and assume that $x, y, d', e' \in A$ are such that $d=fd', e=fe', 1=xd'+ye'$ (see point 1. in this lemma). Note that
        \[ 1 = (xd' + ye')^{2m} = \sum_{i=0}^{2m} \binom{2m}{i} x^i y^{2m-i} (d')^i (e')^{2m-i} \in ((d')^m, (e')^m). \]
        Thus, from the fact that $d^m = f^m (d')^m, e^m = f^m (e')^m$, we get that $(d^m, e^m) = (f^m)$, which finishes the proof. The fact that the $(\cD \cap \cE)$-Green function on $\cX$ does not depend on $m$ follows from the fact that $\min(mg, mh) = m \min(g,h)$.
    \end{enumerate}
\end{proof}

\begin{lemma}~\label{lemma_properties_of_the_GVF_pairing}
    Let $v$ be a valuation on $\Field$. If $\ov{\cD} \in \Adiv_{\RR}(\Field)$, then the value $\GVFpairing(v, \ov{\cD})$ does not depend on the $\Field$-model on which $\ov{\cD}$ is defined. Moreover, the function $\GVFpairing(v, -)$ has the following properties:
    \begin{enumerate}
        \item it is $\RR$-linear,
        \item for every $\ov{\cD}, \ov{\cE} \in \Adiv_{\QQ}(\Field)$ we have $\GVFpairing(v, \ov{\cD} \wedge \ov{\cE}) = \min(\GVFpairing(v, \ov{\cD}), \GVFpairing(v, \ov{\cE}))$.
    \end{enumerate}
\end{lemma}
\begin{proof}
    The fact that $\GVFpairing(v, \ov{\cD})$ does not depend on the $\Field$-model follows from the definition of pullback, see Remark~\ref{remark_pullbacks_of_arithmetic_R_divisors_of_C_0_type}. Linearity is clear. To prove the second point, by Lemma~\ref{lemma_GVF_lattice_structure_well_defined} and $\RR$-linearity we can reduce to the following situation. Assume that $\ov{\cD}, \ov{\cE}$ are arithmetic $\ZZ$-divisors on $\cX$ with $\cD, \cE, \cD \cap \cE$ effective Cartier. If $v$ is Archimedean, the statement follows from the definition, so assume that $v$ is non-Archimedean. Let $p = \supp(v) \in \cX$ and let $\Spec A \subset \cX$ be an open neighbourhood of $p$ where $\cD, \cE, \cD \cap \cE$ are given by $d, e, f \in A$ respectively. Assume additionally that $x, y, d', e' \in A$ are such that $d=fd', e=fe', 1=xd'+ye'$ (see point 1. in Lemma~\ref{lemma_GVF_lattice_structure_well_defined}). We need to show that $v(f) = \min(v(d),v(e))$. Note that
    \[ 0 = v(1) = v(xd'+ye') \geq \min(v(x) + v(d'), v(y)+v(e')) \geq \min(v(d'), v(e')) \]
    and
    \[ v(d) = v(f) + v(d'), v(e) = v(f) + v(e'). \]
    Note that if $v(d'), v(e') > 0$, then $0 \geq \min(v(d'),v(e')) > 0$, which gives a contradiction. Thus one of $v(d'),v(e')$ has to be zero, and we get the result.
\end{proof}

\begin{lemma}~\label{lemma_local_degree_on_minimum_is_minimum}
    Let $\cX$ be an $\Field$-model with a closed point in the generic fiber $x \in X = \cX_{\QQ}$. Let $\ov{\cD}, \ov{\cE}, \ov{\cA}$ be arithmetic $\ZZ$-divisors of $C^0$-type on $\cX$ such that $\cA, \cA+\cD, \cA+\cE$ are effective and $(\cA+\cD) \cap (\cA+\cE)$ is principal, so that $\ov{\cD} \wedge \ov{\cE}$ can be calculated on $\cX$. Assume that $x$ does not belong to the support of any of $\cD, \cE, \cA, \cA+\cD, \cA+\cE$ and let $v$ be a valuation on $\kappa(x)$. Then
    \[ \adeg_v( \ov{\cD} \wedge \ov{\cE}|\ov{\{x\}}) = \min(\adeg_v( \ov{\cD}|\ov{\{x\}}), \adeg_v(\ov{\cE}|\ov{\{x\}})). \]
\end{lemma}
\begin{proof}
    We skip this proof. It is almost the same as the proof of Lemma~\ref{lemma_properties_of_the_GVF_pairing}.
\end{proof}

\begin{definition}
    We denote by $\Ldiv_{\QQ}(\Field)$ the $\QQ$-vector space generated by $\QQ$-vector and lattice operations from principal arithmetic $\ZZ$-divisors $\adiv(f)$, for $f \in \Field$. We call these \textit{lattice divisors} on $\Field$. For an $\Field$-model $\cX$ we denote by $\Ldiv_{\QQ}(\cX)$ the inverse image of lattice divisors on $\Field$ via the natural map $\Adiv_{\QQ}(\cX) \to \Adiv_{\QQ}(\Field)$. We denote by $\Ldiv_{\QQ}^0(\Field), \Ldiv_{\QQ}^0(\cX)$ the subspaces where the corresponding divisor is $0$. 
    % If $\ov{\cD}_1, \dots, \ov{\cD}_m$ are arithmetic $\QQ$-divisors of $C^0$-type on $\cX$ and $t$ is a tropical $\QQ$-polynomial (i.e. an expression using $\QQ$-vector spaces and lattice operations, see Definition~\ref{definition_GVF_Globally_valued_fields}), we say that \textit{the inductive definition of $t(\ov{\cD}_1, \dots, \ov{\cD}_m)$ can be calculated on $\cX$}, if 
\end{definition}

\begin{definition}~\label{definition_GVF_functional}
    Let $l$ be a linear functional on a (real or rational) vector subspace of $\Adiv_{\RR}(\Field)$ (resp. $\QQ$ or $\RR$-linear). We call $l$ a \textit{globally valued field functional} (abbreviated GVF functional), if it satisfies the following properties:
    \begin{itemize}
        \item it maps principal $\RR$-divisors of $C^0$-type in its domain to $0$,
        \item it is non-negative on effective arithmetic $\RR$-divisors of $C^0$-type in its domain.
    \end{itemize}
    Moreover, we call $l$ \textit{normalized}, if its domain contains a pullback of a non-principal arithmetic $\RR$-divisor of $C^0$-type $\ov{\cD}$ on $\Spec(\ZZ)$ and $l$ sends it to $\adeg(\ov{\cD})$. Here we mean the arithmetic degree from Theorem~\ref{theorem_Ikoma_extension_intersection_product_arithmetic_R_divisors} or equivalently the arithmetic degree from Definition~\ref{definition_height_of_a_point_w_r_t_arithmetic_R_divisor}, for the arithmetic variety $\Spec(\ZZ)$.
\end{definition}

\begin{proposition}~\label{proposition_lattice_0_divisors_are_dense_within_artihmetic_zero_divisors}
    Let $\cX$ be an $\Field$-model. Then $\Ldiv_{\QQ}^0(\cX)$ is a dense subset of $\Adiv_{\RR}^0(\cX)$ with respect to the supremum norm.
\end{proposition}
\begin{proof}
    Identify $\Adiv_{\RR}^0(\cX)$ with continuous real valued functions on $\cX(\CC)/F_{\infty}$. By the Stone–Weierstrass theorem (max-closed), it is enough to check the following:
    \begin{itemize}
        \item a constant non-zero function is in $\Ldiv_{\QQ}^0(\cX)$;
        \item if $g \in \Ldiv_{\QQ}^0(\cX)$, then $ag \in \Ldiv_{\QQ}^0(\cX)$ for $a \in \QQ$;
        \item if $g, h \in \Ldiv_{\QQ}^0(\cX)$, then $g+h, \max(g,h) \in \Ldiv_{\QQ}^0(\cX)$;
        \item the algebra $\Ldiv_{\QQ}^0(\cX)$ separates points of $\cX(\CC)/F_{\infty}$.
    \end{itemize}
    For a constant non-zero function we can take $\adiv(2) \wedge \adiv(1)$. The second and third points are automatic. Now we prove the last point. Fix an embedding $\cX \subset \PP_{\ZZ}^n$ and assume that $\cX$ is not contained in any hyperplane. Pick 
    \[ x, y \in \cX(\CC) \subset \PP_{\ZZ}^n(\CC) = \{ [z_0 : \dots : z_n] |(\exists i=0, \dots, n)(z_i \neq 0) \} \]
    which have different conjugation classes. Assume that $x_0 = y_0 = 1$ (possibly after a change of coordinates of $\PP_{\ZZ}^n$). Let $i, i'$ be indexes among $1, \dots, n$ such that $x_i \neq y_i$ and $\ov{x_{i'}} \neq y_{i'}$. Note that we can without loss of generality assume that $i' = i$. Indeed, if $\ov{x_i} = y_i$ and $x_{i'} = y_{i'}$, then we can change the coordinates of $\PP_{\ZZ}^n$ so that the new $i$'th coordinate is a sum of the $i$'th and $i'$'th coordinate, and other coordinates stay the same. The new $i$'th coordinates of $x$ and $y$ become $x_i + x_{i'}, \ov{x_i} + x_{i'}$ respectively, and they have different conjugation classes as $x_i,x_{i'} \in \CC \setminus \RR$.
    
    Let $X_0, \dots, X_n$ be the natural projective coordinates on $\PP_{\ZZ}^n$ and denote by $f_j$ the restriction of $\frac{X_j}{X_0}$ to $\cX$, for each $j=0, \dots, n$. This makes sense as $\cX$ is not contained in any hyperplane. Let $\ov{\cE}$ be the lattice divisor on $\Field$ defined by the formula
    \[ \ov{\cE} := \adiv(1) \wedge \adiv(f_1) \wedge \dots \wedge \adiv(f_{i-1}) \wedge \adiv(f_{i+1}) \wedge \dots \wedge \adiv(f_n). \]
    %\[ \ov{\cE} := \adiv(1) \wedge \adiv(\frac{X_1}{X_0})|_{\cX} \wedge \dots \wedge \adiv(\frac{X_{i-1}}{X_0})|_{\cX} \wedge \adiv(\frac{X_{i+1}}{X_0})|_{\cX} \wedge \dots \wedge \adiv(\frac{X_n}{X_0})|_{\cX}, \]
    %where $X_0, \dots, X_n$ are the natural projective coordinates on $\PP_{\ZZ}^n$ and the restrictions to $\cX$ denote pullbacks via $\cX \subset \PP_{\ZZ}^n$ (see Remark~\ref{remark_pullbacks_of_arithmetic_R_divisors_of_C_0_type}). This makes sense as $\cX$ is not contained in any hyperplane. 
    Pick natural numbers $a, b$ such that the following hold:
    \begin{enumerate}
        \item $\Big| \frac{x_i+a}{x_i+b} \Big| \neq \Big| \frac{y_i+a}{y_i+b} \Big|$;
        \item for $c = a$ or $c=b$ we have 
        \[ |x_i+c| > 1, |x_1|, \dots, |x_n|; \]
        \[ |y_i+c| > 1, |y_1|, \dots, |y_n|. \]
    \end{enumerate}
    Such exists by Lemma~\ref{lemma_complex_numbers_with_diff_conj} with $N$ big enough. Consider the following lattice divisor on $\Field$
    \[ \ov{\cD} := - (\ov{\cE} \wedge \adiv(f_i + a)) + (\ov{\cE} \wedge \adiv(f_i+b)). \]
    %\[ \ov{\cD} := - (\ov{\cE} \wedge \adiv(\frac{X_i}{X_0} + a)|_{\cX}) + (\ov{\cE} \wedge \adiv(\frac{X_i}{X_0} + b)|_{\cX}). \]
    Let $\cX'$ be an $\Field$-model over $\cX$ with $\ov{\cD} \in \Ldiv_{\QQ}(\cX')$. Let $v$ be a non-Archimedean valuation on $\Field$.
    \begin{claim}
        The equality $\GVFpairing(v, \ov{\cD}) = 0$ holds.
    \end{claim}
    \begin{proof}[Proof of the claim]
        By Lemma~\ref{lemma_properties_of_the_GVF_pairing} it is enough to show that
        \[ \min( \{v(f_j):j \neq i\} \cup \{v(f_i + a)\} ) = \min( \{v(f_j):j \neq i\} \cup \{v(f_i + b)\} ), \]
        %\[ \min( \{v(\frac{X_j}{X_0}):j \neq i\} \cup \{v(\frac{X_i}{X_0} + a)\} ) = \min( \{v(\frac{X_j}{X_0}):j \neq i\} \cup \{v(\frac{X_i}{X_0} + b)\} ). \]
        where $j$ above varies form $0$ to $n$ with $f_0 = 1$. This follows from the ultrametric inequality for non-Archimedean valuations.
    \end{proof}
    In particular, the Weil decomposition of $\cD$ is trivial, hence $\cD = 0$ as $\cX'$ is normal. Thus $\ov{\cD} = (0, g)$ for some $C^0$-type function $g$ on $\cX'(\CC)$. Consider the function $H$ defined by the formula
    \[ - \min( \{- \log|z_j|:j \neq i\} \cup \{-\log|z_i+az_0|\} ) + \min( \{-\log|z_j|:j \neq i\} \cup \{-\log|z_i+bz_0|\} ), \]
    for $[z_0 : \dots : z_n] \in \PP_{\ZZ}^n(\CC)$. Note that it is a $C^0$-type function on $\PP_{\ZZ}^n(\CC)$ (it does not depend on the representative of a point) and it restricts to a $C^0$-type function $h := H|_{\cX(\CC)}$ on $\cX(\CC)$. By Lemma~\ref{lemma_properties_of_the_GVF_pairing} and Definition~\ref{definition_GVF_pairing}, for points $p \in \cX'(\CC)$ factoring through the inclusion $\Spec(\Field) \to \cX'$ (note that those are the same for $\cX$ and $\cX'$), we have
    \[ g(p) = - \min( \{-\log|f_j|(p):j \neq i\} \cup \{-\log|f_i + a|(p)\} ) \]
    \[  + \min( \{-\log|f_j|(p):j \neq i\} \cup \{-\log|f_i + b|(p)\} ) \]
    \[ = - \min( \{- \log|z_j|(p):j \neq i\} \cup \{-\log|z_i+az_0|(p)\} ) \]
    \[ + \min( \{-\log|z_j|(p):j \neq i\} \cup \{-\log|z_i+bz_0|(p)\} ) \]
    \[ = H(p) = h(p), \]
    where the second inequality holds as $f_i = \frac{X_i}{X_0}|_{\cX}$ and for $p$ as above $z_0(p) \neq 0$. Since $g, h$ are continuous, $g$ must be a pullback of $h$ to $\cX'(\CC)$, as those agree on a dense subset. Thus $(0,h) \in \Ldiv_{\QQ}(\cX)$. Moreover, by the choice of $a, b$ we get that
    \[ h(x) = \log \bigg| \frac{x_i+a}{x_i+b} \bigg| \neq \log \bigg| \frac{y_i+a}{y_i+b} \bigg| = h(y). \]
    This finishes the proof.
\end{proof}

\begin{lemma}~\label{lemma_complex_numbers_with_diff_conj}
    Let $x \neq y \in \CC$ be such that $\Ima(x), \Ima(y) \geq 0$. Let $N$ be a natural number. Then there are natural numbers $a, b > N$ such that
    \[ \bigg| \frac{x+a}{x+b} \bigg| \neq \bigg| \frac{y+a}{y+b} \bigg| \]
\end{lemma}
\begin{proof}
    Skipped.
\end{proof}

\begin{comment}
\begin{lemma}~\label{lemma_approximating_Green_functions_by_lattice_ones}
    Let $\cX$ be an $\Field$-model and let $h$ be a $C^0$-type function on $\cX(\CC)$. Fix $\varepsilon > 0$. Then there exists a pair of continuous functions $g_1, g_2:\cX(\CC) \to \RR$ such that:
    \begin{itemize}
        \item $g_1 \leq h \leq g_2$,
        \item $0 \leq g_2 - g_1< \varepsilon$,
        \item $(0,g_1), (0, g_2) \in \Ldiv_{\QQ}(\cX')$.
    \end{itemize}
    Here we identify $h$ with its pullback to $\cX'(\CC)$.
\end{lemma}
\begin{proof}
    
\end{proof}
\end{comment}

\begin{lemma}~\label{lemma_GVF_functionals_on_R_and_Q_divisors_are_the_same}
    Every GVF functional
    \[ \Adiv_{\QQ}(\Field) \to \RR \]
    extends uniquely to a GVF functional
    \[ \Adiv_{\RR}(\Field) \to \RR. \]
\end{lemma}
\begin{proof}
    By Lemma~\ref{lemma_divisors_here_are_the_same_is_in_Ikoma_up_to_a_principal} we get the uniqueness. For existence, use Proposition~\ref{proposition_Moriwaki_Proposition_2_4_2} to get surjectivity of the natural map $\Adiv_{\QQ}(\Field) \otimes \RR \to \Adiv_{\RR}(\Field)$ and for the fact that the real span of the image of the effective cone of $\Adiv_{\QQ}(\Field)$ spans the effective cone in $\Adiv_{\RR}(\Field)$.
\end{proof}

\begin{lemma}~\label{lemma_every_divisor_has_a_lattice_metric}
    Let $\cD$ be a $\QQ$-Cartier divisor on an $\Field$-model $\cX$. Then there exists an $\Field$-model $\cX'$ over $\cX$ and an a $\cD$-Green function $g$ on $\cX'$ such that $(\cD, g) \in \Ldiv_{\QQ}(\cX')$.
\end{lemma}
\begin{proof}
    By presenting $\cD$ as a rational number times a difference of two very ample Cartier divisors, we can reduce to the case where $\cD$ is very ample. In particular it is globally generated. Pick a basis $f_1, \dots, f_n \in H^0(\cX, \cD)$ of rational functions with poles at most given by $\cD$.
    \begin{claim}
        $\cD + \bigwedge_{i=1}^n \cdiv(f_i) = 0$.
    \end{claim}
    \begin{proof}[Proof of the claim]
        Recall that for a Cartier divisor $\cD$ by rational functions with poles at most given by $\cD$ we mean
        \[ H^0(\cX,\cD) = \{0\} \cup \{ f \in \kappa(\cX)^{\times} \ | \ \cD + \cdiv(f) \geq 0 \}. \]
        Note that $\{ \cD + \cdiv(f) \ | \ f \in H^0(\cX,\cD) \}$ is the set of effective Cartier divisors linearly equivalent to $\cD$. The fact that $\cO(\cD)$ is globally generated implies that for every $p \in \cX$ there is $\cE \in \{ \cD + \cdiv(f_i) \ | \ i=1, \dots, n \}$ such that $p \not\in \cE$. Thus the ideal defining $\cE$ is $(1)$ at $p$. Since $p$ was arbitrary, we get that the sum of ideals of effective Cartier divisors $\cE_i = \cD + \cdiv(f_i)$ is $(1)$. By the definition this means that (possibly passing to a blowup $\cX'$ of $\cX$)
        \[ 0 = \bigwedge_{i=1, \dots, n} (\cD + \cdiv(f_i)) = \cD + \bigwedge_{i=1, \dots, n} \cdiv(f_i),\]
        which finishes the proof of the claim.
    \end{proof}
    Thus, we can put $\ov{\cD} := - \bigwedge_{i=1}^n \adiv(f_i)$ and this gives a desired $\cD$-Green function $g$ on $\cX'$.
\end{proof}

\begin{proposition}~\label{proposition_GVF_functionals_extend_uniquely_to_arithmetic_R_divisors}
    Every GVF functional
    \[ \Ldiv_{\QQ}(\Field) \to \RR \]
    extends uniquely to a GVF functional
    \[ \Adiv_{\RR}(\Field) \to \RR. \]
\end{proposition}
\begin{proof}
    Let $l:\Ldiv_{\QQ}(\Field) \to \RR$ be a GVF functional. By Lemma~\ref{lemma_GVF_functionals_on_R_and_Q_divisors_are_the_same} it is enough to show that $l$ has a unique extension to $\Adiv_{\QQ}(\Field)$. Pick $\ov{\cD} \in \Adiv_{\QQ}(\Field)$ and let $\cX$ be a generically smooth, normal, arithmetic variety with $\kappa(\cX) \simeq \Field$ such that $\ov{\cD} \in \Adiv_{\QQ}(\cX)$. Let $\ov{\cD} = (\cD, g)$ on $\cX$. By Lemma~\ref{lemma_every_divisor_has_a_lattice_metric} (possibly replacing $\cX$ by an $\Field$-model over it) there is a $\cD$-Green function $g'$ on $\cX$ with $(\cD, g') \in \Ldiv_{\QQ}(\cX)$. The value of $l$ on $(\cD, g')$ is determined, so it is enough to define (and show uniqueness of) the value of the extension at $(0, g-g') \in \Adiv_{\QQ}(\cX)$. Let $h$ be the continuous extension of $g-g'$ to $\cX(\CC)$. By Proposition~\ref{proposition_lattice_0_divisors_are_dense_within_artihmetic_zero_divisors} there is a pair of Green functions for the $0$ divisor $g_1, g_2$ on $\cX$ such that:
    %By Lemma~\ref{lemma_approximating_Green_functions_by_lattice_ones} for every $\varepsilon > 0$ there is a pair of Green functions for the $0$ divisor $g_1, g_2$ on $\cX$ (possibly replacing $\cX$ be an $\Field$-model over it) such that:
    \begin{itemize}
        \item $g_1 \leq h \leq g_2$,
        \item $0 \leq g_2 - g_1< \varepsilon$,
        \item $(0,g_1), (0, g_2) \in \Ldiv_{\QQ}(\cX)$.
    \end{itemize}
    To find these Green functions, take $g_1$ approximating $h - \frac{1}{4} \varepsilon$ up to $\frac{1}{4} \varepsilon$ and similarly with $g_2$. Hence, any GVF extension of $l$ assigns a value between $l((0,g_1))$ and $l((0,g_2))$ to $(0,h)$. Moreover, 
    \[ |l((0,g_2))-l((0,g_1))| = |l((0, g_2-g_1))| < \varepsilon |l((0,1))|. \] 
    Take a sequence $\varepsilon_n \to 0$ and for each $\varepsilon_n$ pick $(g_1)_n, (g_2)_n$ as above, such that $(g_1)_n$ is increasing and $(g_2)_n$ is decreasing (one can do it since $\Ldiv_{\QQ}(\Field)$ is closed under maximum and minimum). For any GVF extension of $l$, we must have $l((0,h)) = \lim_n l((0,(g_i)_n)$ for $i=1,2$.

    This proves that if an extension exists, it has to be given by the above construction. To prove existence one can either show that the above definition gives a GVF functional, or use the Riesz extension theorem (i.e., Theorem~\ref{theorem_Riesz_extension_theorem}) to get one.
\end{proof}

\subsection{Globally valued fields as unbounded continuous logic structures}

\begin{definition}~\label{definition_GVF_Globally_valued_fields}
    A \textit{$\QQ$-tropical polynomial} is a term in the language of divisible ordered abelian groups, i.e., in $+, \min, 0, \alpha \cdot x$ for $\alpha \in \QQ$. We denote by $n_t$ the number of variables of a $\QQ$-tropical polynomial $t$. We also use the symbol $\wedge$ for minimum.
    
    Let $\Field$ be a countable field. A \textit{globally valued field structure} on $\Field$ (abbreviated GVF structure) is a collection of functions (continuous logic predicates) $R_t:(\Field^{\times})^{n_t} \to \RR$ indexed by $\QQ$-tropical polynomials $t$, that are also denoted by
    \[ R_t(\ov{a}) =: \int t(v(\ov{a})) dv, \]
    and which satisfy the following axioms:
    \begin{enumerate}
        \item $R_t$ are compatible with permutations of variables and dummy variables, see the remark below.
        \item (Linearity) $R_{t_1+t_2} = R_{t_1} + R_{t_2}, R_{\alpha t} = \alpha R_t$ for all $\QQ$-tropical polynomials $t_1, t_2, t$ and a rational number $\alpha$.
        \item (Local-global positivity) Assume that $\ov{a}$ is a tuple from $\Field^{\times}$, $t$ is a $\QQ$-tropical polynomial (of the corresponding arity), and for all valuations $v$ on $\Field$ we have $t(v(\ov{a})) \geq 0$. Then
        \[ \int t(v(\ov{a})) dv \geq 0. \]
        \item (Product formula) $\int v(x) dv = 0$ for all $x \in \Field^{\times}$.
    \end{enumerate}
    We call a field equipped with a GVF structure a \textit{globally valued field} (abbreviated GVF). Globally valued fields form a category where morphisms are embeddings of fields such that the predicates $R_t$ of the bigger field restrict to the ones of the smaller one. We call a GVF \textit{trivial}, if all predicates $R_t$ are zero.
\end{definition}

\begin{remark}
    Let us be more clear about the first point in Definition~\ref{definition_GVF_Globally_valued_fields}. Compatibility with permutations means that if $\sigma$ is a permutation of $\{ x_1, \dots, x_n \}$, $t$ is a $\QQ$-tropical polynomial, and $t' = t \circ \sigma$, then $R_{t'} = R_t \circ \sigma$. Compatibility with dummy variables means that if $t(x_1, \dots, x_n) = s(x_1, \dots, x_m)$ for some $m < n$, then $R_t(a_1, \dots, a_n) = R_s(a_1, \dots, a_m)$.
\end{remark}

\begin{definition}
    For a GVF $\Field$ and $a \in \Field^{\times}$ we define the \textit{height} of $a$ to be the number
    \[ \height(a) := \int - \min(v(a), 0) dv. \]
\end{definition}

\begin{remark}
    One can define a GVF structure on any field, moreover in a first order way (in the sense of unbounded continuous logic, see \cite{Ben_Yaacov_unbdd_cont_FOL}). To do so one should replace the ``Local-global positivity'' axiom scheme by the following one: if $\phi(\ov{x})$ is a formula in the ring language implying that non of $x_i$ from $\ov{x} = (x_1, \dots, x_n)$ is zero and such that
    \[ \VF \models (\forall \ov{x})( \phi(\ov{x}) \implies t(v(\ov{x})) \geq 0 ), \]
    then
    \[ (\forall \ov{a})(\phi(\ov{a}) \implies \int t(v(\ov{a})) dv \geq 0) \]
    is an axiom of globally valued field structures. Here $\VF$ is the common theory of valued fields, i.e., two sorted structures with a field sort, a sort for reals, and a predicate for minus logarithm of a norm. For the details on this approach, see \cite{GVF3}. The fact that for countable fields two definitions are equivalent follows from \cite[Lemma 2.5, Proposition 2.15, Remark 2.20]{GVF3}.
\end{remark}

\begin{remark}~\label{remark_GVF_structures_are_measures}
    The notation for continuous logic predicates $R_t$ is not arbitrary - it comes from a fact that for every GVF structure on a countable field $\Field$, there is an admissible measure $\mu$ on the space $\Val_{\Field}$ which is unique up to a renormalisation and satisfies
    \[ \int t(v(\ov{a})) dv = \int_{\Val_{\Field}} t(v(\ov{a})) d\mu(v), \]
    for all $\QQ$-tropical polynomials $t$ and tuples $\ov{a}$ from $\Field^{\times}$. For details on this approach and definitions of admissibility and renormalisation relation, see \cite{GVF3}.
\end{remark}

\begin{example}[Curves over a field]~\label{example_function_field_of_a_curve_GVF}
    Let $C$ be a smooth projective curve over a field $\smallk$. Then closed points of $C$ correspond to classes of valuations on $\smallk(C)$ trivial on $\smallk$. For a closed $p \in C$ let $\ord_p$ be the order of vanishing at $p$ valuation - its value ring is $\cO_{C,p} \subset \smallk(C)$. If $f \in \smallk(C)$, then the degree of the principal divisor generated by $f$ is zero, so we have
    \[ 0 = \deg \sdiv(f) = \sum_p \ord_p(f) \cdot [\kappa(p) : \smallk]. \]
    In other words, the measure $\mu = \sum_p \delta_{\ord_p} \cdot [\kappa(p) : \smallk]$ defines a GVF structure (as in Remark~\ref{remark_GVF_structures_are_measures}). In fact it is the unique (up to multiplication by a positive scalar) GVF structure that restricts to a trivial one on $\smallk$, see \cite[Lemma 12.1]{GVF2}. Thus if we take $C = \PP^1$ and $\kappa(\PP^1) = \smallk(t)$ we get the unique GVF structure trivial on $\smallk$ and with $\height(t) = 1$. By using the uniqueness for curves, this means that also on the algebraic closure $\ov{\smallk(t)}$ there is a unique GVF structure with $\height(t) = 1$, trivial on $\smallk$. We denote it by $\ov{\smallk(t)}[1]$. Note however, that globally valued fields $\ov{\smallk(t)}[1], \ov{\smallk(t)}[2]$ (where in the latter we multiply every continuous logic predicate by two) are isomorphic, as one can send $t \in \ov{\smallk(t)}[1]$ to $\sqrt{t} \in \ov{\smallk(t)}[2]$.
\end{example}

\begin{example}[Number fields]~\label{example_number_field_GVF}
    Let $K$ be a number field with the integer ring $\cO_K$. Then, for $f \in K^*$ we have the following product formula
    \[ \sum_{p \in \Spec(\cO_K)} \ord_p(f) \cdot \log \# \kappa(p) + \sum_{\sigma : K \to \CC} -\log |\sigma(f)| = 0, \]
    where in the first sum we only take closed points and the second sum is taken over embeddings of fields. This yields a GVF structure on $K$ induced (as in Remark~\ref{remark_GVF_structures_are_measures}) by the measure
    \[ \mu := \frac{1}{[K : \QQ]} \biggl( \sum_{p \in \Spec(\cO_K)} \delta_{\ord_p} \cdot \log \# \kappa(p) + \sum_{\sigma : K \to \CC} \delta_{-\log |\sigma(-)|} \biggr). \]
    In fact this GVF structure is unique up to multiplication by a positive scalar, see \cite[Proposition 3.3]{GVF3}. We denote it by $K[1]$. It extends the GVF structure on $\QQ$ given by the same formula. The composition of all finite extensions of $\QQ$ yields the \textit{standard} GVF structure on $\ov{\QQ}$ denoted $\ov{\QQ}[1]$. Note that this is the unique GVF structure on $\ov{\QQ}$ with $\height(2) = \log(2)$. By uniqueness, this GVF structure is Galois-invariant, i.e., if $\ov{a}, \ov{a}'$ lie in the same Galois orbit, then $R_t(\ov{a}) = R_t(\ov{a}')$ for any $\QQ$-tropical polynomial $t$ of the corresponding arity.
\end{example}

\begin{theorem}~\label{theorem_GVF_functional_are_GVF_structures}
    Let $\Field$ be a finitely generated extension of $\QQ$. There is a natural bijection 
    \[ \bigl\{ \textnormal{GVF functionals $\Adiv_{\RR}(\Field) \to \RR$}  \bigr\} 
\longleftrightarrow \bigl\{ \textnormal{GVF structures on $F$}  \bigr\} \]
    \[ l \longmapsto \bigl( R_t^l(\ov{a}):=l(\ov{\cD}_t(\ov{a})) \bigr)_{\textnormal{$\QQ$-tropical polynomials } t} \]
    where $\ov{\cD}_t(\ov{a})=t(\adiv(a_1), \dots, \adiv(a_n))$ if $\ov{a} = (a_1, \dots, a_n)$.
\end{theorem}
The bijection in the theorem is natural as a map between contravariant functors from the category of fields finitely generated over $\QQ$ to sets.
\begin{proof}
    First, we check that the above defined structure $(R_t^l)_t$ on $\Field$ is a GVF structure. Compatibility with permutations and dummy variables is clear. Linearity follows from linearity of $l$. Local-global positivity holds by Remark~\ref{remark_alternative_definition_of_effectivity_of_arithmetic_R_divisors} and Lemma~\ref{lemma_properties_of_the_GVF_pairing}. We have the product formula since $l$ maps principal arithmetic $\RR$-divisors of $C^0$-type to zero.

    On the other hand, a GVF structure on $\Field$ defines a GVF functional
    \[ \Ldiv_{\QQ}(\Field) \to \RR \]
    given by the formula $t(\adiv(f_1), \dots, \adiv(f_m)) \mapsto R_t(f_1, \dots, f_m)$. It is well defined as if $t_1(\adiv(\ov{f}_1)) = t_2(\adiv(\ov{f}_2))$ for some tuples of functions $\ov{f}_1, \ov{f}_2$ in $\Field$, then compatibility with dummy variables and local-global positivity imply that $R_{t_1}(\ov{f}_1) = R_{t_2}(\ov{f}_2)$. By Proposition~\ref{proposition_GVF_functionals_extend_uniquely_to_arithmetic_R_divisors} it extends uniquely to a GVF functional
    \[ \Adiv_{\RR}(\Field) \to \RR. \]
    It is easy to check that these constructions are inverse to each other.
\end{proof}

\begin{remark}~\label{remark_GVF_functional_are_GVF_structures}
    One can generalise the above theorem to an arbitrary characteristic zero field $\Field$.
    Indeed, call a generically smooth, normal arithmetic variety $\cX$ an \textit{$\Field$-submodel}, if it is equipped with a fields embedding $\kappa(\cX) \to \Field$. We equip $\Field$-submodels with a category structure where a morphism $\cX \to \cY$ is a dominant map of arithmetic varieties that respects the field embeddings into $\Field$. Then we put
    \[ \Adiv_{\RR}(\Field) = \varinjlim \Adiv_{\RR}(\cX), \]
    where the limit is taken over the category of $\Field$-submodels with pullback maps on arithmetic $\RR$-divisors of $C^0$-type. A linear functional $l:\Adiv_{\RR}(\Field) \to \RR$ is called a GVF functional, if the natural restriction to $\Adiv_{\RR}(\cX)$ is a GVF functional for all $\Field$-submodels $\cX$. In this case $\Adiv_{\RR}(\Field)$ also has a lattice structure. Moreover, there is a natural bijection 
    \[ \bigl\{ \textnormal{GVF functionals $\Adiv_{\RR}(\Field) \to \RR$}  \bigr\} \longleftrightarrow \bigl\{ \textnormal{GVF structures on $F$}  \bigr\} \]
    \[ l \longmapsto \bigl( R_t^l(\ov{a}):=l(\ov{\cD}_t(\ov{a})) \bigr)_{\textnormal{$\QQ$-tropical polynomials } t} \]
    where $\ov{\cD}_t(\ov{a})=t(\adiv(a_1), \dots, \adiv(a_n))$ if $\ov{a} = (a_1, \dots, a_n)$.
\end{remark}

\begin{definition}~\label{definition_existentially_closed_GVFs}
        Let $\Field$ be a GVF. We call it \textit{existentially closed}, if for any GVF extension $\Field \subset K$ and for any finite tuple $\ov{a}$ in $K$ the following condition holds: if $f_1, \dots, f_m \in \Field(\ov{x})$ and $g_1, \dots, g_n, h \in \Field[\ov{x}]$ satisfy
        \begin{equation}~\label{equation_e_c_conditions}
        \begin{aligned}
          & R_{t_i}(f_i(\ov{a})) = r_i \textnormal{ for } i=1, \dots, m;\\
          & g_1(\ov{a}) = \dots = g_n(\ov{a}) = 0;\\
          & h(\ov{a}) \neq 0;
        \end{aligned}
        \end{equation}
        %\[ R_{t_i}(f_i(\ov{a})) = r_i \textnormal{ for } i=1, \dots, l; \]
        %\[ g_1(\ov{a}) = \dots = g_m(\ov{a}) = 0; \]
        %\[ h_1(\ov{a}) \neq 0, \dots, h_n(\ov{a}) \neq 0; \]
        for some $\QQ$-tropical polynomials $t_1, \dots, t_m$, then for any $\varepsilon > 0$ there is a tuple $\ov{a}'$ in $\Field$ such that
        \begin{equation}~\label{equation_e_c_output}
        \begin{aligned}
        & |R_{t_i}(f_i(\ov{a}')) - r_i| < \varepsilon \textnormal{ for } i=1, \dots, m;\\
        & g_1(\ov{a}') = \dots = g_n(\ov{a}') = 0;\\
        & h(\ov{a}') \neq 0.
        \end{aligned}
        \end{equation}
        %\[ |R_{t_i}(f_i(\ov{a}')) - r_i| < \varepsilon \textnormal{ for } i=1, \dots, l; \]
        %\[ g_1(\ov{a}') = \dots = g_m(\ov{a}') = 0; \]
        %\[ h_1(\ov{a}') \neq 0, \dots, h_n(\ov{a}') \neq 0. \]
\end{definition}

\begin{lemma}~\label{lemma_GVFs_have_finite_extensions_symmetric}
    If $\Field$ is a GVF and $\Field \subset K$ is a finite field extension of $\Field$, then there exists a (not necessarily unique) GVF structure on $K$ making $\Field \subset K$ a GVF embedding.
\end{lemma}
\begin{proof}
    See \cite[Proposition 3.1]{GVF3}.
\end{proof}

In particular, it follows that an existentially closed GVF must be algebraically closed.

\begin{theorem}~\label{theorem_existential_closedness_of_function_field_of_curves}~\cite[Theorem 2.1]{GVF2}
    The globally valued field $\ov{\smallk(t)}[1]$ is existentially closed.
\end{theorem}

\begin{definition}
    Let $\Field$ be a finitely generated extension of $\QQ$ and fix an $\Field$-model $\cX$. Let $\ov{\cD}_1, \dots, \ov{\cD}_m$ be arithmetic $\QQ$-divisors of $C^0$-type on $\cX$, and $t$ be a $\QQ$-tropical polynomial. We say that \textit{the inductive definition of $t(\ov{\cD}_1, \dots, \ov{\cD}_m)$ can be calculated on $\cX$}, in the following cases:
    \begin{itemize}
        \item when $t(x_1, \dots, x_n) = x_i$, we always say so;
        \item when $t = \alpha \cdot t'$, we say so if the inductive definition of $t'(\ov{\cD}_1, \dots, \ov{\cD}_m)$ can be calculated on $\cX$;
        \item when $t = t_1 + t_2$, we say so if the inductive definitions of \[t_1(\ov{\cD}_1, \dots, \ov{\cD}_m), t_2(\ov{\cD}_1, \dots, \ov{\cD}_m)\] can be calculated on $\cX$;
        \item when $t=t_1 \wedge t_2$, we say so if the inductive definitions of \[\ov{\cD} := t_1(\ov{\cD}_1, \dots, \ov{\cD}_m), \ov{\cE} := t_2(\ov{\cD}_1, \dots, \ov{\cD}_m)\] can be calculated on $\cX$, and $\ov{\cD} \wedge \ov{\cE}$ can be calculated on $\cX$.
    \end{itemize}
    % If $\ov{\cD}_1, \dots, \ov{\cD}_m$ are arithmetic $\QQ$-divisors of $C^0$-type on $\cX$ and $t$ is a tropical $\QQ$-polynomial (i.e. an expression using $\QQ$-vector spaces and lattice operations, see Definition~\ref{definition_GVF_Globally_valued_fields}), we say that \textit{the inductive definition of $t(\ov{\cD}_1, \dots, \ov{\cD}_m)$ can be calculated on $\cX$}, if  $f_1, \dots, f_m \in \Field$. Let $\cX$ be an $\Field$-model
\end{definition}

\begin{lemma}~\label{lemma_geometric_translation_of_existencial_closedness}
    Let $\Field=\QQ(\ov{a})$ be a finitely generated extension of $\QQ$ and let $\cX$ be an $\Field$-model. Let $\ov{\cD} = t(\adiv(f_1(\ov{a})), \dots, \adiv(f_n(\ov{a})))$ for some $\QQ$-tropical polynomial $t$ and some rational functions $f_1, \dots, f_n \in \QQ(\ov{y})$ (where $\ov{y}$ are some independent variables). Assume that the inductive definition of $\ov{\cD}$ can be calculated on $\cX$. Let $x \in \cX_{\QQ}$ be a sufficiently generic closed point. Then
    \[ h_{\ov{\cD}}(x) = R_t(f_1(\ov{a})|_x, \dots, f_n(\ov{a})|_x), \]
    where on the right hand side, $f_i(\ov{a})|_x$ is the evaluation of the rational function $f_i(\ov{a})$ at $x$ for every $i=1, \dots, n$, and $R_t$ is calculated with respect to the GVF structure $\kappa(x)[1]$.
    %Let $x \in \cX_{\QQ}$ be a closed point and let $\ov{\cD} = t(\adiv(f_1(\ov{a})), \dots, \adiv(f_n(\ov{a})))$ for some $\QQ$-tropical polynomial $t$ and some rational functions $f_1, \dots, f_n \in \QQ(\ov{y})$ (where $\ov{y}$ are some independent variables). Assume that rational functions $\ov{a}, f_1(\ov{a}), \dots f_n(\ov{a}),g(\ov{a})^{-1} \in \Field$ are regular on $x$ (for some $g \in \QQ[\ov{y}]$ with $g(\ov{a}) \neq 0$) and that $x \not\in \supp(\cD)$. Then $g(\ov{a}|_x) \neq 0$ and:
    %\[ h_{\ov{\cD}}(x) = R_t(\ov{a}|_x), \]
    %where on the right hand side, $\ov{a}|_x$ is the evaluation of rational functions $\ov{a}$ at $x$ and $R_t$ is calculated with respect to the standard GVF structure on the number field $\kappa(x)$.
\end{lemma}
\begin{proof}
    Let $\mu$ be the standard GVF measure on $\Val_{\kappa(x)}$ as described in Example~\ref{example_number_field_GVF} and let $\ov{f} = (f_1, \dots, f_n)$. We calculate
    \[ R_t(\ov{f}(\ov{a})|_x) = \int t(v(\ov{f}(\ov{a})|_x)) dv = \int_{\Val_{\kappa(x)}} t(v(\ov{f}(\ov{a})|_x)) d \mu(v). \]
    By Lemma~\ref{lemma_local_degree_on_minimum_is_minimum} and induction, for sufficiently generic $x$ we get the equality with 
    \[ \int_{\Val_{\kappa(x)}} \adeg_v(\ov{\cD}|\ov{\{x\}}) d \mu(v) = h_{\ov{\cD}}(x), \]
    where the last equality holds by Lemma~\ref{lemma_local_and_global_arithmetic_degree}.
\end{proof}

\begin{theorem}~\label{theorem_QQ_bar_is_existentially_closed}
    $\GVFQ[1]$ is an existentially closed GVF.
\end{theorem}
\begin{proof}
    We start with the following reduction.
    \begin{claim}
    To prove existential closedness of $\GVFQ[1]$ it suffices to prove the following statement: 
    
    (*) For every finitely generated GVF extension $\QQ[1] \subset K$ generated by a tuple $\ov{a}$, if $\ov{a}$ satisfies Equations~(\ref{equation_e_c_conditions}) in Definition~\ref{definition_existentially_closed_GVFs} (with $\Field = \QQ$), then for any $\varepsilon > 0$ there exists a tuple $\ov{a}'$ in $\GVFQ[1]$ satisfying Equations~(\ref{equation_e_c_output}).
    \end{claim}
    \begin{proof}[Proof of the claim]
        Let $\GVFQ[1] \subset K$ be a GVF extension and let $\ov{a}$ be a tuple in $K$. Assume that it satisfies Equations~(\ref{equation_e_c_conditions}) from Definition~\ref{definition_existentially_closed_GVFs} (with $\Field = \GVFQ$). Let $b$ be an element in $\GVFQ$ such that $f_i, g_j, h$ are defined using parameters from $\QQ(b)$ for all $i, j$ as in Definition~\ref{definition_existentially_closed_GVFs}. Fix $\varepsilon > 0$. Use the condition (*) for the finitely generated extension $\QQ[1] \subset \QQ(b, \ov{a})$ to get a tuple $(b', \ov{a}')$ satisfying Equations~(\ref{equation_e_c_output}) (but with $b'$ replacing $b$ in all functions $f_i, g_j, h$). Moreover, by adding some equations over $\QQ$, we can assume that the minimal polynomial of $b'$ is the same as that of $b$. Let $\sigma$ be an automorphism of $\GVFQ$ taking $b'$ to $b$. Since the GVF structure $\GVFQ[1]$ is Galois-invariant (see Example~\ref{example_number_field_GVF}) the tuple $\sigma(\ov{a}')$ satisfies Equations~(\ref{equation_e_c_output}), which finishes the proof of the claim.
    \end{proof}
    We prove (*) for a finitely generated GVF extension $\QQ[1] \subset \Field = \QQ(\ov{a})$. By uniqueness of finite GVF extensions of $\QQ[1]$, we can assume that $\Field$ is not algebraic over $\QQ$. By Theorem~\ref{theorem_GVF_functional_are_GVF_structures} a GVF structure on $\Field$ corresponds to a GVF functional $l:\Adiv_{\RR}(\Field) \to \RR$. Continuous logic predicates in Equations~(\ref{equation_e_c_conditions}) correspond to some finite family $\ov{\cD}_1, \dots, \ov{\cD}_m \in \Ldiv_{\QQ}(\cX)$ on some $\Field$-model $\cX$ with $l(\ov{\cD}_i) = R_{t_i}(\ov{f}_i(\ov{a}))$. By possibly passing to a different $\Field$-model, we can assume that the inductive definitions of these lattice arithmetic $\QQ$-divisors of $C^0$-type can be calculated on $\cX$. We can do this by Remark~\ref{remark_if_wedge_can_be_calculated_on_X_it_can_be_calculated_over_it}. Fix $\varepsilon > 0$. As $\Field$ is of transcendence degree at least $1$ over $\QQ$, we can use Theorem~\ref{theorem_main_arithmetic_approximation_theorem} to find a closed point $x \in \cX_\QQ$ satisfying
    \[ |h_{\ov{\cD}_i}(x) - l(\ov{\cD}_i)| < \varepsilon \textnormal{ for } i=1, \dots, r \]
    and being sufficiently generic so that one can apply Lemma~\ref{lemma_geometric_translation_of_existencial_closedness}.
    By putting $\ov{a}' = \ov{a}|_x$ we get 
    \[ R_{t_i}(\ov{f}_i(\ov{a}')) = R_{t_i}(\ov{f}_i(\ov{a})|_x) = h_{\ov{\cD}_i}(x) \textnormal{ for all } i=1, \dots, r. \]
    This finishes the proof of (*) and hence proves that $\GVFQ[1]$ is an existentially closed GVF.
\end{proof}

\subsection{Essential infimum and GVF functionals}\label{subsection_ess_infimum_and_GVF_functionals}

For this subsection, let $\Field$ be a finitely generated extension of $\QQ$ and let $\cX$ be an $\Field$-model. Assume that the transcendence degree of $\Field$ over $\QQ$ is $d$ (in particular $\cX$ is of dimension $d+1$).

\begin{theorem}~\label{theorem_Riesz_extension_theorem}(M.Riesz extension theorem)
    Let $\KK$ be $\RR$ or $\QQ$ and work in the category of $\KK$ vector spaces. Let $E$ be a vector space, $W \subset E$ be a subspace, and let $K \subset E$ be a convex cone. Assume that a linear functional $l_0:W \to \RR$ non-negative on $W \cap K$ is given. If the condition
    \[ E \subset K+W \]
    is satisfied, then there exists an extension $l:E \to \RR$ of $\phi$, non-negative on $K$. Moreover, $l$ is unique, if for each $e \in E$ and $\varepsilon > 0$, there are $w, w' \in W$ with $w-e, e-w' \in K$ and $l_0(w-w') < \varepsilon$.
\end{theorem}
\begin{proof}
    This is classical. For a proof, see \cite[Lemma 5.8]{GVF3}.
\end{proof}

\begin{corollary}~\label{corolarry_extending_GVF_functionals_using_Riesz_extension}
    Let $V$ be a real vector subspace of $\Adiv_{\RR}(\Field)$ containing a big arithmetic $\RR$-divisor $\ov{\cD}$ of $C^0$-type. Assume that we are given a GVF functional
    \[ l_0:V \to \RR \]
    Then there exists an extension $l_0 \subset l : \Adiv_{\RR}(\Field) \to \RR$ to a GVF functional $l$.
\end{corollary}
\begin{proof}
    Let $W$ be the image of $V$ under the projection $\pi:\Adiv_{\RR}(\Field) \to \aN^1(\Field)$. Note that $l_0$ factors through $W$ as it is a GVF functional. Use Theorem~\ref{theorem_Riesz_extension_theorem} (the existence part) with $E = \aN^1(\Field)$ and $K$ being the image of the cone of effective arithmetic $\RR$-divisors of $C^0$-type under $\pi$. Take a class of an arithmetic $\RR$-divisor of $C^0$-type $\ov{\cE}$ on an $\Field$-model $\cX$. Note that by Theorem~\ref{theorem_differentiability_of_arithmetic_volume}, for a natural number $n$ we have (possibly by passing to a different $\Field$-model)
    \[ \avol(n \ov{\cD} + \ov{\cE}) = n^{d+1} \avol(\ov{\cD} + \frac{1}{n} \ov{\cE}), \]
    and $\lim_n \avol(\ov{\cD} + \frac{1}{n} \ov{\cE}) = \avol(\ov{\cD}) > 0$, so for $n$ big enough $\ov{\cF} := n \ov{\cD} + \ov{\cE}$ is big. In particular by Lemma~\ref{lemma_effective_are_pseudo_effective} its $\pi$ image is in $K$. Note that
    \[ \ov{\cE} = \ov{\cF} + (-n \ov{\cD}), \]
    and by the fact that $\pi(\ov{\cE}) \in E$ was arbitrary, we get an extension function $\aN^1(\Field) \to \RR$ non-negative on classes of effective arithmetic $\RR$-divisors of $C^0$-type. By composing it with $\pi$ we are done.
\end{proof}

\begin{corollary}~\label{corollary_main_arithmetic_approximation_corollary}
    Let $\ov{\cD}_0, \dots, \ov{\cD}_n$ be arithmetic $\RR$-divisors of $C^0$-type on $\cX$ with $\ov{\cD}_0 = (\cdiv(2),0)$. Denote by $V$ the real vector space generated by $\ov{\cD}_0, \dots, \ov{\cD}_n$ and assume that one of $\ov{\cD}_0, \dots, \ov{\cD}_n$ is big. Fix $\varepsilon > 0$, a closed proper subscheme $Z \subset X=\cX \otimes \QQ$, and a normalised GVF functional
    \[ l:V \to \RR, \]
    Then there exists a closed point $x \in X \setminus Z$ such that for all $i=0, \dots, n$
    \[ |h_{\ov{\cD}_i}(x) - l(\ov{\cD}_i)| < \varepsilon. \]
\end{corollary}

\begin{construction}~\label{construction_GVF_functional_from_a_nef_R_arithmetoc_divisor_on_a_model}
    Let $\ov{\cB}$ be a big arithmetic $\RR$-divisor of $C^0$-type on an $\Field$-model $\cX'$. We define a linear functional $l_{\ov{\cB}}:\Adiv_{\RR}(\Field) \to \RR$ by the following formula.
    \[ l_{\ov{\cB}}:\Adiv_{\RR}(\Field) \to \RR \]
    \[ \ov{\cE} \mapsto \langle (\phi^*\ov{\cB})^d \rangle \cdot \ov{\cE}, \]
    where $\ov{\cE}$ is an arithmetic $\RR$-divisor of $C^0$-type on $\phi:\cX'' \to \cX'$. By Lemma~\ref{lemma_intersection_doesnt_depend_on_the_model} $l_{\ov{\cB}}$ is well defined. Moreover, by Theorem~\ref{theorem_differentiability_of_arithmetic_volume} we have
    \[ \langle (\phi^*\ov{\cB})^d \rangle \cdot \ov{\cE} = \frac{1}{d+1} D_{\ov{\cB}} \avol(\ov{\cE}) = \lim_{t \to 0^+} \frac{\avol(\phi^*\ov{\cB} + t\ov{\cE}) - \avol(\phi^*\ov{\cB})}{(d+1)t}. \]
    But for effective $\ov{\cE}$ and $t>0$ we have $\avol(\phi^*\ov{\cB} + t\ov{\cE}) - \avol(\phi^*\ov{\cB}) \geq 0$, see the proof of Lemma~\ref{lemma_effective_are_pseudo_effective}. Thus $l_{\ov{\cB}}(\ov{\cE}) \geq 0$ and so $l_{\ov{\cB}}$ is in fact a GVF functional. Note that if $\ov{\cB}$ is nef, then $l_{\ov{\cB}}(\ov{\cE}) = (\phi^*\ov{\cB})^d \cdot \ov{\cE}$ by Remark~\ref{remark_equiv_of_def_of_pos_intersection_in_Ikoma_and_pos_int_for_nef}.
\end{construction}

\begin{lemma}~\label{lemma_arithmetic_divisor_pseudo_effective_iff_all_GVF_functionals_non_negative}
    Let $\ov{\cD}$ be an arithmetic $\RR$-divisor of $C^0$-type on $\cX$. Then $\ov{\cD}$ is pseudo-effective if and only if the following holds.
    \[ \textnormal{For all GVF functionals }l:\Adiv_{\RR}(\Field) \to \RR \textnormal{ we have } l(\ov{\cD}) \geq 0. \]
    Moreover, $\ov{\cD}$ is principal if and only if it satisfies the following.
    \[ \textnormal{For all GVF functionals }l:\Adiv_{\RR}(\Field) \to \RR \textnormal{ we have } l(\ov{\cD}) = 0. \]
\end{lemma}
\begin{proof}
    We start with the first characterisation. Assume that $\ov{\cD}$ is pseudo-effective and let $\ov{\cB}$ be a big arithmetic $\RR$-divisor of $C^0$-type. Then for every $\delta>0$ we have $\avol(\ov{\cD} + \delta \ov{\cB}) > 0$, so if $l$ is a GVF functional on $\Adiv_{\RR}(\Field)$ we have
    \[ l(\ov{\cD}) = \lim_{\delta \to 0^+} l(\ov{\cD} + \delta \ov{\cB}) \geq 0, \]
    where the last inequality follows from Lemma~\ref{lemma_effective_are_pseudo_effective}. On the other hand, if $\ov{\cD}$ is not pseudo-effective, by Theorem~\ref{theorem_Ikoma_characterisation_of_psef_arithmetic_R_divisors} there is an $\Field$-model $\pi:\cX' \to \cX$ with a nef arithmetic $\RR$-divisor $\ov{\cH}$ such that $\ov{\cH}^d \cdot \pi^* \ov{\cD} < 0$. By Construction~\ref{construction_GVF_functional_from_a_nef_R_arithmetoc_divisor_on_a_model}, $l = l_{\ov{\cH}}$ is a GVF functional, and it satisfies $l(\ov{\cD}) < 0$ which finishes the proof.

    To prove the other characterisation, it's enough to prove that if all GVF functionals vanish on $\ov{\cD}$, then $\ov{\cD}$ is principal. Note that by the first part of the lemma we get that $\ov{\cD}$ is pseudo-effective. Now the result follows from Theorem~\ref{theorem_Ikoma_characterisation_of_psef_arithmetic_R_divisors} and the proof of the first characterisation.
\end{proof}

\begin{remark}~\label{remark_it_is_enoguh_to_check_on_normalised}
    Fix the notation from Lemma~\ref{lemma_arithmetic_divisor_pseudo_effective_iff_all_GVF_functionals_non_negative} and let $\ov{\cD}_0 = (\cdiv(2), 0)$. Note that in both characterisations we could in fact check the corresponding condition only for normalised GVF functionals. Indeed, assume that the first condition holds for normalised GVF functionals on $\Adiv_{\RR}(\Field)$. Take any GVF functional $l:\Adiv_{\RR}(\Field) \to \RR$. If $l(\ov{\cD}_0) \neq 0$, we can rescale it (using a positive scalar as $\ov{\cD}_0$ is effective) to be normalised, so $l(\ov{\cD}) \geq 0$. If $l(\ov{\cD}_0) = 0$, we can take any normalised GVF functional $l'$ on $\Adiv_{\RR}(\Field)$ and if we had $l(\ov{\cD}) < 0$, then for $n$ big enough $(l' + n l)(\ov{\cD}) < 0$, which gives a contradiction.
\end{remark}

By unravelling the proof of the first part of the below theorem, one can see that it is essentially the same as the proof of \cite[Theorem 4.6]{Arithmetic_Demailly_Qu_Yin} (but using Theorem~\ref{theorem_arithmetic_Bertini_by_Wilms} instead of \cite[Theorem 3.9]{Arithmetic_Demailly_Qu_Yin}). We decided to include this theorem here to underline the new interpretation of $\zeta$ using GVF functionals.

\begin{theorem}~\cite[Theorem 4.6]{Arithmetic_Demailly_Qu_Yin}~\label{theorem_essential_infimum_as_infimum_of_GVF_functionals}
    Let $\ov{\cD}$ be an arithmetic $\RR$-divisor of $C^0$-type on $\cX$. Then
    \[ \zeta(\ov{\cD}) \leq \inf \{ l(\ov{\cD}) : \textnormal{$l$ is a normalised GVF functional on } \Adiv_{\RR}(\Field) \}. \]
    Moreover, if $\cD_\QQ$ is big, this is an equality.
\end{theorem}
\begin{proof}
    For the first part, assume that there is a normalised GVF functional $l:\Adiv_{\RR}(\Field) \to \RR$ such that $l(\ov{\cD}) = r$. Fix a closed proper subscheme $\cY \subset \cX$ and $\varepsilon>0$. By Theorem~\ref{theorem_main_arithmetic_approximation_theorem} there is a closed point $x \in X = \cX_\QQ$ such that $x \not\in \cY$ and $|h_{\ov{\cD}}(x)-l(\ov{\cD})|<\varepsilon$. We can use the theorem because $l$ is normalised. Since $\cY$ was arbitrary, we get
    \[ \zeta(\ov{\cD}) = \sup_{\cY \subset \cX} \inf_{x \in \cX \setminus \cY(\ov{\QQ})} h_{\ov{\cD}}(x) \leq r+\varepsilon. \]
    As $l$ and $\varepsilon$ were arbitrary, we get the desired inequality.

    For the other part, by Lemma~\ref{lemma_F_Ballay_lemma_3_16} (using bigness of $\cD_\QQ$) we have
    \[ \zeta(\ov{\cD}) \geq \sup \{ r : \ov{\cD} - \frac{r}{\log 2} \cdot \ov{\cD}_0 \textnormal{ is pseudo-effective} \}, \]
    for $\ov{\cD}_0 = (\cdiv(2),0)$. By Lemma~\ref{lemma_arithmetic_divisor_pseudo_effective_iff_all_GVF_functionals_non_negative} we know that for any fixed $r \in \RR$
    \[ \ov{\cD} - \frac{r}{\log 2} \cdot \ov{\cD}_0 \textnormal{ is pseudo-effective} \]
    if and only if for all GVF functionals $l:\Adiv_{\RR}(\Field) \to \RR$ we have $l(\ov{\cD} - \frac{r}{\log 2} \cdot \ov{\cD}_0) \geq 0$. Note that for a normalised GVF functional $l:\Adiv_{\RR}(\Field) \to \RR$ we have
    \[ l(\ov{\cD} - \frac{r}{\log 2} \cdot \ov{\cD}_0) = l(\ov{\cD}) - r, \]
    so this expression is $\geq 0$ if and only if $l(\ov{\cD}) \geq r$. By Remark~\ref{remark_it_is_enoguh_to_check_on_normalised} we get that
    \[ \zeta(\ov{\cD}) \geq \sup \{ r : (\forall l)(l(\ov{\cD}) \geq r) \},  \]
    where we quantify over all normalised GVF functionals $l$ on $\Adiv_{\RR}(\Field)$. This gives the desired equality in the case where $\cD_{\QQ}$ is big.
\end{proof}

As a corollary one gets \cite[Theorem 1.4]{Arithmetic_Demailly_Qu_Yin} which is equivalent to \cite[Corollary 4.2]{F_Ballay_Succesive_minima} without the semi-positivity assumption.

\begin{corollary}~\cite[Theorem 1.4]{Arithmetic_Demailly_Qu_Yin}~\label{corollary_Ballays_result_without_semi_pos_ass_zeta_formula}
    Let $\ov{\cD}$ be an arithmetic $\RR$-divisor of $C^0$-type on $\cX$ with $\cD_{\QQ}$ big and let $\ov{\cD}_0 = (\cdiv(2),0)$. Then
    \[ \zeta(\ov{\cD}) = \sup \{ r : \ov{\cD} - \frac{r}{\log 2} \cdot \ov{\cD}_0 \textnormal{ is pseudo-effective} \}. \]
\end{corollary}
\begin{proof}
    Note that by Lemma~\ref{lemma_F_Ballay_lemma_3_16}
    \[ \zeta(\ov{\cD}) \geq r \iff \zeta(\ov{\cD} - \frac{r}{\log 2} \ov{\cD}_0) \geq 0, \]
    which by Theorem~\ref{theorem_essential_infimum_as_infimum_of_GVF_functionals} is equivalent to the inequality $\inf_l (l(\ov{\cD}) - \frac{r}{\log 2}) \geq 0$ where the infimum is taken over all normalised GVF functionals on $\Adiv_{\RR}(\Field)$. By Lemma~\ref{lemma_arithmetic_divisor_pseudo_effective_iff_all_GVF_functionals_non_negative} this is equivalent to $\ov{\cD} - \frac{r}{\log 2} \ov{\cD}_0$ being pseudo-effective, which finishes the proof.
\end{proof}

\begin{construction}
    Let $\ov{\cD} \in \Adiv_{\RR}(\Field)$ be big. Note that then $\cD_{\QQ}$ is big by Lemma~\ref{lemma_effective_are_pseudo_effective}. Let
    \[ \varphi(\ov{\cE}) = \frac{\avol(\ov{\cE})}{(d+1) \vol(\cE_{\QQ})} \]
    for $\ov{\cE} \in \Adiv_{\RR}(\Field)$ with big $\cE_{\QQ}$. Note that it is well defined since both $\avol$ and $\vol$ are invariant under birational pullbacks (see Lemma~\ref{lemma_pseudo_effective_cone_is_invariant_under_birational_pullbacks} and \cite{Lazarsfeld_Positivity} for $\vol$). By differentiability of $\avol$ from Theorem~\ref{theorem_differentiability_of_arithmetic_volume} and differentiability of $\vol$ proved in \cite[Theorem A]{Boucksom_DIFFERENTIABILITY_OF_VOLUMES} we get a linear functional
    \[ D_{\ov{\cD}} \varphi : \Adiv_{\RR}(\Field) \to \RR \]
    which is the Gateaux derivative of $\varphi$ calculated on an appropriate $\Field$-model. This linear functional is normalised and zero on principal arithmetic $\RR$-divisors of $C^0$-type, but is not necessarily a GVF functional.
\end{construction}

\begin{lemma}~\label{lemma_compactness_of_the_space_of_GVF_functionals}
    Let $\ov{\cD}$ be a big arithmetic $\RR$-divisor of $C^0$-type on $\cX$. Then there exists a normalised GVF functional $l:\Adiv_{\RR}(\Field) \to \RR$ such that $\zeta(\ov{\cD}) = l(\ov{\cD})$.
\end{lemma}
\begin{proof}
    We give two proofs. First, we can find a sequence of normalised GVF functionals $l_n$ on $\Adiv_{\RR}(\Field)$ with $\lim_n l_n(\ov{\cD}) = \zeta(\ov{\cD})$ by Theorem~\ref{theorem_essential_infimum_as_infimum_of_GVF_functionals}. Pick a bound $C>0$ such that $l_n(\ov{\cD}) < C$ for all natural $n$. Let $\ov{\cE} \in \Adiv_{\RR}(\Field)$. As in the proof of Corollary~\ref{corolarry_extending_GVF_functionals_using_Riesz_extension}, take a natural number $m$ such that $m \ov{\cD} - \ov{\cE}, m \ov{\cD} + \ov{\cE}$ are big. Then for all GVF functionals $l$ on $\Adiv_{\RR}(\Field)$ with $l(\ov{\cD}) \leq C$ we have
    \[ -Cm \leq l(\ov{\cE}) \leq Cm. \]
    Denote by $m = m_{\ov{\cE}}$ a number suitable for $\ov{\cE}$. We get that the set of normalised GVF functionals $l$ on $\Adiv_{\RR}(\Field)$ with $l(\ov{\cD}) \leq C$ embeds as a closed subset into
    \[ \prod_{\ov{\cE} \in \Adiv_{\RR}(\Field)} [-C m_{\ov{\cE}}, C m_{\ov{\cE}} ]. \]
    This product is compact, so we can find a subnet of $(l_n)_{n \in \NN}$ converging to a desired normalised GVF functional.

    Another way to construct such a normalised GVF functional is to consider the vector space $V$ generated by $\ov{\cD}_0 = (\cdiv(2),0)$ and $\ov{\cD}$ in $\Adiv_{\RR}(\Field)$. Next, define $l$ on $V$ by putting $l(\ov{\cD}_0) := \log(2), l(\ov{\cD}) := \zeta(\ov{\cD})$. By Corollary~\ref{corollary_Ballays_result_without_semi_pos_ass_zeta_formula} we get that $l$ is a GVF functional, and by Corollary~\ref{corolarry_extending_GVF_functionals_using_Riesz_extension} it has an extension to $\Adiv_{\RR}(\Field)$, which finishes the second proof.
\end{proof}

\begin{theorem}~\label{theorem_characterisation_of_Zhang_equality_divisors}
    Let $\ov{\cD}$ be a big arithmetic $\RR$-divisor of $C^0$-type on $\cX$. Consider the statements:
    \begin{enumerate}[label=(\arabic*)]
        \item $\varphi(\ov{\cD}) = \zeta(\ov{\cD})$;
        \item $D_{\ov{\cD}} \varphi$ is a GVF functional;
        \item the infimum
        \[ \zeta(\ov{\cD}) = \inf \{ l(\ov{\cD}) : \textnormal{$l$ is a normalised GVF functional on } \Adiv_{\RR}(\Field) \} \]
        is achieved at the unique normalised GVF functional;
        \item $\zeta$ is Gateaux differentiable at $\ov{\cD}$ in every direction.
        % Not so sure about the last one
    \end{enumerate}
    Then $(1) \iff (2) \implies (3) \iff (4)$.
\end{theorem}
\begin{proof}
    Let $\ov{\cE} \in \Adiv_{\RR}(\Field)$. Note that for small real $t$ the arithmetic $\RR$-divisor of $C^0$-type $\ov{\cD}+t \ov{\cE}$ is big. Let $l_0$ be a normalised GVF functional on $\Adiv_{\RR}(\Field)$ with $l_0(\ov{\cD}) = \zeta(\ov{\cD})$. Such exists by Lemma~\ref{lemma_compactness_of_the_space_of_GVF_functionals}. For real $t>0$ small enough we then have
    \begin{equation}~\label{eq_zeta}
        \frac{\zeta(\ov{\cD}+t \ov{\cE})-\zeta(\ov{\cD})}{t} \leq \frac{l_0(\ov{\cD}+t \ov{\cE})-l_0(\ov{\cD})}{t}  = l_0(\ov{\cE}).
    \end{equation}
    %\[ \frac{\zeta(\ov{\cD}+t \ov{\cE})-\zeta(\ov{\cD})}{t} \leq \frac{\l_0(\ov{\cD}+t \ov{\cE})-l_0(\ov{\cD})}{t}  = l_0(\ov{\cE}). \]
    To get the above inequality we use Theorem~\ref{theorem_essential_infimum_as_infimum_of_GVF_functionals}. Now we pass to proving equivalence of the conditions in the theorem.
    
    $(1) \Rightarrow (2):$ Assuming $\varphi(\ov{\cD}) = \zeta(\ov{\cD})$, by Theorem~\ref{theorem_Zhang_inequality} we get
    \[ \frac{\varphi(\ov{\cD}+t \ov{\cE})-\varphi(\ov{\cD})}{t} \leq \frac{\zeta(\ov{\cD}+t \ov{\cE})-\zeta(\ov{\cD})}{t} \leq l_0(\ov{\cE}). \]
    Taking $t \to 0^+$ we get that $D_{\ov{\cD}} \varphi \leq l_0$, but since they are linear this implies that $D_{\ov{\cD}} \varphi = l_0$ is a GVF functional.
    
    $(2) \Rightarrow (1):$ From Theorem~\ref{theorem_essential_infimum_as_infimum_of_GVF_functionals} we then get that $\zeta(\ov{\cD}) \leq D_{\ov{\cD}} \varphi (\ov{\cD}) = \varphi(\ov{\cD})$ which together with Theorem~\ref{theorem_Zhang_inequality} gives the equality.

    $(1) \Rightarrow (4):$ From the proof of $(1) \Rightarrow (2)$ we get that 
    \[ l_0(\ov{\cE}) = D_{\ov{\cD}} \varphi(\ov{\cE}) \leq \lim_{t \to 0^+} \frac{\zeta(\ov{\cD}+t \ov{\cE})-\zeta(\ov{\cD})}{t} \leq l_0(\ov{\cE}).  \]
    For the negative $t$ replace $\ov{\cE}$ by $-\ov{\cE}$.

    $(4) \Rightarrow (3):$ By taking the limit $t \to 0$ in Equation~(\ref{eq_zeta}) we get
    \[ D_{\ov{\cD}} \zeta(\ov{\cE}) \leq l_0(\ov{\cE}). \]
    Since both maps are linear, we get $D_{\ov{\cD}} \zeta = l_0$, hence the uniqueness of $l_0$ follows.

    $(3) \Rightarrow (4):$ Let $V$ be the real vector subspace generated by $\ov{\cD}_0 = (\cdiv(2),0), \ov{\cD}, \ov{\cE}$ in $\Adiv_{\RR}(\Field)$. If the classes of $\ov{\cD}_0, \ov{\cD}, \ov{\cE}$ are linearly dependent in $\aN^1(\Field)$, then differentiability of $\zeta$ at $\ov{\cD}$ in direction $\ov{\cE}$ follows from Lemma~\ref{lemma_properties_of_zeta_function} and Lemma~\ref{lemma_F_Ballay_lemma_3_16}. Assume that those classes are linearly independent. Let $V^+$ be the (euclidean) closure of the big cone in $V$ and pick a hyperplane $P$ in $V$ containing $\ov{\cD}_0, \ov{\cD}$, such that $Q := P \cap V^+$ is a compact convex set. This is possible as $V^+$ does not contain a line. By assumption and Theorem~\ref{theorem_Riesz_extension_theorem} we know that $l_0$ restricted to $P$ is the unique affine function $l$ on $P$ having the following properties:
    \begin{itemize}
        \item $l(\ov{\cD}_0) = \log(2)$,
        \item $l(\ov{\cD}) = \zeta(\ov{\cD})$,
        \item $l$ is non-negative on $Q$.
    \end{itemize}
    Consider the function $f(t) = \zeta(\ov{\cD} + t \ov{\cE})$ on a neighbourhood of $0$. It is concave by Lemma~\ref{lemma_properties_of_zeta_function}. We now show that if $f$ is not differentiable at zero, then we can find a different affine function $l$ having the listed properties. By moving the hypersurface $P$ and rescaling $\ov{\cD}_0, \ov{\cD}, \ov{\cE}$ we can without loss of generality assume that $\ov{\cD}_0, \ov{\cD}, \ov{\cE} \in P$. In other words
    \[ P = \{ a\ov{\cD}_0+b\ov{\cD}+c\ov{\cE} : a+b+c=1 \}. \]
    Fix coordinates on $P$ so that $\ov{\cD}_0$ is the zero and the basis vectors are $e_1 = \ov{\cD} - \ov{\cD}_0, e_2 = \ov{\cE} - \ov{\cD}_0$. Consider the function
    \[ g(\varepsilon) := \sup \{ r : \ov{\cD} + \varepsilon e_2 + r e_1 \textnormal{ is pseudo-effective} \}. \]
    There is a positive $t$ such that $g:(-t,t) \to \RR$ is a concave function, as $Q$ is a compact convex set. Note that
    \[ \ov{\cD} + \varepsilon e_2 + r e_1 = (1+r)\ov{\cD} + \varepsilon \ov{\cE} -(r+\varepsilon)\ov{\cD}_0. \]
    By Corollary~\ref{corollary_Ballays_result_without_semi_pos_ass_zeta_formula} we know that $\ov{\cD} + \varepsilon e_2 + r e_1$ is pseudo-effective if and only if $\zeta(\ov{\cD} + \varepsilon e_2 + r e_1) \geq 0$, which by the above calculation and Lemma~\ref{lemma_F_Ballay_lemma_3_16} is equivalent to
    \[ \zeta(\ov{\cD} + \varepsilon e_2 + r e_1) = \zeta((1+r)\ov{\cD} + \varepsilon \ov{\cE}) - (r+\varepsilon) \log(2) \]
    \[ = (1+r) \zeta(\ov{\cD} + \frac{\varepsilon}{1+r} \ov{\cE}) - (r+\varepsilon) \log(2) = (1+r) f(\frac{\varepsilon}{1+r}) - (r+\varepsilon) \log(2) \geq 0. \]
    The supremum of $r$ where this quantity is non-negative is achieved when we have the equality with zero, so we get the functional equation
    \[ (1+g(\varepsilon)) f(\frac{\varepsilon}{1+g(\varepsilon)}) - (g(\varepsilon)+\varepsilon) \log(2) = 0. \]
    Thus
    \[ f(\frac{\varepsilon}{1+g(\varepsilon)}) = \log(2) \frac{g(\varepsilon)+\varepsilon}{1+g(\varepsilon)}. \]
    Note that $g(0)$ is strictly positive as $\avol(\ov{\cD}) > 0$. If $g$ was differentiable at $0$, then $f$ also would be, as
    \[ \frac{f(\frac{\varepsilon}{1+g(\varepsilon)}) - f(0)}{\frac{\varepsilon}{1+g(\varepsilon)}} = \log(2) \frac{\frac{g(\varepsilon)+\varepsilon}{1+g(\varepsilon)} - \frac{g(0)}{1+g(0)}}{\frac{\varepsilon}{1+g(\varepsilon)}}. \]
    Thus $g$ is not differentiable at $0$. This means that there exist two different lines $X_1, X_2$ on $P$ such that $X_i \cap Q = \{\ov{\cD} + g(0)e_1\}$ for $i=1,2$. Consider affine functions $l_1, l_2$ on $P$ determined by the following conditions
    \begin{itemize}
        \item $l_i$ vanishes on $X_i$,
        \item $l_i(\ov{\cD}_0) = \log(2)$,
        \item $l_i(\ov{\cD}) = \zeta(\ov{\cD})$,
    \end{itemize}
    for $i=1, 2$. Note that such functions exists by Corollary~\ref{corollary_Ballays_result_without_semi_pos_ass_zeta_formula}. Moreover, they are non-negative on $Q$, since they are affine, their zero sets are supporting lines for $Q$, and they are positive on a point $\ov{\cD}_0$ in $Q$. This gives a non-uniqueness of $l_0$, which finishes the proof.
\end{proof}

\begin{question}
    What is the set of all arithmetic big $\RR$-divisors of $C^0$-type $\ov{\cD}$ with $D_{\ov{\cD}} \zeta$ Gateaux differentiable? How to find a counterexample for the implication $(3) \Rightarrow (2)$ in Theorem~\ref{theorem_characterisation_of_Zhang_equality_divisors}?
\end{question}

Before the next lemma, note that if $x \in \cX_{\QQ}$ is a sufficiently general closed point, then for $\ov{\cE} \in \Adiv_{\RR}(\Field)$ one can make sense out of $h_{\ov{\cE}}(x)$. Indeed, if $\ov{\cE} \in \Adiv_{\RR}(\cX')$ for some $\Field$-model $\phi:\cX' \to \cX$, then if $x$ is in the open subset where $\phi$ is an isomorphism, we can define $h_{\ov{\cE}}(x) := h_{\ov{\cE}}(\phi^{-1}(x))$.

The below lemma follows the lines of \cite[Lemma (8.2)]{Chambert_Loir_survey} or \cite[Section 5]{Chen_differentiability_of_arithmetic_volume} and proves that differentiability of $\zeta$ is enough for equdistribution (over all places).

\begin{lemma}~\label{lemma_equdistributuin_with_zeta_differentiable}
    Let $\ov{\cD}$ be a big arithmetic $\RR$-divisor of $C^0$-type on $\cX$. Assume that $(x_n)_{n \in \NN}$ is a sequence of closed points in the generic fiber of $\cX$ which converges to the generic point. If $\zeta$ is Gateaux differentiable at $\ov{\cD}$ in directions $\ov{\cE}, - \ov{\cE} \in \Adiv_{\RR}(\Field)$ and
    \[ \lim_n h_{\ov{\cD}}(x_n) = \zeta(\ov{\cD}), \]
    then
    \[ \lim_n h_{\ov{\cE}}(x_n) = D_{\ov{\cD}} \zeta(\ov{\cE}). \]
\end{lemma}
\begin{proof}
    Let $t>0$ be a real number. We calculate
    \[ \liminf_{n \to \infty} h_{\ov{\cE}}(x_n) = \liminf_{n \to \infty} \frac{h_{\ov{\cD} + t\ov{\cE}}(x_n) - h_{\ov{\cD}}(x_n)}{t}. \]
    % = \frac{\limsup_{n \to \infty} h_{\ov{\cD} + t\ov{\cE}}(x_n) - \zeta(\ov{\cD})}{t} 
   Note that $\liminf_{n \to \infty} h_{\ov{\cD} + t\ov{\cE}}(x_n) \geq \zeta(\ov{\cD} + t\ov{\cE})$ and $\liminf_{n \to \infty} h_{\ov{\cD}}(x_n) = \zeta(\ov{\cD})$. Thus we get
    \[ \liminf_{n \to \infty} \frac{h_{\ov{\cD} + t\ov{\cE}}(x_n) - h_{\ov{\cD}}(x_n)}{t} \geq \frac{\zeta(\ov{\cD} + t\ov{\cE}) - \zeta(\ov{\cD})}{t}. \]
    Taking the limit $t \to 0^+$ we have
    \[ \liminf_{n \to \infty} h_{\ov{\cE}}(x_n) \geq D_{\ov{\cD}} \zeta(\ov{\cE}). \]
    Replacing $\ov{\cE}$ with $- \ov{\cE}$ we get
    \[ \limsup_{n \to \infty} h_{\ov{\cE}}(x_n) \leq D_{\ov{\cD}} \zeta(\ov{\cE}). \]
    Together those inequalities give $\lim_n h_{\ov{\cE}}(x_n) = D_{\ov{\cD}} \zeta(\ov{\cE})$ which finishes the proof.
\end{proof}

\printbibliography

\end{document}